\documentclass[11pt]{article}

\usepackage{color}
\usepackage{amssymb,amsmath,amsthm,setspace}
\usepackage[top=2cm, bottom=2cm, left=2cm, right=2cm]{geometry}

\theoremstyle{plain}

\newtheorem{thm}{THEOREM}[section]
\newtheorem{lm}[thm]{LEMMA}

\newtheorem{prop}[thm]{PROPOSITION}
\theoremstyle{definition}

\theoremstyle{definition}

\newcommand{\C}{{\mathord{\mathbb C}}}

\newcommand{\supp}{{\mathop{\rm supp\ }}}

\numberwithin{equation}{section}
\begin{document}

\title{An Extension of the Bianchi-Egnell Stability Estimate to Bakry, Gentil, and Ledoux's Generalization of the Sobolev Inequality to Continuous Dimensions}
\date{November 5, 2015}
\author{Francis Seuffert\thanks{Work partially supported by U.S. National Science Foundation grants DMS 1201354 and DMS 1501007} \\ francis.seuffert@gmail.com \\ Math Department, Rutgers University}

\maketitle

\section{Introduction}

\begin{abstract} 
This paper extends a stability estimate of the Sobolev Inequality established by Bianchi and Egnell in [3]. Bianchi and Egnell's Stability Estimate answers the question raised by H. Brezis and E. H. Lieb in [5]: ``Is there a natural way to bound $\| \nabla \varphi \|_2^2 - C_N^2 \| \varphi \|_\frac{2N}{N-2}^2$ from below in terms of the `distance' of $\varphi$ from the manifold of optimizers in the Sobolev Inequality?'' Establishing stability estimates - also known as quantitative versions of sharp inequalities - of other forms of the Sobolev Inequality, as well as other inequalities, is an active topic. See [9], [11], and [12], for stability estimates involving Sobolev inequalities and [6], [11], and [14] for stability estimates on other inequalities. In this paper, we extend Bianchi and Egnell's Stability Estimate to a Sobolev Inequality for ``continuous dimensions.'' Bakry, Gentil, and Ledoux have recently proved a sharp extension of the Sobolev Inequality for functions on $\mathbb{R}_+ \times \mathbb{R}^n$, which can be considered as an extension to ``continuous dimensions.'' V. H. Nguyen determined all cases of equality. The present paper extends the Bianchi-Egnell stability analysis for the Sobolev Inequality to this ``continuous dimensional'' generalization.
\end{abstract}

\subsection{The Sharp Sobolev Inequality}

Let $\dot{H}^1 (\mathbb{R}^N)$ be the the completion of the space of smooth real-valued functions with compact support under the norm
\[
\|\varphi \|_{\dot{H}^1} := 
\| \nabla \varphi \|_2 = \left(\int_{\mathbb{R}^N} | \nabla \varphi |^2
\mathrm{d}x\right)^{1/2} \,,
\]
where for $1 \leq p < \infty$, (and $2$ in particular) $\| \varphi \|_p$ denotes the $L^p$ norm of $\varphi$,
$\displaystyle{
\| \varphi \|_p = \left(\int_{\mathbb{R}^N} | \varphi |^p \mathrm{d}x\right)^{1/p}}
$.
Define \begin{equation}
2^* := \frac{2N}{N-2}\ .
\end{equation}
The Sobolev Inequality provides a lower bound for $\| \varphi \|_{\dot{H}^1}$ in terms
of $\| \varphi \|_{2^*}$.
\begin{thm}[Sharp Sobolev Inequality]\label{sob}
Let $N \geq 3$ be an integer. Then, for all $\varphi \in \dot{H}^1 (\mathbb{R}^N) \setminus \{ 0 \}$,
\begin{equation}\label{SobIneqCl}
\frac{\|\varphi\|_{2^*}}{\|\varphi\|_{\dot{H}^1}} \leq \frac{\| F_{1,0} \|_{2^*}}{\| F_{1,0} \|_{\dot{H}^1}}
=: C_N \,,
\end{equation}
where
\begin{equation}\label{ClSobIneqExt}
F_{t, x_0} (x) := k_0 \left( \frac{t}{1 + t^2 | x - x_0 |^2} \right)^{\frac{N-2}{2}} \,,
\end{equation}
for $t > 0$, $x_0 \in \mathbb{R}^N$, and $k_0 > 0$ a constant such that $\| F_{1,0} \|_{\dot{H}^1} = 1$. There is equality if
and only if $\varphi = z F_{t, x_0}$ for some $t > 0$, $x_0 \in \mathbb{R}^N$, and some $z \in \mathbb{R} \setminus \{ 0 \}$. 
\end{thm}
\noindent Theorem \ref{sob} in this sharp form, including specification of the cases of equality, was proved by Talenti in [22]. The result is also true for complex-valued functions, the only difference being that equality holds for $\varphi = z F_{t,0}$ with $z \in \mathbb{C} \setminus  \{ 0 \}$, but for the moment we will restrict our attention to real-valued functions. For another reference on the Sobolev Inequality, see [13].

The fact that the conformal group of $\mathbb{R}^N$ has an action on functions on $\mathbb{R}^N$ that is simultaneously isometric in the $L^{2^*}$ and $\dot{H}^1$ norms determines the sharp constants and optimizers of the inequality. There is a way of using competing symmetries to help deduce the full class of extremals of the Sobolev Inequality for functions on $\mathbb{R}^N$ for $N \geq 3$.  This is done as part of a more general setting in a paper by Carlen and Loss, see [8]. Lieb also has a paper, see [15], in which the Sobolev Inequality is derived and its extremals are deduced via an ODE.  In our paper, we deal with a more general setting of the Sobolev Inequality on continuous dimension; we will introduce this in a more precise fashion shortly.  The techniques of Carlen and Loss, as well as Lieb, do not appear to have straightforward adaptations to our settings. The conformal subgroup of $\mathbb{R}^N$ that is invariant on $\| \cdot \|_{2^*}$ and $\| \cdot \|_{\dot{H}^1}$ is generated by the following three operations
\begin{eqnarray}
\text{(inversion) } \varphi (x) &\mapsto& |x|^{-N+2} \varphi (x/|x|) \nonumber \\
\text{(dilation) } \varphi (x) &\mapsto& \sigma^{\frac{N-2}{2}} \varphi (\sigma x)
\text{, } \sigma \in
\mathbb{R}_+ \nonumber \\
\text{(translation) } \varphi (x) &\mapsto& \varphi (x-x_0) \text{, } x_0 \in
\mathbb{R}^N \,. \nonumber
\end{eqnarray}
The extremal functions, $M := \{ z F_{t, x_0} | z \in \mathbb{R}, t \in \mathbb{R}_+, x_0 \in \mathbb{R}^N \}$, of (\ref{SobIneqCl}) comprise an $(N+2)$-dimensional manifold in $\dot{H}^1 (\mathbb{R}^N)$.  We can obtain $M$ by taking the union of the orbits of $z F_{1,0}$ for all $z \in \mathbb{R}$  under the action of conformal group. In fact, $M$ is the union of the orbits of $z F_{1,0}$ for all $z \in \mathbb{R}$ under the subgroup generated by translations and dilations alone.

\subsection{Bianchi and Egnell's Stability Estimate}

A question raised by Brezis and Lieb concerns approximate optimizers of the Sobolev inequality. Suppose for some small
$\epsilon>0$,
\[
\frac{\|\varphi\|_{2^*}}{\|\nabla\varphi\|_2} \geq (1-\epsilon)C_N\ .
\]
Does it then follow that $\varphi$ is close, in some metric, to a Sobolev optimizer? A theorem of Bianchi and Egnell gives a strong positive answer to this question. Define the distance between $M$
and a function $\varphi \in \dot{H}^1 (\mathbb{R}^N)$ as
\begin{equation}\label{SobDistCl}
\delta (\varphi, M) := \inf_{h \in M} \| \nabla (\varphi - h) \|_2 = \inf_{z, t, x_0} \| \nabla (\varphi - z F_{t, x_0} ) \|_2 \,.
\end{equation}
Bianchi and Egnell's answer to Brezis and Lieb's question is summarized in the following
\begin{thm}[Bianchi-Egnell Stability Estimate]\label{BEthm}
There is a positive constant, $\alpha$, depending only on the dimension, $N$, so
that
\begin{equation}
C_N^2 \| \nabla \varphi \|_2^2 -  \| \varphi \|_{2^*}^2 \geq \alpha \delta (\varphi, M)^2 \,,
\end{equation}
$\forall \varphi \in \dot{H}^1 (\mathbb{R}^N)$. Furthermore, the result is sharp in
the sense that it is no longer true if $\delta (\varphi, M)^2$ in (1.5) is replaced
with $\delta (\varphi,
M)^\beta \| \nabla \varphi \|_2^{2 - \beta}$, where $\beta < 2$.
\end{thm}

Recently, Bakry, Gentil, and Ledoux proved a sharp extension of the Sobolev inequality to ``fractional dimensions,'' and showed how it relates to certain optimal Gagliardo-Nirenberg inequalities. V. H. Nguyen has determined all of the extremals in the version of the inequality for real-valued functions. The goal of the present paper is to prove an analogue of Theorem \ref{BEthm} for this extended Sobolev inequality. Actually, the case we treat is more general, because we generalize Nguyen's classification of extremals from real-valued functions to complex-valued functions. We then prove the analogue of Theorem \ref{BEthm} for this generalization of Bakry, Gentil, and Ledoux's Theorem with classification of extremals for complex-valued functions. This is notable, because Bianchi and Egnell prove their stability estimate for real-valued functions only, while our stability estimate is for complex-valued functions. This is one of the aspects that make our proof more intricate than Bianchi and Egnell's, but it is hardly the most notable or the most difficult aspect to deal with. Some of the steps in the proof of our extension of the Bianchi-Egnell Stability Estimate are a fairly direct adaptation of steps in Bianchi and Egnell's proof.  Others are not.  To help highlight these differences, we outline a proof of Theorem~\ref{BEthm} based upon the steps of the proof to our extension of the Bianchi-Egnell Stability Estimate. This outline is provided in subsection 1.5. In the outline, we point out where our approach differs from Bianchi and Egnell's, and in particular, which parts require new arguments.

\subsection{Bakry, Gentil, and Ledoux's Extension of the Sharp Sobolev Inequality with Nguyen's Classification of Extremals}

One can generalize the Sharp Sobolev Inequality to continuous dimension, $N > 2$.  We can define functions on noninteger dimensions by generalizing the notion of radial functions.  To be precise, the $L^p$-norm of a radial function, $\varphi$, on $N$-dimensional Euclidean space is given by
\[
\| \varphi \|_p = \left( \int_{\mathbb{R}_+} | \varphi (\rho) |^p \omega_N \rho^{N-1} \mathrm{d}\rho \right)^{1/p} \,,
\]
where $\omega_N$ is the area of the unit $(N-1)$-sphere given by
\begin{equation}\label{Sph Area}
\omega_N := \frac{2 \pi^{N/2}}{\Gamma(\frac{N}{2})} \,.
\end{equation}
We use this definition to generalize the notion of the area of a unit $(N-1)$-sphere for $N > 0$, possibly noninteger. Correspondingly, the $L^p$-norm of the gradient in radial coordinates is given by
\[
\| \varphi \|_{\dot{W}^{1,p}} := \| \nabla \varphi \|_p = \left( \int_{\mathbb{R}_+} | \varphi' (\rho) |^p \omega_N \rho^{N-1} \mathrm{d} \rho \right) \,.
\]
Allowing $N$ to take noninteger values larger than 2 gives a generalization of the norms $\| \cdot \|_p$ and $\| \cdot \|_{\dot{W}^{1,p}}$ for noninteger dimensions.  In this setting, the analogue of $\dot{H}^1 (\mathbb{R}^N)$ from subsection 1.1 will be denoted $\dot{W}^{1,p} (\mathbb{R}_+, \omega_N \rho^{N-1} \mathrm{d} \rho)$.  $\varphi: [0, \infty) \to \mathbb{C}$ is in $\dot{W}^{1,p} ( \mathbb{R}_+, \omega_N \rho^{N-1} \mathrm{d} \rho)$ if and only if $\| \varphi \|_{\dot{W}^{1,p}} < \infty$ and $\varphi$ is eventually zero in the sense that the measure of $\{ | \varphi(\rho) | > \varepsilon \}$ is finite for all $\varepsilon > 0$ with respect to the measure induced by $\omega_N \rho^{N-1} \mathrm{d} \rho$.  Having established the appropriate ideas and notation, we state
\begin{thm}[Sharp Sobolev Inequality for Radial Functions]\label{sob cts dim}
Let $N > 2$, not necessarily an integer, and $p<N$. Then, for all $\varphi \in \dot{W}^{1,p} (\mathbb{R}_+, \omega_N \rho^{N-1} \mathrm{d}\rho)
\setminus \{ 0 \}$
\begin{equation}
\frac{\| \varphi \|_{p^*}}{\| \varphi \|_{\dot{W}^{1,p}}} \leq \frac{\| F_1 \|_{p^*}}{\| F_1 \|_{\dot{W}^{1,p}}}
=: C_N \,,
\end{equation}
where
\[
p^* = \frac{pN}{N-p} \,,
\]
and
\begin{equation}
F_t (\rho) := k_0 \left( \frac{t}{1 + t^\frac{p}{p-1} \rho^\frac{p}{p-1}} \right)^{\frac{N-p}{p}} \,,
\end{equation}
for $t > 0$ and $k_0 > 0$ a constant such that $\| F_1 \|_{\dot{W}^{1,p}} = 1$. There is equality if and only if $\varphi = z F_t$ for some $t > 0$ and some non-zero $z\in \C$. 
\end{thm}
\noindent G.A. Bliss proved this in 1930, see [4].  In order to deduce Theorem 1.3 from Bliss's theorem, one has to make the substitution of variables given by
\[
x = \rho^{-\frac{N-p}{p-1}} \,,
\]
and set
\[
m = p \,, \text{ and } n = p^* \,.
\]
Actually, Bliss only proves Theorem 1.3 for nonnegative real-valued functions in $\dot{W}^{1,p}$. But, since replacing $\varphi$ by $|\varphi|$ preserves the $L^p$-norm and cannot decrease the $\dot{W}^{1,p}$-norm, Theorem 1.3 must hold for functions with both positive and negative values. Once we have this result, it is easy to generalize to complex-valued functions, as seen in an argument given after the statement of Theorem 1.4.

Bakry, Gentil, and Ledoux proved an extension of the Sobolev Inequality in p. 322-323 of [2]. This extension is for ``cylindrically symmetric'' functions on Euclidean space of $m+n$ dimensions, where one of $m$ and $n$ is not necessarily an integer. In this paper, we will take $m$ to be the number that is not necessarily an integer. Our motivation for considering the extension of the Sobolev Inequality to such functions is to extend the Bianchi-Egnell Stability Estimate of the Sobolev Inequality to cylindrically symmetric functions in continuous dimensions. This in turn allows us to generalize a stability estimate, proved by Carlen and Figalli (see Theorem 1.2 of [6]), for a sharp Gagliardo-Nirenberg inequality, to a family of sharp Gagliardo-Nirenberg inequalities established by Del Pino and Dolbeault in [10]. The continuous variable (i.e. noninteger dimension) in the extension of the Bianchi-Egnell Stability Estimate is necessary for this generalization of the Carlen-Figalli Stability Estimate. The extension of the Bianchi-Egnell Stability Estimate, which we state in detail later in this introduction, is proved in this paper. The generalization of the Carlen-Figalli Stability Estimate, which was one of the original goals of our research, will be proved in a subsequent paper.

To state Bakry, Gentil, and Ledoux's extension of the Sharp Sobolev Inequality, we define the appropriate norms and spaces. First, we establish some properties of cylindrically symmetric functions. Let $\varphi: [ 0, \infty ) \times \mathbb{R}^{n} \to \mathbb{C}$ be a cylindrically symmetric function.  What we mean when we say that $\varphi$ is a cylindrically symmetric function is that if we write $\varphi$ as $\varphi (\rho, x)$, where $\rho$ is a variable with values in $[0,\infty)$ and $x$ is the standard $n$-tuple on $n$ Cartesian coordinates, that the $\rho$ variable acts as a radial variable in $m$-dimensions while the $x$ variable represents the other $n$-dimensions on which $\varphi$ acts.  If $m$ is an integer, then $\varphi$ would also have a representation as a function on $\mathbb{R}^{m+n}$.  For example,
\[
\varphi (\rho, x) = (1 + \rho^2 + |x|)^{-1}
\]
as a cylindrically symmetric function for $m = 2$ and $n = 2$ has the representation as a
function on $\mathbb{R}^4$ given by
\[
\varphi (x_1, x_2, x_3, x_4) = \left( 1 + x_1^2 + x_2^2 + \sqrt{x_3^2 + x_4^2} \right)^{-1} \,,
\]
where $x_1$ and $x_2$ correspond to the $\rho$-variable of $\varphi (\rho, x)$ and $x_3$ and $x_4$ correspond to the $x$-variable of $\varphi (\rho, x)$.  However, we want $m$ to  also possibly be noninteger.  Note, that the value of $m$ is not provided when we give the equation for $\varphi$.  In this paper, the value of $m$ will be determined by the dimensions over which our norms are integrated.  To be more precise, the $m$ dimensions of Euclidean space are encoded in the measure of integration corresponding to the $\rho$ variable.  This measure is $\omega_m \rho^{m-1} \mathrm{d}\rho$, where $\omega_m$ is a generalized notion of the area of the unit $(m-1)$-sphere given by (\ref{Sph Area}) - note that this formula is valid for $m > 0$. In this case, the $L^p$-norm of $\varphi$ is given by
\[
\| \varphi \|_p = \left( \int_{\mathbb{R}^n} \int_{\mathbb{R}_+}
| \varphi (\rho,x) |^p \omega_m \rho^{m-1} \mathrm{d} \rho \mathrm{d} x \right)^{1/p} \,.
\]
Note that when $m$ is an integer and $\tilde{\varphi}: \mathbb{R}^{m+n} \to \mathbb{C}$ is given by $\tilde{\varphi} (\tilde{x}, x) = \varphi ( | \tilde{x} |, x )$, then
\[
\| \varphi \|_p = \left( \int_{\mathbb{R}^n} \int_{\mathbb{R}^m} |
\tilde{\varphi} (\tilde{x}, x) |^p \mathrm{d} \tilde{x} \mathrm{d} x
\right)^{1/p} \,.
\]
The extension of the gradient square norm, i.e. $\| \nabla \cdot \|_2$, is given by
\[
\| \varphi (\rho, x) \|_{\dot{H}^1} := \| \nabla_{\rho, x} \varphi \|_2 = \left( \int_{\mathbb{R}^n} \int_{\mathbb{R}_+} ( | \varphi_\rho |^2 + | \nabla_x \varphi |^2) \omega_m \rho^{m-1} \mathrm{d} \rho
\mathrm{d} x \right)^{1/2} \,,
\]
where the subscript $\rho$ indicates a partial derivative with respect to
$\rho$.  Note that when $m$ is an integer
\[
\| \varphi \|_{\dot{H}^1} = \left( \int_{\mathbb{R}^n} \int_{\mathbb{R}^m} | \nabla_{x, \tilde{x}} \tilde{\varphi} (\tilde{x}, x) |^2 \mathrm{d} \tilde{x} \mathrm{d} x \right)^{1/2} \,.
\]
Let $\Lambda$ be the measure on $\mathbb{R}_+ \times \mathbb{R}^n$ induced by $\omega_m \rho^{m-1} \mathrm{d}\rho \mathrm{d}x$. Accordingly, $\mathrm{d}\Lambda$ is given by
\begin{equation}\label{dLDef}
\mathrm{d}\Lambda = \omega_m \rho^{m-1} \mathrm{d}\rho \mathrm{d}x \,,
\end{equation}
and for measureable $K \subseteq \mathbb{R}_+ \times \mathbb{R}^n$,
\begin{equation}\label{LDef}
\Lambda (K) = \int_K  \omega_m \rho^{m-1} \mathrm{d}\rho \mathrm{d}x \,.
\end{equation}
Then, the space, $\dot{H}_\mathbb{C}^1 (\mathbb{R}_+ \times \mathbb{R}^n, \omega_m \rho^{m-1} \mathrm{d}\rho \mathrm{d}x)$, of complex-valued cylindrically symmetric functions in continuous dimension will be defined as follows: $\varphi \in \dot{H}_\mathbb{C}^1 (\mathbb{R}_+ \times \mathbb{R}^n, \omega_m \rho^{m-1} \mathrm{d}\rho \mathrm{d}x)$ if and only if
\begin{enumerate}
\item
$\varphi$ is a complex-valued cylindrically symmetric function with a distributional gradient,
\item
$\| \nabla \varphi \|_2 < \infty$, and
\item
$\varphi$ is eventually zero in the sense that
\[
\Lambda ( \{ (\rho, x) \in \mathbb{R}_+ \times \mathbb{R}^n \big| | \varphi (\rho,x) | > \varepsilon \} ) < \infty \,,
\]
for all $\varepsilon > 0$.
\end{enumerate}
As a general rule, we will refer to $\dot{H}_\mathbb{C}^1 (\mathbb{R}_+ \times \mathbb{R}^n, \omega_m \rho^{m-1} \mathrm{d}\rho \mathrm{d}x)$ as $\dot{H}_\mathbb{C}^1$.  The subspace of real-valued functions in $\dot{H}_\mathbb{C}^1$ will be denoted by $\dot{H}^1$.  It will often be useful for us to think of $\dot{H}_\mathbb{C}^1$ as the direct sum of two copies of $\dot{H}^1$.  Also, in this setting, we define
\begin{equation}\label{2*aGDef}
2^* := \frac{2(m+n)}{m+n-2} \text{ and } \gamma := \frac{m+n-2}{2} \,.
\end{equation}
Having established this background, we can state Bakry, Gentil, and Ledoux's generalization of the Sobolev Inequality to continuous dimensions with Nguyen's classification of extremals (for reference, see [2] and [18]):
\begin{thm}[Sobolev Inequality Extension]\label{SobExtThm}
Let $m+ n > 2$, $n$ an integer, $m \geq 1$ possibly noninteger. Then, for all $\varphi \in
\dot{H}_\mathbb{C}^1(\mathbb{R}_+
\times \mathbb{R}^n, \omega_m \rho^{m-1} \mathrm{d} \rho \mathrm{d}x)$ 
\begin{equation}\label{SobIneqExt}
\frac{\| \varphi \|_{2^*}}{\| \varphi \|_{\dot{H}^1}} \leq
\frac{\|F_{1,0}\|_{2^*}}{\| F_{1,0} \|_{\dot{H}^1}} =: C_{m,n} \,,
\end{equation}
where
\begin{equation}\label{zFtx0Def}
z F_{t,x_0} (\rho, x) := k_0 z t^\gamma (1 + t^2 \rho^2 + t^2 |x - x_0|^2 )^{-\gamma} \,,
\end{equation}
for $x_0 \in \mathbb{R}^n$, $t \in \mathbb{R}_+$, $z \in \mathbb{C}$, and $k_0 > 0$ a number such that $\| F_{1,0} \|_{\dot{H}^1} = 1$. (\ref{SobIneqExt}) gives equality if and only if $\varphi = z F_{t, x_0}$ for some $t > 0$, $x_0 \in \mathbb{R}^n$, and nonzero $z \in \mathbb{C}$. The extremal functions characterized by (\ref{zFtx0Def}) comprise an $(n+3)$-dimensional manifold, $M \subseteq \dot{H}_\mathbb{C}^1 (\mathbb{R}_+ \times \mathbb{R}^n, \omega_m \rho^{m-1} \mathrm{d}\rho \mathrm{d}x)$.
\end{thm}
Bakry, Gentil, and Ledoux derived (\ref{SobIneqExt}). Nguyen classified the extremals for the statement of this theorem for real-valued functions only. However, once one has the Sobolev Inequality with the classification of extremals for real-valued functions, the generalization to complex-valued functions is easy to deduce.  To see this, consider $\varphi \in \dot{H}_\mathbb{C}^1$ and let
\[
\varphi (\rho, x) = R (\rho, x) e^{i \Theta (\rho, x)} \,,
\]
where $R$ and $\Theta$ are real-valued.  Then
\begin{eqnarray}
C_{m,n} \| \nabla \varphi \|_{\dot{H}^1} &=& C_{m,n} ( \| \nabla_{\rho, x} R \|_2^2 + \| R \nabla_{\rho, x} \Theta
\|_2^2 )^{1/2} \nonumber \\
&\geq& C_{m,n} \| \nabla_{\rho, x} R \|_2 \text{ which by the Sobolev Inequality for
real-valued functions} \nonumber \\
&\geq& \| R \|_{2^*} \nonumber \\
&=& \| \varphi \|_{2^*} \,, \nonumber
\end{eqnarray}
i.e. the Sobolev inequality with sharp constant.  Moreover, the extremals for
complex-valued functions are derived by taking the extremals in the real-valued
case and multiplying them by all possible complex numbers.  We deduce this by
observing that the complex-valued Sobolev Inequality deduced above cannot
achieve equality unless $R \nabla_{\rho, x} \Theta$ is zero almost everywhere and $R$ is
an extremal.

Bakry, Gentil, and Ledoux derive the generalization of the Sharp Sobolev Inequality by relating a Sobolev Inequality on $\mathbb{S}^{n+1}$ to $(\mathbb{R}_+ \times \mathbb{R}^n, \omega_m \rho^{m-1} \mathrm{d}\rho \mathrm{d}x)$ via stereographic projection, see p. 322-323 of [2] for detail. Bakry's proof is an application of an abstract curvature-dimension condition.  Nguyen provides a proof of Theorem \ref{SobExtThm} from a mass-transport approach.  Nguyen's approach, unlike Bakry's, provides a full classification of extremals (for the Sobolev Theorem Extension for real-valued functions).

\subsection{The Main Theorem}

The main theorem that we prove in this paper is a Bianchi-Egnell Stability Estimate for Theorem \ref{SobExtThm}. The extremals of Theorem \ref{SobExtThm} are given by
\[
z F_{t,x_0} (\rho, x) = k_0 z
\left( \frac{t}{1 + t^2 \rho^2 + t^2 |x -
x_0|^2} \right)^\gamma \,,
\]
for $x_0 \in \mathbb{R}^n$, $t \in \mathbb{R}_+$, $z \in \mathbb{C} \setminus \{ 0 \}$, and $k_0 > 0$ a number such that $\| F_{1,0} \|_{\dot{H}^1} = 1$.  These extremal functions comprise an $(n+3)$-dimensional manifold, $M \subseteq \dot{H}_\mathbb{C}^1 (\mathbb{R}_+ \times \mathbb{R}^n, \omega_m \rho^{m-1} \mathrm{d} \rho \mathrm{d}x)$. The distance, $\delta (\varphi, M)$, between this manifold and a function $\varphi \in \dot{H}_\mathbb{C}^1 (\mathbb{R}_+ \times \mathbb{R}^n, \omega_m \rho^{m-1} \mathrm{d} \rho \mathrm{d}x)$ will also be given by (\ref{SobDistCl}).

The stability estimate we prove here is
\begin{thm}[Bianchi-Egnell Extension]\label{MainThm}
There is a positive constant, $\alpha$, depending only on the parameters, $m$
and $n$, $m\geq1$ and $n \geq 1$ an integer, so that
\begin{equation}\label{BEIneq}
C_{m,n}^2 \| \varphi \|_{\dot{H}^1}^2 - \| \varphi \|_{2^*}^2 \geq \alpha \delta (\varphi, M)^2 \,,
\end{equation}
$\forall \varphi \in \dot{H}_\mathbb{C}^1$. Furthermore, the result is sharp in the sense that it is no longer true if $\delta (\varphi, M)^2$ in (\ref{BEIneq}) is
replaced with $\delta (\varphi, M)^\beta \| \varphi \|_{H^1}^{2 - \beta}$, where $\beta < 2$.
\end{thm}
\noindent A notable theorem that we prove in order to prove Theorem \ref{MainThm} is the following local compactness theorem:
\begin{thm}\label{LocalCompactnessThm}
Let $K \subseteq [0,\infty) \times \mathbb{R}^n$ satisfy the cone property in $\mathbb{R}^{n+1}$, $K \subseteq \{ (\rho, x) \in [0, \infty) \times \mathbb{R}^n | \rho_1 < \rho < \rho_2 \}$ for some $0 < \rho_1 < \rho_2 < \infty$, and $\Lambda (K) < \infty$, where $\Lambda$ denotes the measure on $\mathbb{R}_+ \times \mathbb{R}^n$ defined by (\ref{LDef}).  If $(\varphi_j)$ is bounded in $\dot{H}_\mathbb{C}^1$ and $U$ is an open subset of $K$, then for  $1 \leq p < \max \left\{ 2^*, \frac{2n+2}{n-1} \right\}$, there is some $\varphi \in \dot{H}_\mathbb{C}^1$ and some subsequence, $( \varphi_{j_k})$, such that $\varphi_{j_k} \to \varphi$ in $L_\mathbb{C}^p ( U, \omega_m \rho^{m-1} \mathrm{d}\rho \mathrm{d}x)$, where $L_\mathbb{C}^p (U, \omega_m \rho^{m-1} \mathrm{d} \rho \mathrm{d}x)$ denotes the space of complex-valued functions on the weighted space $(U, \omega_m \rho^{m-1} \mathrm{d} \rho \mathrm{d}x)$.
\end{thm}
\noindent The proof of Theorem \ref{LocalCompactnessThm} is not too hard and is provided in the final section of this paper.

\subsection{Outline of a Proof of the Bianchi-Egnell Stability Estimate}

In this subsection, we outline a proof of Theorem \ref{BEthm} based upon the techniques we used to prove Theorem \ref{MainThm}. In this outline, we highlight the differences between proving our Bianchi-Egnell Extension and Bianchi and Egnell's proof of the original Bianchi-Egnell Stability Estimate. There are two key steps to proving Theorem \ref{BEthm}.  These key steps in our proof are the same key steps that Bianchi and Egnell's proof are based upon. However, the details for establishing them are quite different at times.  The first step is a \textit{Local  Bianchi-Egnell Stability Theorem}: If $\varphi \in \dot{H}^1 (\mathbb{R}^N)$ is such that $\| \nabla \varphi \|_2 = 1$ and $\delta (\varphi, M) \leq \frac{1}{2}$, then
\begin{equation}\label{LocBEIneq}
C_N^2 \| \nabla \varphi \|_2^2 - \| \varphi \|_{2^*}^2 \geq \alpha_N \delta (\varphi, M)^2 - \kappa_N \delta (\varphi, M)^{\beta_N} \,,
\end{equation}
where $\kappa_N$ and $\beta_N$ are calculable constants, with  $\beta_N > 2$ and $\alpha_N$ being the smallest positive eigenvalue of a linear operator with nonnegative discrete spectrum.  This allows us to prove Theorem~\ref{BEthm} in a local sense.  The second step is a \textit{Concentration Compactness argument} by which we show that if Theorem~\ref{BEthm} is not true ``outside'' our local region then the Local Bianchi-Egnell Stability Theorem would not be true.

\underline{Step 1 - Prove the Local Bianchi-Egnell Stability Theorem:}

\textit{Part A: Taylor Expand $\| \varphi \|_{2^*}^2$.} Since $\delta
(\varphi, M) \leq \frac{1}{2} < 1 = \| \varphi
\|_{\dot{H}^1}$, there is some $F \in M$ such that
\[
\delta (\varphi, M) = \| \varphi - F \|_{\dot{H}^1} \,.
\]
In fact,
\[
\varphi = F + \delta (\varphi,M) \psi \,,
\]
for some $\psi \in \dot{H}^1$ such that $\| \psi \|_{\dot{H}^1} = 1$ and $\psi
\perp_{\dot{H}^1} F$. 
Taylor expanding $\| F + \varepsilon \psi \|_{2^*}^2$ about $\varepsilon = 0$ to the
second degree, estimating the remainder, and setting $\varepsilon = \delta (\varphi,
M)$, we get that
\[
\| \varphi \|_{2^*}^2 = \| F + \delta (\varphi, M) \psi \|_{2^*}^2
\leq \| F \|_{2^*}^2 + \langle \psi, S \psi \rangle_{\dot{H}^1} \delta (\varphi, M)^2 + \kappa_N \delta (\varphi, M )^{\beta_N} \,,
\]
where $S: \dot{H}^1 \to \dot{H}^1$ is a linear operator, and $\kappa_N$ and $\beta_N$ are calculable constants. Our explicit calculation of the term $\kappa_N \delta (\varphi, M)^{\beta_N}$, which is a bound on the remainder term of the second order Taylor expansion, is an improvement upon Bianchi and Egnell's proof. They use the Brezis-Lieb Lemma to conclude that the remainder term in the second order Taylor expansion is $o (\delta(\varphi,M)^2)$.

\textit{Part B: Use the Taylor Expansion of $\| \varphi \|_{2^*}^2$ to Deduce the Local Bianchi-Egnell Stability Theorem.}  Using the last calculation above, and the facts that $\| F \|_{2^*} = C_N \| F \|_{\dot{H}^1}$ and $\psi \perp_{\dot{H}^1} F$, we conclude that
\[
C_N^2 \| \varphi \|_{\dot{H}^1}^2 - \| \varphi \|_{2^*}^2 \geq \langle \psi, (C_N^2 I - S) \psi \rangle_{\dot{H}^1} \delta (\varphi,
M)^2 - \kappa_N \delta (\varphi, M)^{\beta_N} \,.
\]
Next, we assume the following facts: $C_N^2 I - S: \dot{H}^1 \to \dot{H}^1$ has nonnegative, discrete spectrum whose nullspace is spanned by $F$ and $\frac{\mathrm{d}}{\mathrm{d}t} F$ ($t$ is the parameter in the class of Sobolev optimizers corresponding to dilation - refer to (\ref{ClSobIneqExt}) if necessary) and has a gap at 0.  In Bianchi and Egnell's proof, they prove the analogous facts by calculating a few of the lowest eigenvalues and some of the corresponding eigenfunctions of their operator.  For Bianchi and Egnell, this is done through separation of variables and analysis of the resulting ODEs.  Proving the desired properties of the analogue to $C_N^2 I - S$ in our paper turns out to be more difficult, because the resulting PDE does not separate nicely. So, we delay the proof of these properties of the operator a little bit. Assuming the desired properties, we conclude (\ref{LocBEIneq}), with $\alpha_N$ being the smallest positive eigenvalue of $C_N^2 I - S: \dot{H}^1 \to \dot{H}^1$. Parts C and D are devoted to proving the desired properties of the spectrum and nullspace of $C_N^2 I - S: \dot{H}^1 \to \dot{H}^1$.

\textit{Part C: Show That $S: \dot{H}^1 \to \dot{H}^1$ Is a Compact
Self-Adjoint Operator.} Proving self-adjointness is easy. Proving compactness for the analogue of $S: \dot{H}^1 \to \dot{H}^1$ in our Bianchi-Egnell Extension is difficult, see section 4 for the precise argument. The heart of the argument for proving compactness in our setting is comparing a part of $S$ to a similar operator for which eigenvalue and eigenfunction analysis is easier to carry out. The comparison to the closely related operator is illuminated by a change of coordinates. The change of coordinates in the current setting would be made by representing $\varphi$ in terms of its radial and spherical parts and then making a logarithmic substitution, i.e.
\begin{align}\label{LogCoor}
\varphi (x) &= \varphi (r, \zeta) \text{, } r \in [0, \infty) \text{ and } \zeta \in
\mathbb{S}^{N-1} \nonumber \\
&= r^{-\frac{N-2}{2}} \varphi (\ln r, \zeta) \,.
\end{align}
We would work with the representative of $\varphi$ given by $\varphi (u, \zeta)$, $u = \ln r$, as these coordinates make some of our calculations simpler. Bianchi and Egnell only need to make the change from Euclidean coordinates, i.e. $x \in \mathbb{R}^N$, to radial and spherical coordinates, i.e. $(\rho, \zeta) \in \mathbb{R}_+ \times \mathbb{S}^{N-1}$, to carry out the analysis of the second order operator that they obtain from their Taylor expansion.

\textit{Part D: Show that $C_N^2 I - S: \dot{H}^1 \to \dot{H}^1$ is Positive and Its Nullspace is Spanned by $F$ and $\frac{\mathrm{d}}{\mathrm{d}t} F$.} Establishing positivity of $C_N^2 I - S: \dot{H}^1 \to \dot{H}^1$ and showing that $F$ and $\frac{\mathrm{d}}{\mathrm{d}t} F$ are in its nullspace is not hard.  In our proof, showing that the analogue of $C_N^2 I - S$ does not have any element in its nullspace that is not a linear combination of $F$ and $\frac{\mathrm{d}}{\mathrm{d}t} F$ is not easy. We reduce the 0-eigenvalue problem to an ODE by showing that any element in the nullspace of $C_N^2 I - S$ that is orthogonal to $F$ in $\dot{H}^1$ must satisfy an ODE.  We then show that constant multiples of $\frac{\mathrm{d}}{\mathrm{d}t} F$ are the only solutions off this ODE with finite energy, and consequently the only solutions in $\dot{H}^1$. We conclude by spending some time showing that elements in the nullspace of our analogue to $C_N^2 I - S: \dot{H}^1 \to \dot{H}^1$ must be independent of the variables other than $u$ (keep in mind we are in logarithmic coordinates like those defined in (\ref{LogCoor})). See section 5 for this argument, in particular, see Proposition \ref{FuSpanProp} for the ODE argument.

\underline{Step 2 - Use Concentration Compactness to conclude the Bianchi-Egnell Stability Estimate:}

The local theorem whose proof we just outlined in fact gives
Bianchi and Egnell's Stability Theorem in a region around $M$.  We, as well as
Bianchi and Egnell, use a Concentration Compactness argument to show that some
stability estimate must hold outside this local region also.  However, in
Bianchi and Egnell's case, since they are working in $\mathbb{R}^N$, they are
able to cite this step as a straightforward application of Concentration
Compactness as presented by P.L. Lions or M. Struwe.  In our proof of the
analogue of their stability theorem, the variables and space we work in prevent
such a straightforward application. Our treatment of the Concentration Compactness argument in continuous dimensions may be useful, as the argument is delicate and to our knowledge is absent in the current literature.

\subsection{Outline of Proof of Theorem \ref{MainThm}}

In this subsection, we outline the proof of Theorem \ref{MainThm}. Each section of the proof is named, with its name given in italics, and then described.

\textit{A Second Order Taylor Expansion of $\| f + \varepsilon \psi \|_p^2$ at $\varepsilon = 0$ and an Estimate of the Remainder:} In this section, we improve upon the traditional strategy of bounding the remainder of a second order Taylor expansion of the square of the $p$-norm of a function. In particular, we Taylor expand $\| f + \varepsilon \psi \|_p^2$ around $\varepsilon = 0$ to the second degree for $2 < p < \infty$ and $f$ real-valued and calculate a precise bound for the remainder term. The previously established strategy for dealing with the remainder term is to apply the Brezis-Lieb Lemma to conclude that the remainder is $o (\varepsilon^2)$.

\textit{Statement of a Local Bianchi-Egnell Extension and Outline of Proof:} In this section, we restrict our attention to $\varphi \in \dot{H}_\mathbb{C}^1$ in a neighborhood of $M$, the manifold of extremals of the extended Sobolev Inequality. We state a Bianchi-Egnell Stability Estimate in this local setting and then outline its proof. We begin by showing that
\[
\varphi = F + \delta (\varphi, M) \psi \,,
\]
for some $F \in M$ and then reducing the proof to the case where $F$ is real-valued. Applying the Taylor Expansion result of the previous section and making some additional arguments, we deduce
\begin{eqnarray}
C_{m,n}^2 \| \varphi \|_{\dot{H}^1}^2 - \| \varphi \|_{2^*}^2 &=& C_{m,n}^2 \| F + \delta (\varphi, M) \psi \|_{\dot{H}^1}^2 - \| F + \delta (\varphi, M) \psi \|_{2^*}^2 \nonumber \\
&\geq& \left\langle ( C_{m,n}^2 I - S_t ) \psi, \psi \right\rangle_{\dot{H}_\mathbb{C}^1} \delta (\varphi, M)^2 - \kappa_{2^*} \delta (\varphi, M)^{\beta_{2^*}} \,, \nonumber
\end{eqnarray}
for some calculable constants $\kappa_{2^*} > 0$ and $\beta_{2^*} > 2$, and linear operator $S_t: \dot{H}_\mathbb{C}^1 \to \dot{H}_\mathbb{C}^1$. We then assert that $S_t: \dot{H}_\mathbb{C}^1 \to \dot{H}_\mathbb{C}^1$ is a compact self-adjoint operator and that $C_{m,n}^2 I - S_t: \dot{H}_\mathbb{C}^1 \to \dot{H}_\mathbb{C}^1$ is a positive operator whose nullspace is spanned by $\{ (F,0),(\frac{\mathrm{d}}{\mathrm{d}t}F,0),(0,F) \}$ - we use the convention that for $\varphi \in \dot{H}_\mathbb{C}^1$, $\varphi = (\xi, \eta) \in \dot{H}^1 \oplus \dot{H}^1$ - and is orthogonal in $\dot{H}_\mathbb{C}^1$ to $\psi$. Assuming these facts - their proof is delayed to the next two sections - we deduce that
\[
\left\langle ( C_{m,n}^2 I - S_t ) \psi, \psi \right\rangle_{\dot{H}_\mathbb{C}^1} \geq \alpha_{m,n} \,,
\]
where $\alpha_{m,n}$ is the smallest positive eigenvalue of $C_{m,n}^2 I - S_t: \dot{H}_\mathbb{C}^1 \to \dot{H}_\mathbb{C}^1$. Combining the last two inequalities above, we conclude that
\[
C_{m,n}^2 \| \varphi \|_{\dot{H}^1}^2 - \| \varphi \|_{2^*}^2 \geq \alpha_{m,n} \delta (\varphi, M)^2 - \kappa_{2^*} \delta (\varphi, M)^{\beta_{2^*}} \,,
\]
yielding a Local Bianchi-Egnell Stability Estimate when $\delta (\varphi, M)$ is sufficiently small.

\textit{$S_t: \dot{H}_\mathbb{C}^1 \to \dot{H}_\mathbb{C}^1$ is a Self-Adjoint, Compact Operator:} Self-adjointness follows easily. Compactness does not. The gist of the argument is to show that some closely related positive operator is compact and use this fact to show that $S: \dot{H}_\mathbb{C}^1 \to \dot{H}_\mathbb{C}^1$ is compact. To do this, we use the kernel of the closely related positive operator to show that some positive even power of this operator is trace class. Hence, by a comparison argument, the original operator will be compact. There is a change to logarithmic variables done in this section, like the one suggested in the paragraph with heading ``\textit{Part C}'' in subsection 1.4. This change of variables is essential in helping us figure out the precise form of the closely related operator that we use to prove compactness. The argument presented in this section is somewhat lengthy and delicate.

\textit{The Nullspace of $C_{m,n}^2 I - S_t: \dot{H}_\mathbb{C}^1 \to \dot{H}_\mathbb{C}^1$:} Here we demonstrate that $C_{m,n}^2 I - S_t: \dot{H}_\mathbb{C}^1 \to \dot{H}_\mathbb{C}^1$ is positive, its nullspace is spanned by $\{ (F,0),(\frac{\mathrm{d}}{\mathrm{d}t} F,0),(0,F) \}$, and that all elements in its nullspace are orthogonal in $\dot{H}_\mathbb{C}^1$ to $\psi$. This concludes the proof of the local Bianchi-Egnell Stability Estimate.

\textit{Proof of Theorem \ref{MainThm}:}  Here, we use Concentration Compactness to show that if the Bianchi-Egnell Stability Estimate Extension is not true, then the Local Bianchi-Egnell Stability Estimate that we proved earlier would be untrue.  This, of course, is a contradiction.  Hence, we conclude the main theorem of this paper.  However, Concentration Compactness theorems for cylindrically symmetric functions in continuous dimension are, to our knowledge, absent from literature.  Moreover, it was not clear to us that there is an easy way to take preexisting arguments for functions defined on subsets of $\mathbb{R}^N$ to deduce a Concentration Compactness result for cylindrically symmetric functions on continuous dimensions.  Thus, we have the next section.

\textit{Concentration Compactness:}  We begin by showing that for $( \varphi_j ) \subseteq \dot{H}_\mathbb{C}^1$ such that $\| \varphi_j \|_{2^*} = 1$ for all $j$, and $C_{m,n}^2 \| \varphi_j \|_{\dot{H}^1}^2 - \| \varphi_j \|_{2^*}^2 \to 0$, that dilating the functions appropriately, i.e. taking $( \varphi_j^{\sigma_j} )$, where
\[
\varphi^\sigma (\rho, x) := \sigma^\gamma \varphi ( \sigma \rho, \sigma x ) \text{, } \sigma > 0 \,,
\]
for appropriate $\sigma_j$ gives us a subsequence such that
\begin{equation}\label{BdMe}
\Lambda ( \{ | \varphi_{j_k}^{\sigma_{j_k}} ( x ) | > \varepsilon, \rho \leq 4 \} ) > C \,,
\end{equation}
for some $C > 0$ and $\varepsilon > 0$ and where $\Lambda$ denotes the measure defined in (\ref{LDef}). We then apply an analogue of a Concentration Compactness Theorem proved by Lieb to conclude that a translated subsequence of $( \varphi_{j_k}^{\sigma_{j_k}} (x) )$ converges weakly in $\dot{H}_\mathbb{C}^1$ to a nonzero element $\varphi$. The tricky part of proving concentration compactness in this section is proving (\ref{BdMe}). The gist of the argument proving (\ref{BdMe}) is to show that if we dilate the functions in our sequence appropriately and then take their symmetric decreasing rearrangements, then some subsequence of the modified sequence will satisfy the $p, q, r$-Theorem on a subregion of an annulus.  What makes this complicated in our case is that we are working with cylindrically symmetric functions in continuous dimensions. Using some straightforward functional analysis arguments, we conclude that $\varphi \in M$ and $\| \varphi \|_{2^*} = 1$. Noteworthy of these arguments is the use of a Local Compactness Theorem on cylindrically symmetric functions on continuous dimension. This Theorem substitutes for the Rellich-Kondrachov Theorem and can be thought of as a weaker version of the Rellich-Kondrachov Theorem for cylindrically symmetric functions on continuous dimensions. The precise statement of our Local Compactness Theorem is Theorem \ref{LocalCompactnessThm} in subsection 1.4.

\textit{Proof of Local Compactness Theorem:} Here we prove Theorem \ref{LocalCompactnessThm}.

\subsection{Applications of Bianchi and Egnell's Stability Analysis and Our Bianchi-Egnell Stability Estimate}

We begin this subsection with a discussion of what motivated us to pose and prove the Bianchi-Egnell Extension, Theorem \ref{MainThm}. Bakry, Gentil, and Ledoux showed that their extension of the Sobolev Inequality implies a sharp family of Gagliardo-Nirenberg inequalities that had only recently been proven by Del Pino and Dolbeault, see [10]. Carlen and Figalli explored this connection and used Bakry, Gentil, and Ledoux's techniques to establish a stability estimate for a single case in this family of sharp Gagliardo-Nirenberg inequalities. An essential step to obtaining this stability estimate, is a direct application of the Bianchi-Egnell Stability Estimate. Carlen and Figalli's use of the Bianchi-Egnell Stability Estimate in establishing a stability estimate of the Gagliardo-Nirenberg inequality, raises the question as to whether or not one can generalize their stability estimate to the entire family of Gagliardo-Nirenberg inequalities classified by Del Pino and Dolbaeault. The answer to this question is yes, but using Carlen and Figalli's techniques requires the Bianchi-Egnell Stability Estimate that we prove in this paper. The Bianchi-Egnell Stability Estimate on integer dimensions is not sufficient, because there is an integration step that links the Sobolev Inequality to these sharp Gagliardo-Nirenberg inequalities. This integration step induces a correlation between the dimension of the Sobolev Inequality and a parameter in the Gagliardo-Nirenberg inequalities of Del Pino and Dolbeault. In order to deduce a stability estimate for the Gagliardo-Nirenberg inequalities corresponding to all possible values of this paremeter, one needs a Sobolev Inequality and a Bianchi-Egnell Stability Estimate for cylindrically symmetric functions on continuous dimensions. Thus, we set out to prove the Bianchi-Egnell Stability Estimate Extension in this paper, in order to develop this necessary piece of machinery in proving a stability estimate for the full class of sharp Gagliardo-Nirenberg inequalities of Del Pino and Dolbeault. We will use our Bianchi-Egnell Stability Estimate Extension in a future paper to prove a stability estimate for the full family of Gagliardo-Nirenberg inequalities classified by Del Pino and Dolbeault in [10].

The techniques used to prove the stability estimate of Bianchi and Egnell have been used in solving many partial differential equations, see [1], [17], and [19] for some examples. The Gagliardo-Nirenberg stability estimate of Carlen and Figalli, which is a direct application of the Bianchi-Egnell Stability Estimate, was also used to solve a Keller-Segel equation, see [6].

\section{A Second Order Taylor Expansion of $\| f + \varepsilon \psi \|_p^2$ at $\varepsilon = 0$ and an Estimate of the Remainder}

In this section, we will Taylor expand $\| f + \varepsilon \psi \|_p^2$ at $\varepsilon = 0$ to the second degree and estimate the remainder term. This may seem like a simple enough process, but this expansion lies at the heart of Theorem \ref{MainThm}. Moreover, our estimate is an improvement upon the conventional method of dealing with the remainder term. The conventional method is to apply the Brezis-Lieb Lemma to conclude that the remainder term is $o(\varepsilon^2)$. We show that the remainder is bounded by $\kappa_p |\varepsilon|^{\beta_p}$ with $\kappa_p > 0$, $\beta_p > 2$ calculable 
constants.

Although the stability estimate that we prove in our paper is for complex-valued functions, we will reduce our calculations to real-valued functions. To this end, we treat $L_\mathbb{C}^p$, complex-valued $L^p$ functions, as the direct sum of two copies of the space of real-valued $L^p$ functions.  To be more precise, we let $L_\mathbb{C}^p = L^p \oplus L^p$, where $L^p$ denotes the space of real-valued $L^p$ functions.  When representing elements in $L_\mathbb{C}^p$ as a direct sum, the first coordinate will represent the real part of the function and the second coordinate will represent the imaginary part; i.e., if $\psi \in L_\mathbb{C}^p$, then $\psi = (\xi, \eta)$ for some $\xi, \eta \in L^p$.  The calculation of the Taylor Expansion is summarized in the following
\begin{thm}\label{TayThm}
 Let $P_\psi : [-1,1] \to \mathbb{R}$ be given by
\begin{equation}
 P_\psi ( \varepsilon ) = \| f + \varepsilon \psi \|_p^2 \,,
\end{equation}
for $\psi \in L_\mathbb{C}^p$, real-valued $f \in L_\mathbb{C}^p$, $\| \psi \|_p = \| f \|_p = 1$, and $2 < p < \infty$. Let $\psi = (\xi, \eta)$. Then,
\begin{equation}\label{TayEst}
 \left| P_\psi (\varepsilon) - 1 - 2 \langle f |f|^{p-2}, \xi \rangle_{L^2} \varepsilon - \langle \mathcal{L}_{f,p} \psi, \psi \rangle_{L^2 \oplus L^2} \varepsilon^2 \right| \leq \kappa_p |\varepsilon|^{\beta_p} \,,
\end{equation}
where $\mathcal{L}_{f,p} = \mathcal{L}_{f,p}^{Re} \oplus \mathcal{L}_{f,p}^{Im}$ is given by
\begin{align}
\mathcal{L}_{f,p}^{Re} \xi &= -(p-2) \left( \int f |f|^{p-2} \xi \right) f |f|^{p-2} + (p-1) |f|^{p-2} \xi \label{LRe} \\
\mathcal{L}_{f,p}^{Im} \eta &= |f|^{p-2} \eta \,, \label{LIm}
\end{align}
and
\begin{align}
{\beta_p} &=
\begin{cases}
 3 &\text{if } p \geq 4 \\
1 + \frac{p}{2} &\text{if } 2 < p \leq 4
\end{cases} \label{beta_p} \\
\kappa_p &=
\begin{cases}
 \frac{4}{3} (5p^2 - 12p + 7) &\text{if } p \geq 4 \\
\frac{16(p-1)}{p(p+2)} \left[ 4 (p-2) + \left( \frac{3p}{p-2} \right)^{\frac{p}{2} - 1} \right] &\text{if } 2 < p \leq 4 \,.
\end{cases} \label{kappa_p}
\end{align}
\end{thm}

\begin{proof}
The proof of Theorem \ref{TayThm} breaks into three parts.  The first is the Taylor Expansion summarized in
\begin{lm}\label{TayLm1}
Given the assumptions of Theorem \ref{TayThm},
\begin{equation}\label{TayExp}
P_\psi (\varepsilon) = 1 + 2 \langle f |f|^{p-2}, \xi \rangle_{L^2} \varepsilon + \langle \mathcal{L}_{f,p} \psi, \psi \rangle_{L^2 \oplus L^2} \varepsilon^2 + \int_0^\varepsilon \int_0^s P_\psi'' (y) - P_\psi''(0) \mathrm{d}y \mathrm{d}s \,.
\end{equation}
\end{lm}

\begin{proof}
Taylor expanding $P_\psi$ to the second order with remainder yields
\begin{equation}\label{TayExp1}
 P_\psi (\varepsilon) = P_\psi (0) + P_\psi' (0) \varepsilon + \frac{1}{2} P_\psi'' (0) \varepsilon^2 + \int_0^\varepsilon \int_0^s P_\psi'' (y) - P_\psi''(0) \mathrm{d}y \mathrm{d}s \,.
\end{equation}
A straightforward calculation shows that
\begin{equation}\label{TayExp2}
 P_\psi (0) = 1 \,, \text{ } P_\psi' (0) = 2 \langle f |f|^{p-2} \,, \text{ } \xi \rangle_{L^2} \,, \text{ } \frac{1}{2} P_\psi'' (0) = \langle \mathcal{L}_{f,p} \psi, \psi \rangle_{L^2 \oplus L^2} \,.
\end{equation}
Combining (\ref{TayExp1}) and (\ref{TayExp2}) yields (\ref{TayExp}).
\end{proof}

The expansion above is a straightforward calculation.  Much of the work from here on out is devoted to identifying the behavior of the individual terms in (\ref{TayEst}) as applied to our particular 
setup. We begin this process by getting an estimate on the remainder term of the right hand side of (\ref{TayEst}). This is a bit subtle. Our proof hinges on a piece of machinery developed by Carlen, Frank, and Lieb in [7]. This machinery is the \textit{duality map}, $\mathcal{D}_p$, on functions from $L_\mathbb{C}^p$ to the unit sphere in $L_{\mathbb{C}}^{p'}$ (we take the convention that $\frac{1}{p'} := 1 - \frac{1}{p}$) given by
\[
\mathcal{D}_p (g) = \| g \|_p^{1-p} |g|^{p-2} \overline{g} \,.
\]
As a consequence of uniform convexity of $L_\mathbb{C}^p$ for $1 < p < \infty$, Carlen, Lieb, and Frank deduce Holder continuity of the duality map, see Lemma 3.3 of [7] for detail. This Holder continuity is crucial in bounding the remainder term in (\ref{TayEst}). We set up the application of this Holder continuity of the duality map in the second part of the proof of Theorem \ref{TayThm} summarized by
\begin{lm}\label{TayLm2}
Given the assumptions of Theorem \ref{TayThm},
\begin{equation}\label{P''Bd}
 \left| P_\psi'' (y) - P_\psi''(0) \right| \leq 4(p-2) \left\| \mathcal{D}_p (f^{(y)}) - \mathcal{D}_p (f) \right\|_{p'} + 2 (p-1) \left\| \mathcal{D}_\frac{p}{2} ([ f^{(y)} ]^2) - \mathcal{D}_\frac{p}{2} (f^2) \right\|_{\left( \frac{p}{2} \right)'} \,,
\end{equation}
for $f^{(y)} := f + y \psi$.
\end{lm}

\begin{proof}
We begin with the inequality
\begin{align}\label{A1A2A3}
\frac{1}{2(p-2)} \left| P_\psi '' (y) - P_\psi '' (0) \right| \leq& \left| \| f^{(y)} \|_p^{2-2p} \left( \int | f^{(y)} |^{p-2} [ (f+y\xi) \xi  + y \eta^2 ] \right)^2 - \left( \int |f|^{p-2} f \xi \right)^2 \right| \nonumber \\
&+ \left| \| f^{(y)} \|_p^{2-p} \left( \int | f^{(y)} |^{p-4} [ (f + y \xi) \xi + y \eta^2 ]^2 \right) - \left( \int |f|^{p-4} f^2 \xi^2 \right) \right| \nonumber \\
&+ (p-2)^{-1} \left| \| f^{(y)} \|_p^{2-p} \left( \int | f^{(y)} |^{p-2} | \psi |^2 \right) - \left( \int |f|^{p-2} |\psi|^2 \right) \right| \nonumber \\
=:& A_1 + A_2 + A_3 \,.
\end{align}
We bound $A_1$, $A_2$, and $A_3$ below:
\begin{align}
A_1 &= \left| \| f^{(y)} \|_p^{2-2p} \left( \int |f^{(y)}|^{p-2} [ (f+y\xi) \xi + y \eta^2 ] \right)^2 - \left( \int |f|^{p-2} f \xi \right)^2 \right| \nonumber \\
& \text{using the fact that $a^2-b^2 = (a-b)(a+b)$ and elementary properties of complex numbers} \nonumber \\
&\leq \left| \left( \int \mathcal{D}_p (f^{(y)}) \psi \right) + \left( \int \mathcal{D}_p (f) \psi \right) \right| \cdot \left| \left( \int \mathcal{D}_p (f^{(y)}) \psi \right) - \left( \int \mathcal{D}_p (f) \psi \right) \right| \nonumber \\
&\leq \left( \frac{ \| |f^{(y)}|^{p-1} \|_{p'} }{ \| f^{(y)} \|_p^{p-1} } + 1 \right) \left( \int [ \mathcal{D}_p (f^{(y)} ) - \mathcal{D}_p (f) ] \psi \right) \nonumber \\
&\leq 2 \| \mathcal{D}_p (f^{(y)}) - \mathcal{D}_p (f) \|_{p'} \,, \text{ and} \label{A1Bd} \\
A_2 &= \left| \| f^{(y)} \|_p^{2-p} \left( \int |f^{(y)}|^{p-4} [ (f + y \xi) \xi + y \eta^2 ]^2 \right) - \left( \int |f|^{p-4} f^2 \xi^2 \right) \right| \nonumber \\
&\leq \int \left| \| f^{(y)} \|_p^{2-p} | f^{(y)}|^{p-4} [ (f + y\xi) \xi + y \eta^2 ]^2 - |f|^{p-4} f^2 \xi^2 \right| \nonumber \\
& \text{using the fact that $a^2 - b^2 = (a-b)(a+b)$ and elementary properties of complex numbers} \nonumber \\
&\leq \int \left| \mathcal{D}_{\frac{p}{2}} ( [f^{(y)}]^2 ) - \mathcal{D}_{\frac{p}{2}} (f^2) \right| |\psi^2| \,, \text{ which by Holder's Inequality} \nonumber \\
&\leq \| \mathcal{D}_{\frac{p}{2}} ( [f^{(y)}]^2 ) - \mathcal{D}_{\frac{p}{2}} (f^2) \|_{( \frac{p}{2} )'} \,. \label{A2Bd}
\end{align}
And, a process similar to the one used to bound $A_2$ shows that
\begin{equation}\label{A3Bd}
A_3 \leq (p-2)^{-1} \left\| \mathcal{D}_\frac{p}{2} ([f^{(y)}]^2) - \mathcal{D}_\frac{p}{2} (f^2) \right\|_{\left( \frac{p}{2} \right)'} \,.
\end{equation}
Combining (\ref{A1A2A3})-(\ref{A3Bd}), we conclude (\ref{P''Bd}).
\end{proof}

In the third part of the proof of Theorem \ref{TayThm}, we estimate the right hand side of (\ref{P''Bd}) via Holder continuity of the duality map, $\mathcal{D}_p$. For completeness, we state this Holder continuity property below:
\begin{lm}[Holder Continuity of the Duality Map]\label{HCLm}
Let $f,g \in L^p (X,\mu)$ for $X$ a measure space and $\mu$ its measure. Then 
\begin{align}
\left\| \mathcal{D}_p (f) - \mathcal{D}_p (g) \right\|_{p'} &\leq 4(p-1) \frac{\| f - g\|_p}{\|f\|_p + \|g\|_p} \,, \text{ for } p \geq 2 \\
\left\| \mathcal{D}_p (f) - \mathcal{D}_p (g) \right\|_{p'} &\leq 2 \left( p' \frac{\| f - g \|_p}{\| f \|_p + \| g \|_p} \right)^{p-1} \,, \text{ for } 1 < p \leq 2 \,.
\end{align}
\end{lm}
\noindent Applying Lemma \ref{HCLm} to the right hand side of (\ref{P''Bd}) yields
\begin{eqnarray}
\left| P_\psi''(y) - P_\psi'' (0) \right| &\leq&
\begin{cases}
16(p-1)(p-2) \frac{\| y \psi \|_p}{\| f^{(y)} \|_p + 1} + 8(p-1)^2 \frac{ \| 2 f y \psi + y^2 \psi^2 \|_{\frac{p}{2}} }{\| [f^{(y)}]^2 \| + \| f^2 \|_{\frac{p}{2}} } & \text{if } \frac{p}{2} \geq 2 \\
16(p-1)(p-2) \frac{\| y \psi \|_p}{\| f^{(y)} \|_p + 1} + 4(p-1) \left[ \left(\frac{p}{2} \right)' \frac{ \| 2 f y \psi + y^2 \psi^2 \|_{\frac{p}{2}} }{\| [f^{(y)}]^2 \| + \| f^2 \|_{\frac{p}{2}} } \right]^{\frac{p}{2} - 1} & \text{if } 1 < \frac{p}{2} \leq 2
\end{cases} \nonumber \\
&& \text{which by the Holder and triangle inequalities, and because $y \in [-1,1]$} \nonumber \\
&\leq&
\begin{cases}
5 (9p^2 - 12p + 7) |y| & \text{if } \frac{p}{2} \geq 2 \\
4 (p-1) \left[ 4(p-2) + \left( \frac{3p}{p-2} \right)^{\frac{p}{2} - 1} \right] |y|^{\frac{p}{2}-1} & \text{if } 1 < p \leq 2 \,.
\end{cases} \nonumber
\end{eqnarray}
Thus,
\begin{align}\label{ReBd}
\left| \int_0^\varepsilon \int_0^s P_\psi '' (y) - P_\psi '' (0) \mathrm{d}y \mathrm{d}s \right| &\leq
\begin{cases}
5 (8p^2 - 12p + 7) \int_0^{|\varepsilon|} \int_0^s y \mathrm{d}y \mathrm{d}s & \text{if } \frac{p}{2} \geq 2 \\
4(p-1) \left[ 4 (p-2) + \left( \frac{3p}{p-2} \right)^{\frac{p}{2}-1} \right] \int_0^{|\varepsilon|} \int_0^s y^{\frac{p}{2} - 1} \mathrm{d}y \mathrm{d}s & \text{if } 1 < \frac{p}{2} \leq 2
\end{cases} \nonumber \\
&\leq \kappa_p |\varepsilon|^{\beta_p} \,.
\end{align}
Combining (\ref{ReBd}) with Lemma \ref{TayLm1} we conclude Theorem \ref{TayThm}.
\end{proof}

\section{Statement of a Local Bianchi-Egnell Extension and Outline of Proof}

In this section, we state and outline the proof of the following
\begin{thm}[Local Bianchi-Egnell Extension]\label{LocBEThm}
Let $\varphi \in \dot{H}_\mathbb{C}^1 (\mathbb{R}_+ \times \mathbb{R}^n, \omega_m \rho^{m-1} \mathrm{d}\rho \mathrm{d}x)$ be such that
\begin{equation}\label{SaDCdn}
\| \varphi \|_{\dot{H}^1} = 1 \text{, and } \delta (\varphi, M) \leq \frac{1}{2} \,.
\end{equation}
Then
\begin{equation}\label{LBEwRem}
C_{m,n}^2 \| \varphi \|_{\dot{H}^1}^2 - \| \varphi \|_{2^*}^2 \geq \alpha_{m,n} \delta (\varphi,M)^2 - \frac{\kappa_{2^*} C_{m,n}^2}{4 \cdot 3^{\frac{\beta_{2^*}}{2} - 1}} \delta (\varphi, M)^{\beta_{2^*}} \,,
\end{equation}
where $\alpha_{m,n}$ is the smallest positive eigenvalue of the operator $C_{m,n}^2 I - A^{-1} \mathcal{L}_{F_{1,0},2^*}: \dot{H}_\mathbb{C}^1 \to \dot{H}_\mathbb{C}^1$ for $\mathcal{L}_{f,p}$ as 
defined in section 2 and
\[
A = - \Delta_x - \frac{\partial^2}{\partial \rho^2} - \frac{m-1}{\rho} \frac{\partial}{\partial \rho} \,.
\]
This gives a local version of a Bianchi-Egnell stability estimate for $\varphi \in \dot{H}_\mathbb{C}^1$ such that $\| \varphi \|_{\dot{H}^1} = 1$, provided
\begin{equation}\label{LocCdn}
\delta (\varphi, M) \leq \min \left\{ \left( \frac{\alpha_{m,n}}{2 \kappa_{2^*}} \right)^{1/(\beta_{2^*} - 2)}, \frac{1}{2} \right\} \,.
\end{equation}
\end{thm}

We begin with the proof of the last sentence of Theorem \ref{LocBEThm}
\begin{proof}
Let $\varphi \in \dot{H}_\mathbb{C}^1$ obey (\ref{LocCdn}).  Then,
\begin{eqnarray}
C_{m,n}^2 \| \varphi \|_{\dot{H}^1}^2 - \| \varphi \|_{2^*}^2 &\geq& \alpha_{m,n}
\delta (\varphi, M)^2 - \kappa_{2^*} \delta(\varphi, M)^{\beta_{2^*}} \nonumber
\\
&=& \delta(\varphi, {M})^2 \left[ \alpha_{m,n} - \kappa_{2^*} \delta
(\varphi, M)^{\beta_{2^*} - 2} \right] \nonumber \\
&\geq& \frac{\alpha_{m,n}}{2} \delta (\varphi, M)^2 \text{, by (\ref{LocCdn}).} \nonumber
\end{eqnarray}
The inequality we deduced above is in fact a Bianchi-Egnell stability estimate as characterized by (\ref{BEIneq}) with $\alpha$ in (\ref{BEIneq}) equal $\frac{\alpha_{m,n}}{2}$.
\end{proof}
\noindent With the last sentence out of the way, we only need to show that (\ref{SaDCdn}) implies (\ref{LBEwRem}).

Once we have proved Theorem \ref{LocBEThm}, we will be able to use it and a Concentration Compactness argument to prove Theorem \ref{MainThm}.  About half of the work in this paper is devoted to proving 
Theorem \ref{LocBEThm} - or rather the following reduction of Theorem \ref{LocBEThm}:
\begin{lm}\label{LocBELm}
Let $\varphi \in \dot{H}_\mathbb{C}^1$ obey (\ref{SaDCdn}) and be such that
\begin{equation}\label{MinFCdn}
\delta (\varphi,M) = \| \varphi - z F_{t,0} \|_{\dot{H}^1} \text{, for some $z \in \mathbb{R}$ and $t \in \mathbb{R}_+$} \,.
\end{equation}
Then (\ref{LBEwRem}) holds for $\varphi$.
\end{lm}
\noindent It is important to note that we stipulated that $z \in \mathbb{R}$, not $\mathbb{C}$, because this is half of the reduction. The other half of the reduction is that a minimizing element of $\varphi$ is $zF_{t,0}$ as opposed to a more general $zF_{t,x_0}$.

In order to prove Theorem \ref{LocBEThm}, we would like to use Theorem \ref{TayThm}. This requires $\varphi$ to be in the form $f + \delta (\varphi,M) \psi$. $\varphi$ is in fact in such a form due to
\begin{lm}\label{DistMinLm}
Let $\varphi \in \dot{H}_\mathbb{C}^1$ be such that
\begin{equation}\label{D<NCdn}
\delta (\varphi, M) < \| \varphi \|_{\dot{H}^1} \,.
\end{equation}
Then, $\exists z F_{t,x_0} \in M$ such that
\begin{equation}\label{MinFCdn1}
\delta (\varphi, M) = \| \varphi - z F_{t,x_0} \|_{\dot{H}^1} \,.
\end{equation}
\end{lm}

\begin{proof}
$M$ can be viewed as a continuous imbedding of $\mathbb{C} \times \mathbb{R}_+ \times \mathbb{R}^n$ into $\dot{H}_\mathbb{C}^1$ by the map
\[
(z,t,x_0) \mapsto z F_{t,x_0} \,.
\]
Thus, the existence of an element $z F_{t,x_0}$ satisfying (\ref{MinFCdn1}) is a consequence of the continuity of the map from $\mathbb{C} \times \mathbb{R}_+ \times \mathbb{R}$ to $\mathbb{R}$ given by
\[
(z,t,x_0) \mapsto \| \varphi - z F_{t,x_0} \|_{\dot{H}^1}
\]
and the fact that (\ref{D<NCdn}) implies that any such minimizing triple $(z,t,x_0)$ must occur on a set away from the origin and infinity in $\mathbb{C} \times \mathbb{R}_+ \times \mathbb{R}^n$. 
Thus, a minimizing element $z F_{t,x_0}$ exists by lower semicontinuity on bounded sets in Euclidean space. This is an adaptation of a proof to an analogous statement in [3], see Lemma 1 
in [3] for more detail.
\end{proof}

Applying Lemma \ref{DistMinLm}, we conclude that $\varphi$ obeying the assumptions of Theorem \ref{LocBEThm} has the form
\[
\varphi = z F_{t,x_0} + \delta (\varphi,M) \psi \,,
\]
for some $\psi \in \dot{H}_\mathbb{C}^1$ such that $\| \psi \|_{\dot{H}^1} = 1$.  If we multiply $\varphi$ by $\overline{z}/|z|$ and translate $(0,x_0)$ to the origin - both operations are invariant 
on $\| \cdot \|_{\dot{H}^1}$, $\| \cdot \|_{2^*}$, $\delta (\cdot, M)$ - then we end up with some $\tilde{\varphi} = |z| F_{t,0} + \delta (\tilde{\varphi},M) \psi$ whose relevant norms and distances 
are the same as $\varphi$.  Thus, if (\ref{LBEwRem}) holds for functions obeying (\ref{SaDCdn}) and (\ref{MinFCdn}), then they hold for all functions obeying (\ref{SaDCdn}), i.e. Theorem \ref{LocBEThm} 
is a corollary of Lemma \ref{LocBELm}.

Now that we have shown that Theorem \ref{LocBEThm} is a corollary of Lemma \ref{LocBELm}, we will use Theorem \ref{TayThm} to begin to prove Lemma \ref{LocBELm}.  To be precise, we will prove the following
\begin{lm}\label{LocBELm1}
Let $\varphi \in \dot{H}_\mathbb{C}^1$ satisfy the assumptions of Lemma \ref{LocBELm}.  Then
\begin{equation}\label{LBEwRem1}
C_{m,n}^2 \| \varphi \|_{\dot{H}^1}^2 - \| \varphi \|_{2^*}^2 \geq \left\langle ( C_{m,n}^2 I - A^{-1} \mathcal{L}_{C_{m,n}^{-1} F_{t,0}, 2^*} ) \psi, \psi \right\rangle_{\dot{H}_\mathbb{C}^1} \delta (\varphi, M)^2 - \frac{\kappa_{2^*} C_{m,n}^2}{4 \cdot 3^{\frac{\beta_{2^*}}{2} - 1}} \delta (\varphi, M)^{\beta_{2^*}} \,,
\end{equation}
where $\psi \in \dot{H}_\mathbb{C}^1$ is such that
\begin{equation}\label{phisum}
\varphi = zF_{t,0} + \delta (\varphi,M) \psi \text{, and } \| \psi \|_{\dot{H}^1} = 1 \,.
\end{equation}
\end{lm}

\begin{proof}
$\varphi$ satisfies (\ref{phisum}) and $\psi \perp_{\dot{H}_\mathbb{C}^1} F_{t,0}$ as a result of (\ref{MinFCdn}). Consistent with the notation of section 2, we let $\psi = (\xi, \eta) \in \dot{H}^1 \oplus \dot{H}^1$. Applying Theorem \ref{TayThm} to $\| \varphi \|_{2^*}^2$ yields
\begin{align}\label{Exp}
\| \varphi \|_{2^*}^2 =& z^2 C_{m,n}^2 \left\| \frac{F_{t,0}}{C_{m,n}} + \frac{\delta (\varphi,M) \| \psi \|_{2^*}}{z C_{m,n}} \cdot \frac{\psi}{\| \psi \|_{2^*}} \right\|_{2^*}^2 \nonumber \\
\leq& z^2 C_{m,n}^2 + 2 z C_{m,n}^{2 - 2^*} \langle F_{t,0} |F_{t,0}|^{2^*-2}, \xi \rangle_{L^2} \delta (\varphi, M) + \left\langle \mathcal{L}_{C_{m,n}^{-1} F_{t,0},2^*} \psi, \psi \right\rangle_{L^2 \oplus L^2} \delta (\varphi, M)^2 \nonumber \\
&+ \frac{\kappa_{2^*} \| \psi \|_{2^*}^{\beta_{2^*}}}{(z C_{m,n})^{\beta_{2^*}-2}} \delta (\varphi, M)^{\beta_{2^*}} \,.
\end{align}
We claim that the coefficient of first order in the right hand side of (\ref{Exp}) equals zero. To see this, we consider the function $R_\psi: \mathbb{R} \to \mathbb{R}$ given by
\[
R_\psi (\varepsilon) = \frac{\| F_{t,0} + \varepsilon \psi \|_{2^*}^2}{\| F_{t,0} + \varepsilon \psi \|_{\dot{H}^1}^2} \,.
\]
Since $F_{t,0}$ is an extremal of the Sobolev Inequality
\begin{align}\label{FOCdn}
& 0 = R_\psi' (0) = 2 \| F_{t,0} \|_{2^*}^{2-2^*} \left\langle F_{t,0} | F_{t,0} |^{2^*-2}, \xi \right\rangle_{L^2} - 2 \| F_{t,0} \|_{2^*}^2 \left\langle F_{t,0}, \xi \right\rangle_{\dot{H}^1} \nonumber \\
&\implies \left\langle F_{t,0} | F_{t,0} |^{2^*-2}, \xi \right\rangle_{L^2} = \| F_{t,0} \|_{2^*}^{2^*} \left\langle F_{t,0}, \xi \right\rangle_{\dot{H}^1} = 0 \text{, as } \psi \perp_{\dot{H}_\mathbb{C}^1} F_{t,0} \,.
\end{align}
Thus, by (\ref{Exp}), (\ref{FOCdn}), and the fact that $\psi \perp_{\dot{H}_\mathbb{C}^1} F_{t,0}$
\begin{align}\label{SimpExp}
C_{m,n}^2 \| \varphi \|_{\dot{H}^1}^2 - \| \varphi \|_{2^*}^2 \geq& C_{m,n}^2 \left[ \| z F_{t,0} \|_{\dot{H}^1}^2 + \| \psi \|_{\dot{H}^1}^2 \delta (\varphi, M)^2 \right] \nonumber \\
& - \left[ C_{m,n}^2 z^2 + \left\langle \mathcal{L}_{C_{m,n}^{-1} F_{t,0},2^*} \psi, \psi \right\rangle_{L^2 \oplus L^2} \delta (\varphi, M)^2 + \frac{\kappa_{2^*} \| \psi \|_{2^*}^{\beta_{2^*}}}{(z C_{m,n})^{\beta_{2^*}-2}} \delta (\varphi, M)^{\beta_{2^*}} \right] \nonumber \\
& \text{and since $\left\langle A \varphi_1, \varphi_2 \right\rangle_{L^2 \oplus L^2} = \langle \varphi_1, \varphi_2 \rangle_{\dot{H}_\mathbb{C}^1}$ for all $\varphi_1, \varphi_2 \in \dot{H}_\mathbb{C}^1$} \nonumber \\
=& \left\langle ( C_{m,n}^2 I - A^{-1} \mathcal{L}_{C_{m,n}^{-1} F_{t,0},2^*} ) \psi, \psi \right\rangle_{\dot{H}_\mathbb{C}^1} \delta (\varphi, M)^2 - \frac{\kappa_{2^*} \| \psi \|_{2^*}^{\beta_{2^*}}}{(z C_{m,n})^{\beta_{2^*}-2}} \delta ( \varphi, M)^{\beta_{2^*}} \,.
\end{align}
(\ref{SaDCdn}), (\ref{phisum}), and the Sobolev Inequality imply
\begin{equation}\label{SCdn}
|z| \geq \frac{\sqrt{3}}{2} \text{, and } \| \psi \|_{2^*} \leq C_{m,n}/2 \,.
\end{equation}
(\ref{SimpExp}) and (\ref{SCdn}) allow us to conclude (\ref{LBEwRem1}).
\end{proof}

Having proved that under the assumptions of Lemma \ref{LocBELm} that (\ref{LBEwRem1}) holds, we only need to show that
\begin{equation}\label{EVCdn}
\left\langle \left( C_{m,n}^2 I - A^{-1} \mathcal{L}_{C_{m,n}^{-1} F_{t,0},0} \right) \psi, \psi \right\rangle_{\dot{H}_\mathbb{C}^1} \geq \alpha_{m,n} \,,
\end{equation}
in order to prove Lemma \ref{LocBELm}, which in turn proves Theorem \ref{LocBEThm}. In order to simplify notation, we let
\[
S_t = A^{-1} \mathcal{L}_{C_{m,n}^{-1} F_{t,0},2^*} \,.
\]
We prove ($\ref{EVCdn}$) by proving the following
\begin{thm}\label{OpPropThm}
$C_{m,n}^2 I - S_t: \dot{H}_\mathbb{C}^1 \to \dot{H}_\mathbb{C}^1$ has a nonnegative, bounded, discrete spectrum, whose eigenvalues are independent of the value of the parameter $t$. This spectrum has at most one accumulation point, which if it exists, is at $C_{m,n}^2$. 
Let $\lambda_i$, $i = 0,1,2,\dots$, (with this list possibly finite) denote the eigenvalues of $C_{m,n}^2 I - S_t: \dot{H}_\mathbb{C}^1 \to \dot{H}_\mathbb{C}^1$ whose value are less than $C_{m,n}^2$, with $\lambda_i$ listed 
in increasing order. Then, $\lambda_0 = 0$ and its corresponding eigenspace is spanned by $\{ (F_{t,0},0), (\frac{\mathrm{d}}{\mathrm{d}t} F_{t,0},0), (0,F_{t,0}) \}$. Finally, 
$\{ (F_{t,0},0), (\frac{\mathrm{d}}{\mathrm{d}t} F_{t,0},0), (0,F_{t,0}) \} \perp_{\dot{H}^1} \psi$.
\end{thm}
\noindent We split the proof of Theorem \ref{OpPropThm} into the proof of two smaller theorems and a brief argument establishing independence of eigenvalues from the value of the parameter $t$. We state these theorems below
\begin{thm}\label{StSACThm}
$S_t: \dot{H}_\mathbb{C}^1 \to \dot{H}_\mathbb{C}^1$ is a self-adjoint compact operator.
\end{thm}
\begin{thm}\label{NSThm}
$C_{m,n}^2 I - S_t: \dot{H}_\mathbb{C}^1 \to \dot{H}_\mathbb{C}^1$ is a positive operator, its nullspace is spanned by $\{ (F_{t,0}, 0), (\frac{\mathrm{d}}{\mathrm{d}t} F_{t,0}, 0), (0, F_{t,0}) \}$, 
and $\psi \perp_{\dot{H}_\mathbb{C}^1} \{ (F_{t,0}, 0), (\frac{\mathrm{d}}{\mathrm{d}t} F_{t,0}, 0), (0, F_{t,0}) \}$.
\end{thm}
\noindent Once we have proved theorems \ref{StSACThm} and \ref{NSThm}, all of Theorem \ref{OpPropThm}, except for the independence of eigenvalues from the value of $t$, follows via Fredholm Theory. The proofs of theorems \ref{StSACThm} and \ref{NSThm} are somewhat difficult and are presented in sections four and five respectively. We prove the independence of eigenvalues of $C_{m,n}^2 I - S_t: \dot{H}_\mathbb{C}^1 \to \dot{H}_\mathbb{C}^1$ from the value of $t$ here; a change of coordinates makes this proof more readily apparent. We obtain the appropriate coordinate system through several changes. First we change to $(w, \theta, \zeta)$-coordinates, $(w, \theta, \zeta) \in [0,\infty) \times [0,\pi/2] \times \mathbb{S}^{n-1}$, where
\begin{equation}\label{wtzCoord}
\varphi (w, \theta, \zeta) = \varphi (\rho,x) \text{, for } \rho = w \cos \theta, x = (w \sin\theta, \zeta) \,.
\end{equation}
And then, we change to $(u, \theta, \zeta)$-coordinates, $(u, \theta, \zeta) \in \mathbb{R} \times [0,\pi/2] \times \mathbb{S}^{n-1}$, given by
\begin{equation}\label{utzCoord}
\varphi (u, \theta, \zeta) = w^\gamma \varphi (w, \theta, \zeta) \text{, for } u = \ln w \text{ and } \gamma \text{ given by (\ref{2*aGDef})}\,.
\end{equation}
In $(u, \theta, \zeta)$-coordinates
\begin{equation}\label{FinutzCoord}
F_{t,0} (u, \theta, \zeta) = k_0 2^{-\gamma} \cosh^{-\gamma} (u + \ln t) \,.
\end{equation}
Thus, $F_{t,0}$ and $F_{t',0}$ are related by a translation of in $u$-coordinates by $\ln t' - \ln t$. This fact combined with the explicit formula of $\mathcal{L}$ as per Theorem \ref{TayThm} and that $C_{m,n}^2 I - S_t = C_{m,n}^2 I - A^{-1} \mathcal{L}_{C_{m,n}^{-1} F_{t,0},0}$ allows us to conclude that the eigenvalues of $C_{m,n}^2 I - S_t: \dot{H}_\mathbb{C}^1 \to \dot{H}_\mathbb{C}^1$ are independent of $t$. Combining this with Theorem \ref{OpPropThm}, we conclude (\ref{EVCdn}) with $\alpha_{m,n} = \lambda_1$, which is the definition of $\alpha_{m,n}$ as per Theorem \ref{LocBEThm}. Thus, we have shown Lemma \ref{LocBELm}, which in turn proves Theorem \ref{LocBEThm}.

\section{$S_t: \dot{H}_\mathbb{C}^1 \to \dot{H}_\mathbb{C}^1$ is a Self-Adjoint, Compact Operator}

In this section, we prove Theorem \ref{StSACThm}. We begin with the following
\begin{lm}\label{StSALm}
$S_t: \dot{H}_\mathbb{C}^1 \to \dot{H}_\mathbb{C}^1$ is self-adjoint.
\end{lm}

\begin{proof}
Let $\varphi_1, \varphi_2 \in \dot{H}_\mathbb{C}^1$. Then,
\[
\langle \varphi_1, S_t \varphi_2 \rangle_{\dot{H}_\mathbb{C}^1} = \left\langle \varphi_1, A^{-1} \mathcal{L}_{C_{m,n}^{-1} F_{t,0},2^*} \varphi_2 \right\rangle_{\dot{H}_\mathbb{C}^1} \nonumber = \left\langle \varphi_1, \mathcal{L}_{C_{m,n}^{-1} F_{t,0},2^*} \varphi_2 \right\rangle_{L^2 \oplus L^2} \,.
\]
It is clear from the explicit form of $\mathcal{L}_{C_{m,n}^{-1} F_{t,0},2^*}$ provided in Theorem \ref{TayThm} that $\mathcal{L}_{C_{m,n}^{-1} F_{t,0},2^*}: L^2 \oplus L^2 \to L^2 \oplus L^2$ is self-adjoint. Thus, $S_t: \dot{H}_\mathbb{C}^1 \to \dot{H}_\mathbb{C}^1$ is self-adjoint.
\end{proof}

Next, we prove the following
\begin{lm}\label{StCLm}
$S_t: \dot{H}_\mathbb{C}^1 \to \dot{H}_\mathbb{C}^1$ is compact.
\end{lm}

\begin{proof}
This proof is quite involved. We use this paragraph to outline the proof and then carry out the proof in mini-sections headed by phrases in italics. First, we reduce proving compactness of $S_t: \dot{H}_\mathbb{C}^1 \to \dot{H}_\mathbb{C}^1$ to proving compactness of $A^{-1} \mathcal{L}_{C_{m,n}^{-1} F_{t,0},2^*}^{Im}: \dot{H}^1 \to \dot{H}^1$ (we will omit the subscripts $C_{m,n}^{-1} F_{t,0}, 2^*$ henceforth). Next, we use the fact that $A^{1/2}: \dot{H}^1 \to L^2$, the square root of $A$, is an isometry to reduce proving compactness of $A^{-1} \mathcal{L}^{Im}: \dot{H}^1 \to \dot{H}^1$ to proving compactness of $A^{-1/2} \mathcal{L}^{Im} A^{-1/2}: L^2 \to L^2$. This follows due to commutativity of the following diagram
\begin{equation}\label{CommDiag}
\begin{array}[c]{ccc}
\dot{H}^1&\stackrel{A^{-1} \mathcal{L}^{Im}}{\rightarrow}&\dot{H}^1\\
\downarrow\scriptstyle{A^{1/2}}&&\uparrow\scriptstyle{A^{-1/2}}\\
L^2&\stackrel{A^{-1/2} \mathcal{L}^{Im} A^{-1/2}}{\rightarrow}&L^2 \,.
\end{array}
\end{equation}
This reduction is crucial, because it reduces the proof of compactness over $\dot{H}^1$ to $L^2$, where verification of compactness is much easier to do directly. Next, we change coordinates and reduce showing compactness of $A^{-1/2} \mathcal{L}^{Im} A^{-1/2}: L^2 \to L^2$ to a closely related operator, $(\mathcal{L}^{Im})^{1/2} \hat{A}^{-1} (\mathcal{L}^{Im})^{1/2}: L^2 \to L^2$. The change of coordinates preceding this reduction is also crucial, because it helps illuminate the route that we take to verify compactness of $A^{-1/2} \mathcal{L}^{Im} A^{-1/2}: L^2 \to L^2$. At this point, we have reduced  the compactness problem to a more manageable situation. We proceed by endeavoring to show that $(\mathcal{L}^{Im})^{1/2} \hat{A}^{-1} (\mathcal{L}^{Im})^{1/2}: L^2 \to L^2$ has arbitrarily good finite rank approximation in the operator norm. In particular, we calculate the Green's function of $\hat{A}$ and use this calculation to show that the trace of $[ (\mathcal{L}^{Im})^{1/2} \hat{A}^{-1} (\mathcal{L}^{Im})^{1/2} ]^d$ is finite for some suitably large even value of $d$, i.e. $[ (\mathcal{L}^{Im})^{1/2} \hat{A}^{-1} (\mathcal{L}^{Im})^{1/2} ]^d$ is trace class. Since $[ (\mathcal{L}^{Im})^{1/2} \hat{A}^{-1} (\mathcal{L}^{Im})^{1/2} ]^d: L^2 \to L^2$ is trace class, it is compact and its eigenvalues converge to zero. By standard spectral theory, the eigenvalues of $[ (\mathcal{L}^{Im})^{1/2} \hat{A}^{-1} (\mathcal{L}^{Im})^{1/2} ]^d$ are the $d$-th power of the eigenvalues of $(\mathcal{L}^{Im})^{1/2} \hat{A}^{-1} (\mathcal{L}^{Im})^{1/2}$. Thus, the eigenvalues of $(\mathcal{L}^{Im})^{1/2} \hat{A}^{-1} (\mathcal{L}^{Im})^{1/2}$ also converge to zero, which implies that $(\mathcal{L}^{Im})^{1/2} \hat{A}^{-1} (\mathcal{L}^{Im})^{1/2}: L^2 \to L^2$ has arbitrarily good finite rank approximation in the operator norm. Thus, to show that $(\mathcal{L}^{Im})^{1/2} \hat{A}^{-1} (\mathcal{L}^{Im})^{1/2}: L^2 \to L^2$ is compact, we only need to show that $[ (\mathcal{L}^{Im})^{1/2} \hat{A}^{-1} (\mathcal{L}^{Im})^{1/2} ]^d$ is trace class. One should note that in the course of the proof, we show that $d > \frac{n+1}{2}$ is sufficient, and it does not appear that $[ (\mathcal{L}^{Im})^{1/2} \hat{A}^{-1} (\mathcal{L}^{Im})^{1/2} ]^d$ is necessarily trace class for smaller $d$.

\textit{\underline{Mini-Section 1:} Compactness of $S_t: \dot{H}_\mathbb{C}^1 \to \dot{H}_\mathbb{C}^1$ is implied by compactness of $A^{-1} \mathcal{L}^{Im}: \dot{H}^1 \to \dot{H}^1$.} Note that
\[
S_t = A^{-1} \mathcal{L} = A^{-1} \mathcal{L}^{Re} \oplus A^{-1} \mathcal{L}^{Im} \,,
\]
as per Theorem \ref{TayThm}. Thus, to show that $S_t: \dot{H}_\mathbb{C}^1 \to \dot{H}_\mathbb{C}^1$ is compact, it suffices to show that 
$A^{-1} \mathcal{L}^{Re}, A^{-1} \mathcal{L}^{Im}: \dot{H}^1 \to \dot{H}^1$ are compact. Also,
\begin{align}\label{LRdn}
A^{-1} \mathcal{L}^{Re} \xi &= A^{-1} \left[ - (2^*-2) C_{m,n}^{2-2 \cdot 2^*} \left( \int F_{t,0}^{2^*-1} \xi \right) F_{t,0}^{2^*-1} \right] + A^{-1} \left[ C_{m,n}^{2-2^*} F_{t,0}^{2^*-2} (2^*-1) \xi \right] \nonumber \\
&=: P \xi + (2^*-1) A^{-1} \mathcal{L}^{Im} \xi \,.
\end{align}
Note that the calculations used to obtain (\ref{FOCdn}) could also be used to show that
\[
\left\langle F_{t,0}^{2^*-1}, \xi \right\rangle_{L^2} = \| F_{t,0} \|_{2^*}^{2^*} \left\langle A F_{t,0}, \xi \right\rangle_{L^2}, \forall \xi \in \dot{H}^1 \,.
\]
Since $F$ is of class $\mathcal{C}^\infty$, this implies that
\begin{equation}\label{ELEqn}
F_{t,0}^{2^*-1} = C_{m,n}^{2^*} A F_{t,0} \text{, because $\| F \|_{2^*} = C_{m,n}$} \,.
\end{equation}
This in turn implies that
\begin{equation}\label{PEqn}
P \xi = -(2^*-2) C_{m,n}^{2^*} \left\langle F_{t,0}, \xi \right\rangle_{\dot{H}^1} F_{t,0} \,.
\end{equation}
Thus, $P: \dot{H}^1 \to \dot{H}^1$ is a projection operator onto $F_{t,0}$, which implies that $P: \dot{H}^1 \to \dot{H}^1$ is compact. Combining this fact with (\ref{LRdn}), we only need 
to show that $A^{-1} \mathcal{L}^{Im}: \dot{H}^1 \to \dot{H}^1$ is compact in order to conclude that $S_t: \dot{H}_\mathbb{C}^1 \to \dot{H}_\mathbb{C}^1$ is compact.

\textit{\underline{Mini-Section 2:} Compactness of $A^{-1} \mathcal{L}^{Im}: \dot{H}^1 \to \dot{H}^1$ is implied by compactness of $(\mathcal{L}^{Im})^{1/2} \hat{A}^{-1} (\mathcal{L}^{Im})^{1/2}: \dot{H}^1 \to \dot{H}^1$.} We already explained in the above via the commutative diagram, (\ref{CommDiag}), that $A^{-1} \mathcal{L}^{Im}: \dot{H}^1 \to \dot{H}^1$ is compact if and only if $A^{-1/2} \mathcal{L}^{Im} A^{-1/2}: L^2 \to L^2$ is compact. This reduces the problem of proving compactness of an operator from $\dot{H}^1$ to $\dot{H}^1$ to an operator from $L^2$ to $L^2$, making direct verification of compactness easier. However, the explicit form of $A^{-1/2} \mathcal{L}^{Im} A^{-1/2}$ does not seem to suggest any easy way to verify the desired compactness. But, changing from $(\rho,x)$-coordinates to $(u, v, \zeta)$-coordinates, $(u, v, \zeta) \in \mathbb{R} \times [-1,1] \times \mathbb{S}^{n-1}$, provides a set of coordinates for which this verification is easier. We obtain $(u,v,\zeta)$-coordinates by making a change of variables from $(u,\theta,\zeta)$-coordinates (see (\ref{wtzCoord}) and (\ref{utzCoord}) for reference) with respect to the angular coordinate, $\theta$, given by
\[
v = 2 \cos^2 \theta - 1 \,.
\]
In these coordinates, $A$ has the explicit form ($\gamma$ is as defined in (\ref{2*aGDef}))
\begin{equation}\label{ACoordChg}
A = \gamma^2 I - \frac{\partial^2}{\partial u^2} - 4(1-v^2) \frac{\partial^2}{\partial v^2} - 4 \left( \frac{m-n}{2} - \frac{m+n}{2} v \right) \frac{\partial}{\partial v} - \frac{2}{v+1} \Delta_{\mathbb{S}^{n-1} (\zeta)} \,.
\end{equation}
(\ref{ACoordChg}) is almost a nice formula of $A$ for which we can write the Green's function of $A$ and use this to prove the desired compactness. However, the last term in the right hand side of 
(\ref{ACoordChg}) is nonlinear and makes figuring out the Green's function of $A$ difficult. Thus, we use the closely related operator $\hat{A}$ given by
\begin{equation}\label{HatA}
\hat{A} = \gamma^2 I - \frac{\partial^2}{\partial u^2} - 4 (1-v^2) \frac{\partial^2}{\partial v^2} - 4 \left( \frac{m-n}{2} - \frac{m+n}{2} v \right) \frac{\partial}{\partial v} - \Delta_{\mathbb{S}^{n-1} (\zeta)} \,,
\end{equation}
to help us show the desired compactness. More precisely, we show that $A^{-1/2} \mathcal{L}^{Im} A^{-1/2}: L^2 \to L^2$ is compact by showing that $(\mathcal{L}^{Im})^{1/2} \hat{A}^{-1} (\mathcal{L}^{Im})^{1/2}: L^2 \to L^2$ is compact. This last reduction is justified as follows: if $(\mathcal{L}^{Im})^{1/2} \hat{A}^{-1} (\mathcal{L}^{Im})^{1/2}: L^2 \to L^2$ is compact, then $(\mathcal{L}^{Im})^{1/2} A^{-1} (\mathcal{L}^{Im})^{1/2}: L^2 \to L^2$ is compact.  This is because
\[
A \geq \hat{A} \geq 0 \implies (\mathcal{L}^{Im})^{1/2} \hat{A}^{-1} (\mathcal{L}^{Im})^{1/2} \geq (\mathcal{L}^{Im})^{1/2} A^{-1} (\mathcal{L}^{Im})^{1/2} \geq 0 \,.
\]
Next, let
\[
B := A^{-1/2} (\mathcal{L}^{Im})^{1/2} \,.
\]
If $(\mathcal{L}^{Im})^{1/2} A^{-1} (\mathcal{L}^{Im})^{1/2}: L^2 \to L^2$ is compact, then
\begin{eqnarray}
&& B^* B = (\mathcal{L}^{Im})^{1/2} A^{-1} (\mathcal{L}^{Im})^{1/2}: L^2 \to L^2 \text{ is compact} \nonumber \\
&\implies& B,B^*: L^2 \to L^2 \text{ are bounded} \nonumber \\
&\implies& B \mathcal{L}^{Im} A^{-1} (\mathcal{L}^{Im})^{1/2} B^* = (A^{-1/2} \mathcal{L}^{Im} A^{-1/2})^2: L^2 \to L^2 \text{ is compact} \nonumber \\
&\implies& A^{-1/2} \mathcal{L}^{Im} A^{-1/2}: L^2 \to L^2 \text{ is compact} \,. \nonumber
\end{eqnarray}
Thus, compactness of $(\mathcal{L}^{Im})^{1/2} \hat{A}^{-1} (\mathcal{L}^{Im})^{1/2}: L^2 \to L^2$ implies compactness of $A^{-1/2} \mathcal{L}^{Im} A^{-1/2}: L^2 \to L^2$.

\textit{\underline{Mini-Section 3:} Calculating the Green's function of $\hat{A}$.} We can write $\hat{A}$ as
\[
\hat{A} = U + V + W \,,
\]
where
\begin{eqnarray}
U &=& \gamma^2 I -\frac{\partial^2}{\partial u^2} \nonumber \\
V &=& -4(1-v^2)\frac{\partial^2}{\partial v^2} - 4[\hat{\alpha} - \hat{\beta} - (\hat{\alpha}+\hat{\beta}+2)v] \frac{\partial}{\partial v} \text{, and} \nonumber \\
W &=& - \Delta_{\mathbb{S}^{n-1} (\zeta)} \,, \nonumber
\end{eqnarray}
for $\hat{\alpha} = \frac{m-2}{2}$ and $\hat{\beta} = \frac{n-2}{2}$. We can build the Green's function for $\hat{A}$ out of the eigenfunctions and eigenvalues of $U$, $V$, and $W$. The eigenfunctions of $U$ are $q_k (u) := e^{-iku}$, with corresponding eigenvalues $\gamma^2 + k^2$. The eigenfunctions of $V$ are the Jacobi Polynomials, $p_j^{\hat{\alpha}, \hat{\beta}} (v)$, with corresponding eigenvalues $\sigma_j = 4j(j+\frac{m+n}{2}-1)$; for more detail, see p. 60 of [21]. The eigenfunctions of $W$ are the spherical harmonics of $\mathbb{S}^{n-1}$. We will let $g_l (\zeta)$ denote the spherical harmonics, arranged in such a fashion that their corresponding eigenvalues, $\tau_l$, for $W$ are nonincreasing (all of the eigenvalues will be nonnegative). Thus, the Green's function of $\hat{A}$ is
\begin{align} \label{GFnc}
G(u, v, \zeta| \tilde{u}, \tilde{v}, \tilde{\zeta}) &= \sum_{j, l \geq 0} \int_\mathbb{R} \frac{1}{\gamma^2 + k^2 + \tau_l + \sigma_j} q_k (u) p_j^{\hat{\alpha}, \hat{\beta}} (v) g_l (\zeta) \bar{q}_k (\tilde{u}) \bar{p}_j^{\hat{\alpha}, \hat{\beta}} (\tilde{v}) \bar{g}_l (\tilde{\zeta}) \mathrm{d}k \nonumber \\
&= \sum_{j, l \geq 0} \pi (\gamma^2 + \tau_l + \sigma_j)^{-1/2}
e^{- |u-\tilde{u}| \sqrt{\gamma^2 + \tau_l + \sigma_j}} p_j^{\hat{\alpha}, \hat{\beta}} (v) g_l (\zeta) \bar{p}_j^{\hat{\alpha},\hat{\beta}} (\tilde{v}) \bar{g}_l (\tilde{\zeta}) \,.
\end{align}
Using the Green's function above, we will show that $(\mathcal{L}^{Im})^{1/2} \hat{A}^{-1/2} (\mathcal{L}^{Im})^{1/2}: L^2 \to L^2$ is compact.

\textit{\underline{Mini-Section 4:} $(\mathcal{L}^{Im})^{1/2} \hat{A}^{-1/2} (\mathcal{L}^{Im})^{1/2}: L^2 \to L^2$ has arbitrarily good finite rank approximation in the operator norm.} We will show that $(\mathcal{L}^{Im})^{1/2} \hat{A}^{-1/2} (\mathcal{L}^{Im})^{1/2}: L^2 \to L^2$ is compact by showing that for $d > \frac{n+1}{2}$,
\begin{equation}\label{FinTrProp}
\operatorname{Tr} \left[ ((\mathcal{L}^{Im})^{1/2} \hat{A}^{-1} (\mathcal{L}^{Im})^{1/2})^d \right] < \infty \,,
\end{equation}
i.e. $((\mathcal{L}^{Im})^{1/2} \hat{A}^{-1} (\mathcal{L}^{Im})^{1/2})^d$ is trace class. At the end of the first paragraph of the proof of Lemma \ref{StCLm}, we showed that if we prove (\ref{FinTrProp}), then we can conclude that $(\mathcal{L}^{Im})^{1/2} \hat{A}^{-1/2} (\mathcal{L}^{Im})^{1/2}: L^2 \to L^2$ has arbitrarily good finite rank approximation in the operator norm, and so is compact. Thus, to show that $(\mathcal{L}^{Im})^{1/2} \hat{A}^{-1} (\mathcal{L}^{Im})^{1/2}: L^2 \to L^2$ is compact, we only need to prove (\ref{FinTrProp}).

In order to prove (\ref{FinTrProp}), we derive the kernel for $((\mathcal{L}^{Im})^{1/2} \hat{A}^{-1} (\mathcal{L}^{Im})^{1/2})^d$. Using (\ref{GFnc}), we calculate that the kernel for 
$(\mathcal{L}^{Im})^{1/2} \hat{A}^{-1} (\mathcal{L}^{Im})^{1/2}$ is
\[
K(u,v,\zeta|\tilde{u},\tilde{v},\tilde{\zeta}) = \hat{F}(u) \hat{F}(\tilde{u}) \sum_{j,l} \frac{\pi}{\xi_{j,l}} e^{-|u-v|\xi_{j,l}} p_j^{\alpha, \beta} (v) \bar{p}_j^{\alpha, \beta} (\tilde{v}) g_l
(\zeta) \bar{g}_l (\tilde{\zeta}) \,,
\]
where $\hat{F} = (C_{m,n}^{-1} F_{t,0})^{(2^*-2)/2}$ and $\xi_{j,l} = \sqrt{\gamma^2 + \tau_l + \sigma_j}$. Using this, we calculate the kernel, $K_d$, for $((\mathcal{L}^{Im})^{1/2} \hat{A}^{-1} (\mathcal{L}^{Im})^{1/2})^d$.  Before doing so, we make the following convention to simplify notation:
\[
\int \cdot \text{ } \mathrm{d} \Lambda (u,v,\zeta)
\]
denotes the integral over $\mathbb{R} \times [-1,1] \times \mathbb{S}^{n-1}$ with $\mathrm{d} \Lambda (u,v,\zeta)$ representing $\mathrm{d} \Lambda$ as defined in (\ref{dLDef}) corresponding to the change of coordinates from $(\rho,x)$ to $(u,v,\zeta)$. Thus,
\begin{eqnarray}
&& K_{d} (u_1, v_1, \zeta_1 | u_{d+1}, v_{d+1}, \zeta_{d+1}) \nonumber \\
&=& \int \int \cdots \int \prod\limits_{i=1}^d K(u_i, v_i, \zeta_i | u_{i+1}, v_{i+1},
\zeta_{i+1})
\mathrm{d}\Lambda(u_2,v_2,\zeta_2) \mathrm{d}\Lambda(u_3,v_3,\zeta_3) \cdots
\mathrm{d}\Lambda(u_{d},v_{d},\zeta_{d}) \nonumber \\
&=& \hat{F}(u_1) \hat{F}(u_{d+1}) \sum_{j,l \geq 0} \frac{\pi^d}{\xi_{j,l}^d}
p_j^{\hat{\alpha}, \hat{\beta}}
(v_1) \bar{p}_j^{\hat{\alpha}, \hat{\beta}} (v_{d+1}) g_l (\zeta_1) \bar{g}
(\zeta_{d+1}) \nonumber \\
&& \int \int \cdots \int \prod\limits_{i=2}^{d} \hat{F}^2 (u_i) |p_j^{\hat{\alpha}, \hat{\beta}}
(v_i)|^2 |g_l (\zeta_i)|^2 
\prod\limits_{a=1}^{d} e^{-\xi_{j,l} |u_{a+1} - u_a|} \mathrm{d}\Lambda(u_2, v_2,
\zeta_2) \mathrm{d}\Lambda(u_3, v_3, \zeta_3) \cdots \mathrm{d}\Lambda(u_{d}, v_{d},
\zeta_{d})
\nonumber \\
&=& \hat{F}(u_1) \hat{F}(u_{d+1}) \sum_{j,l \geq 0} \frac{\pi^d}{\xi_{j,l}^d}
p_j^{\hat{\alpha}, \hat{\beta}}
(v_1) \bar{p}_j^{\hat{\alpha}, \hat{\beta}} (v_{d+1}) g_l (\zeta_1) \bar{g}
(\zeta_{d+1}) \int_{\mathbb{R}^{d-1}} \prod\limits_{i=2}^{d} \hat{F}^2 (u_i)
\prod\limits_{a=1}^{d} e^{-\xi_{j,l}|u_{a+1}-u_a|} \mathrm{d}u_2
\mathrm{d}u_3 \cdots \mathrm{d}u_{d} \,. \nonumber
\end{eqnarray}
Thus,
\begin{align}\label{TrSum}
\operatorname{Tr} \left[ ((\mathcal{L}^{Im})^{1/2} \hat{A}^{-1} (\mathcal{L}^{Im})^{1/2})^d \right] =& \left\| K_d (u_1, v_1, \zeta_1 | u_1, v_1, \zeta_1) \right\|_1 \nonumber \\
=& \int \int_{\mathbb{R}^{d-1}} \sum_{j,l \geq 0}
\frac{\pi^d}{\xi_{j,l}^d}
|p_j^{\hat{\alpha}, \hat{\beta}} (v_1)|^2 |g_l (\zeta_1)|^2 \hat{F}^2 (u_1)
\nonumber \\
& \prod\limits_{i = 1}^d \hat{F}^2 (u_i) \left( \prod\limits_{a=1}^{d-1} e^{-\xi_{j,l}|u_{a+1}-u_1|} \right) e^{-\xi_{j,l}|u_1-u_d|} \mathrm{d}u_2 \mathrm{d}u_3 \cdots \mathrm{d}u_d \mathrm{d}\Lambda(u_1, v_1, \zeta_1) \nonumber \\
=& \sum_{j,l \geq 0} \frac{\pi^d}{\xi_{j,l}^d} \int_{\mathbb{R}^d} \left(
\prod\limits_{i=1}^d \hat{F}^2 (u_i) \right) e^{-\xi_{j,l} |u_1 - u_d|}
\prod\limits_{a=1}^{d-1} e^{-\xi_{j,l} |u_{a+1} - u_a|} \mathrm{d}u_1
\mathrm{d}u_2 \cdots \mathrm{d}u_d \,.
\end{align}
We can apply the Generalized Young's Inequality to each of the
integrals in the right hand side of (\ref{TrSum}).  The version of the Generalized Young's Inequality
we use in this setting is

\noindent \textbf{\underline{Generalized Young's Inequality:}} \textit{Let $R_i$ be a real $1 \times d$
matrix and $h_i: \mathbb{R} \to \mathbb{R}$
be a function for $i = 1, 2, \dots, 2d$.  Also, let $p_1, p_2, \dots, p_{2d}$ be
such that
\[
\sum_{i=1}^{2d} \frac{1}{p_i} = d \,.
\]
Then,
\[
\int_{\mathbb{R}^d} \prod\limits_{i=1}^{2d} h_i (R_i \cdot \vec{u}) \mathrm{d}u_1 \mathrm{d}u_2 \cdots \mathrm{d}u_d \leq \hat{C}_{d, p_1, p_2, \dots, p_{2d}} \prod\limits_{i=1}^{2d} \| h_i \|_{L^{p_i} ( \mathbb{R} )} \,,
\]
for some finite constant $\hat{C}_{d,p_1,p_2,\dots,p_{2d}}$ provided for any $J
\subseteq \{ 1,2,\dots,2d \}$ such that $\operatorname{card} (J) \leq d$ has the property that
\begin{equation}\label{DimSpnCdn}
\operatorname{dim} ( \operatorname{span}_{i \in J} \{ R_i \} ) \geq \sum_{i \in J} \frac{1}{p_i} \,.
\end{equation}}
In order to apply the Generalized Young's Inequality to (\ref{TrSum}), we prove
\begin{lm}
The expression on the right hand side of (\ref{TrSum}) satisfies the conditions of the Generalized Young's Inequality as stated above.  In particular, (\ref{DimSpnCdn}) is satisfied.
\end{lm}

\begin{proof}
In each integral in the sum of the right hand side of (\ref{TrSum}), we can take
\begin{eqnarray}
p_i = 2 && \text{for } i = 1, 2, \ldots, 2d; \nonumber \\
R_i = \hat{e}_i && \text{for } i = 1, \dots, d; \nonumber \\
h_i (u) = \hat{F}^2 (u) && \text{for } i = 1, \dots, d; \nonumber \\
R_i = \hat{e}_{i - d + 1} - \hat{e}_{i - d} && \text{for } i = d+1, \dots, 2d-1;
\nonumber \\
R_{2d} = \hat{e}_1 - \hat{e}_d && \text{and} \nonumber \\
h_i (u) = e^{-\xi_{j,l} |u|} && \text{for } i = d+1, \dots, 2d \,. \nonumber
\end{eqnarray}
Since $p_i = 2$, (\ref{DimSpnCdn}) reduces to showing that
\begin{equation}\label{DimSpnCdn1}
\operatorname{dim} ( \operatorname{span}_{i \in J} \{ R_i \} ) \geq \frac{\operatorname{card} (J)}{2} \,.
\end{equation}
If at least $\operatorname{Card} (J)/2$ of the elements in $J$ are a subset of
$\{ 1, \dots ,d \}$, then the corresponding $R_i = \hat{e}_1$ and (\ref{DimSpnCdn1}) is
satisfied.  Next, we observe that any proper subset of
\[
\{ R_{d+1} = \hat{e}_2 -
\hat{e}_1, R_{d+2} = \hat{e}_3 - \hat{e}_2, \dots , R_{2d-1} = \hat{e}_d =
\hat{e}_{d-1}, R_{2d} = \hat{e}_1 - \hat{e}_d \}
\]
is linearly independent. Thus, in the case where at least $\operatorname{card} (J) / 2$ of the elements in $J$ are a subset of $\{ d+1, \dots ,2d \}$, the corresponding $R_i$ will be linearly independent, unless $J = \{ d+1, \dots, 2d \}$, in which case their span will have $d-1$ dimensions, while $\operatorname{card} (J) = d$, where $d \geq 2$.  In both these cases, (\ref{DimSpnCdn1}) is satisfied.
\end{proof}
Applying the Generalized Young's Inequality to (\ref{TrSum}), we get that
\begin{align}\label{TrSumBd}
\operatorname{Tr} \left[ ((\mathcal{L}^{Im})^{1/2} \hat{A}^{-1} (\mathcal{L}^{Im})^{1/2})^d \right] \leq& \hat{C}_{d,2,2,\dots,2} \sum_{j,l} \frac{\pi^d}{\xi_{j,l}^d} \| \hat{F}^2 \|_{L^2 (\mathbb{R})}^d \| e^{-\xi_{j,l}} \|_{L^2 (\mathbb{R})}^d \nonumber \\
=& \hat{C}_{d,2,2,\dots,2} \pi^d \| \hat{F}^2 \|_{L^2 (\mathbb{R})}^d \sum_{j,l} (\gamma^2 + \tau_l + \sigma_j )^{-d}
\end{align}
Since the eigenvalues of the spherical harmonics on $\mathbb{S}^{n-1}$ are $\tau_l = l(l+n-2)$ with corresponding multiplicity ${n+l-1 \choose l}-{n+l-2 \choose l-1}$, see [20] for reference, and 
$\sigma_j = 4j (j + \frac{m+n}{2} - 1)$, the right hand side of (\ref{TrSumBd}) equals
\begin{align}\label{TrSumEq}
\hat{C}_d \sum_{j,l \geq 0} \frac{{n+l-1 \choose l}-{n+l-2 \choose
l-1}}{[\gamma^2 + 4j(j+\frac{m+n}{2}-1) + l(l+n-2)]^d} =& \hat{C}_d \sum\limits_{0 \leq j,l < n-1} \frac{{n+l-1 \choose l}-{n+l-2 \choose l-1}}{[\gamma^2 + 4j(j+\frac{m+n}{2}-1) + l(l+n-2)]^d} \nonumber \\
&+ \hat{C}_d \sum\limits_{j,l \geq n-1} \frac{{n+l-1 \choose l}-{n+l-2 \choose l-1}}{[\gamma^2 + 4j(j+\frac{m+n}{2}-1) + l(l+n-2)]^d} \,,
\end{align}
where $\hat{C}_d = \hat{C}_{d, 2, 2, \dots, 2} \pi^d \| \hat{F}^2 \|_{L^2
(\mathbb{R}^d)}^d$.  Now,
\begin{eqnarray}
&& {n+l-1 \choose l} - {n+l-2 \choose l-1} \nonumber \\
&=& \frac{1}{(n-1)!} [ (n+l-1) (n+l-2) \cdots (l+1) - (n+l-2) (n+l-3) \cdots l ]
\text{, which for } l \geq n-1 \nonumber \\
&\leq& (2l)^{n-1} \,. \nonumber
\end{eqnarray}
Thus, if $d > \frac{n+1}{2}$,
\begin{eqnarray}
\sum\limits_{j,l \geq n-1} \frac{{n+l-1 \choose l}-{n+l-2 \choose
l-1}}{[\gamma^2 + 4j(j+\frac{m+n}{2}-1) + l(l+n-2)]^d} &\leq& 2^{n-1} \sum_{j,l \geq n-1} \frac{l^{n-1}}{[\gamma^2 + j^2 + l^2]^d}
\nonumber \\
&\leq& 2^{n-1} \sum_{j,l \geq n-1} \frac{1}{[\gamma^2 + j^2 + l^2]^{\left( d -
\frac{n-1}{2} \right) }} \nonumber \\
&<& \infty \,. \nonumber
\end{eqnarray}
We can conclude by the above that (\ref{TrSumEq}) is finite when $d > \frac{n+1}{2}$. Hence, (\ref{FinTrProp}) holds for $d > \frac{n+1}{2}$. Thus, $(\mathcal{L}^{Im})^{1/2} \hat{A}^{-1} (\mathcal{L}^{Im})^{1/2}: L^2 \to L^2$, and ultimately $S_t: \dot{H}_\mathbb{C}^1 \to \dot{H}_\mathbb{C}^1$, is compact.
\end{proof}

Combining lemmas \ref{StSALm} and \ref{StCLm} we conclude Theorem \ref{StSACThm}.

\section{The Nullspace of $C_{m,n}^2 I - S_t: \dot{H}_\mathbb{C}^1 \to \dot{H}_\mathbb{C}^1$}

In this section, we prove Theorem \ref{NSThm}. We will use $(u, \theta, \zeta)$-coordinates (see (\ref{wtzCoord}) and (\ref{utzCoord}) for reference). To simplify notation, we will often omit the subscripts $t, 0$ from $F_{t,0}$, $t$ from $S_t$, and the subscripts $C_{m,n}^{-1} F_{t,0}, 2^*$ from the 
operators $\mathcal{L}_{C_{m,n}^{-1} F_{t,0},2^*}$, $\mathcal{L}_{C_{m,n}^{-1} F_{t,0},2^*}^{Re}$, and $\mathcal{L}_{C_{m,n}^{-1} F_{t,0},2^*}^{Im}$. Also to simplify notation, $\operatorname{Null} (C_{m,n}^2 I - S)$ will denote the nullspace of $C_{m,n}^2 I - S: \dot{H}_\mathbb{C}^1 \to \dot{H}_\mathbb{C}^1$.

We prove Theorem \ref{NSThm} in a series of three lemmas. In the first lemma, we show that $\{ (F,0), (\frac{\mathrm{d}}{\mathrm{d}t} F, 0), (0,F)\} \subseteq \operatorname{Null} (C_{m,n}^2 I - S)$ and $C_{m,n}^2 I - S: \dot{H}_\mathbb{C}^1 \to \dot{H}_\mathbb{C}^1$ is positive. In the second lemma, we show that no element of $\operatorname{Null} (C_{m,n}^2 I - S)$ is linearly independent of $\{ (F,0), (\frac{\mathrm{d}}{\mathrm{d}t} F,0), (0,F) \}$. Proving this lemma is a bit tricky. The proof breaks into four steps, each of which is headed by a phrase in italics. In the first step, we reduce the proof of the lemma to showing that the space of zeroes of $C_{m,n}^2 A - \mathcal{L}^{Re}$ in $\{ F \}^{\perp_{\dot{H}^1}}$, functions in $\dot{H}^1$ perpendicular to $F$, is spanned by $\frac{\mathrm{d}}{\mathrm{d}t} F$. In the second step, we show that the zeroes of $C_{m,n}^2 A - \mathcal{L}^{Re}$ in $\{ F \}^{\perp_{\dot{H}^1}}$ that are independent of $\theta$ and $\zeta$ are constant multiples of $\frac{\mathrm{d}F}{\mathrm{d}t}$. The proof of this fact boils down to showing that any zero of $C_{m,n}^2 A - \mathcal{L}^{Re}$ that is linearly independent of $\frac{\mathrm{d}F}{\mathrm{d}t}$ would have infinite energy. The precise proof of this fact is clever, and perhaps the most interesting proof in this section. In the third and fourth steps, we show that zeroes of $C_{m,n}^2 A - \mathcal{L}^{Re}$ in $\{ F \}^{\perp_{\dot{H}^1}}$ are independent of $\theta$ and $\zeta$. We conclude the section with the third lemma, in which we show that $\psi \perp_{\dot{H}_\mathbb{C}^1} \operatorname{Null}(C_{m,n}^2 I - S)$.

We begin with the following
\begin{lm}\label{OpPLm}
$\{ (F,0), (\frac{\mathrm{d}}{\mathrm{d}t} F,0), (0,F) \} \subseteq \operatorname{Null} (C_{m,n}^2 I - S)$ and $C_{m,n}^2 I - S: \dot{H}_\mathbb{C}^1 \to \dot{H}_\mathbb{C}^1$ is a positive operator.
\end{lm}

\begin{proof}
We begin by observing that
\begin{align}\label{FZMEqn}
0 &= \frac{\mathrm{d}^2}{\mathrm{d} \varepsilon^2} \big|_{\varepsilon = 0} \frac{\| F + \varepsilon F \|_{2^*}^2}{\| F + \varepsilon F \|_{\dot{H}^1}^2} \nonumber \\
&= \frac{\mathrm{d}^2}{\mathrm{d} \varepsilon^2} \big|_{\varepsilon = 0} \| F + \varepsilon F \|_{2^*}^2 - 2 C_{m,n}^2 \nonumber \nonumber \\
&= 2 \left( \left\langle \mathcal{L} F, F \right\rangle_{L^2 \oplus L^2} - C_{m,n}^2 \right) \nonumber \\
&= 2 \left\langle (S - C_{m,n}^2 I) F, F \right\rangle_{\dot{H}_\mathbb{C}^1} \,.
\end{align}
Thus, $(F,0) \in \operatorname{Null} (C_{m,n}^2 I - S)$.  A similar calculation shows that $(0,F) \in \operatorname{Null} (C_{m,n}^2 I - S)$. In a similar manner, differentiating $\| F_{t + \varepsilon,0} \|_{2^*}^2 / \| F_{t + \varepsilon,0} \|_{\dot{H}^1}^2$ to the second order and evaluating at $\varepsilon = 0$ shows that $(\frac{\mathrm{d}}{\mathrm{d}t} F, 0) \in \operatorname{Null} (C_{m,n}^2 I - S)$.

Let $\tilde{\psi} \in \dot{H}_\mathbb{C}^1$ be such that $\| \tilde{\psi} \|_{\dot{H}^1} = 1$ and $\tilde{\psi} \perp_{\dot{H}_\mathbb{C}^1} F$, and $\varepsilon \in [-1,1]$, then
\begin{align}\label{OpAP}
0 \leq& C_{m,n}^2 \| F + \varepsilon \tilde{\psi} \|_{\dot{H}^1}^2 - \| F + \varepsilon \tilde{\psi} \|_{2^*}^2 \text{, which by Theorem \ref{TayThm} and the fact that $\tilde{\psi} \perp_{\dot{H}_\mathbb{C}^1} F$} \nonumber \\
\leq& C_{m,n}^2 \left( \| F \|_{\dot{H}^1}^2 + \varepsilon^2 \| \tilde{\psi} \|_{\dot{H}^1}^2 \right) - \left( \| F \|_{2^*}^2 + \langle S \tilde{\psi}, \tilde{\psi} \rangle_{\dot{H}_\mathbb{C}} \varepsilon^2 - 
\frac{\kappa_{2^*} C_{m,n}^2}{4 \cdot 3^{\frac{\beta_{2^*}}{2} - 1}} \varepsilon^{\beta_{2^*}} \right) \nonumber \\
=& \left\langle (C_{m,n}^2 I - S) \tilde{\psi}, \tilde{\psi} \right\rangle_{\dot{H}_\mathbb{C}^1} \varepsilon^2 + \frac{\kappa_{2^*} C_{m,n}^2}{4 \cdot 2^{\frac{\beta_{2^*}}{2} - 1}} |\varepsilon|^{\beta_{2^*}} \nonumber \\
\implies& \left\langle (C_{m,n}^2 I - S) \tilde{\psi}, \tilde{\psi} \right\rangle_{\dot{H}_\mathbb{C}^1} \geq - \frac{\kappa_{2^*} C_{m,n}^2}{4 \cdot 2^{\frac{\beta_{2^*}}{2} - 1}} |\varepsilon|^{\beta_{2^*} - 2} \text{, $\forall \varepsilon \in [-1,1]$. And, since $\beta_{2^*} > 2$, this} \nonumber \\
\implies& \left\langle (C_{m,n}^2 I - S) \tilde{\psi}, \tilde{\psi} \right\rangle_{\dot{H}_\mathbb{C}^1} \geq 0 \text{, for $\tilde{\psi} \perp_{\dot{H}_\mathbb{C}^1} F$} \,.
\end{align}
Combining (\ref{OpAP}) with the fact that $(F,0) \in \operatorname{Null} (C_{m,n}^2 I - S)$ concludes the proof of Lemma \ref{OpPLm}.
\end{proof}
At this point, we have proved that $\{ (F,0), (\frac{\mathrm{d}}{\mathrm{d}t}F,0), (0,F) \} \subseteq \operatorname{Null} (C_{m,n}^2 I - S)$. So, to prove that $\{ (F,0), (\frac{\mathrm{d}}{\mathrm{d}t}F,0), (0,F) \}$ spans $\operatorname{Null} (C_{m,n}^2 I - S)$, we show the following
\begin{lm}\label{NSLm2}
No element in $\operatorname{Null} (C_{m,n}^2 I - S)$ is linearly independent of $\{ (F,0), (\frac{\mathrm{d}}{\mathrm{d}t} F, 0), (0,F) \}$
\end{lm}
\begin{proof}
\textit{\underline{Step 1:} Reduce the proof to showing that the space of zeroes of $C_{m,n}^2 A - \mathcal{L}^{Re}$ in $\{ F \}^{\perp_{\dot{H}^1}}$ is spanned by $\frac{\mathrm{d}}{\mathrm{d}t} F$.} We begin by proving
\begin{equation}\label{ImOSP}
C_{m,n}^2 I - A^{-1} \mathcal{L}^{Im} > 0 \text{ on } \{ F \}^{\perp_{\dot{H}^1}} \,,
\end{equation}
where
\[
\{ F \}^{\perp_{\dot{H}^1}} := \left\{ \xi \in \dot{H}^1 \big| \xi \perp_{\dot{H}^1} F \right\} \,.
\]
Combining (\ref{LRdn}) and (\ref{PEqn}) yields
\[
A^{-1} \mathcal{L}^{Re} = (2^*-1) A^{-1} \mathcal{L}^{Im} \text{ on } \{ F \}^{\perp_{\dot{H}^1}} \,.
\]
Thus,
\[
C_{m,n}^2 I - A^{-1} \mathcal{L}^{Im} = \frac{2^*-2}{2^*-1} C_{m,n}^2 I + \frac{1}{2^*-1} (C_{m,n}^2 I - A^{-1} \mathcal{L}^{Re}) \text{ on } \{ F \}^{\perp_{\dot{H}^1}} \,.
\]
This combined with the fact that $C_{m,n}^2 I - S: \dot{H}_\mathbb{C}^1 \to \dot{H}_\mathbb{C}^1$ is positive (and so $C_{m,n}^2 I - A^{-1} \mathcal{L}^{Re}: \dot{H}^1 \to \dot{H}^1$ is positive), allows us to conclude (\ref{ImOSP}). Thus, if $\varphi = (\xi, \eta)$ is in $\operatorname{Null} (C_{m,n}^2 I - S)$ with $\xi, \eta \in \{ F \}^{\perp_{\dot{H}^1}}$, then it is of the form $(\xi,0)$ for some $\xi \in \{ F \}^{\perp_{\dot{H}^1}}$. Note that
\begin{equation}\label{dFdtEqn}
\frac{\mathrm{d} F}{\mathrm{d}t} = \frac{1}{t} \frac{\partial F}{\partial u} \,,
\end{equation}
and integration by parts in the $u$-variable shows that $\frac{\partial F}{\partial u} \perp_{\dot{H}^1} F$. Thus, $\frac{\mathrm{d}F}{\mathrm{d}t} \in \{ F \}^{\perp_{\dot{H}^1}}$. And so, if we can show that
\begin{equation}\label{ReONS}
(C_{m,n}^2 A - \mathcal{L}^{Re}) \xi = 0 \text{ and } \xi \in \{ F \}^{\perp_{\dot{H}^1}} \implies \xi = c \frac{\mathrm{d}F}{\mathrm{d}t} \text{, for some } c \in \mathbb{R} \,,
\end{equation}
then we have proved Lemma \ref{NSLm2}. We could try to prove $(\ref{ReONS})$ by thinking of
\begin{equation}\label{DiffEqn}
(C_{m,n}^2 A - \mathcal{L}^{Re}) \xi = 0
\end{equation}
as a differential equation and trying to find all of its solutions. However, this would be tricky, as $A$ does not separate nicely. Also,  not all solutions of (\ref{DiffEqn}) are in $\dot{H}^1$. So we would need to identify which solutions of (\ref{DiffEqn}) are in $\dot{H}^1$. What we do instead is show that any solution of (\ref{DiffEqn}) dependent upon the $u$-variable only and linearly independent of $\frac{\mathrm{d}}{\mathrm{d}t} F$ must have infinite energy. And then, we show that solutions of (\ref{DiffEqn}) must be independent of the $\theta$ and $\zeta$ variables.

\textit{\underline{Step 2:} The zeroes of $C_{m,n}^2 A - \mathcal{L}^{Re}$ that are linearly independent of $\frac{\mathrm{d}}{\mathrm{d}t} F$ have infinite energy.} The zeroes of $C_{m,n}^2 A - \mathcal{L}^{Re}$ in $\{ F \}^{\perp_{\dot{H}^1}}$ are independent of $\theta$ and $\zeta$, and so are radial - we delay the proof of this fact to steps three and four. Hence, if $(C_{m,n}^2 A - \mathcal{L}^{Re}) \xi = 0$ for some $\xi \in \{ F \}^{\perp_{\dot{H}^1}}$, then
\begin{equation}\label{RadEqn}
X \xi := C_{m,n}^2 ( \gamma^2 \xi - \frac{\partial^2}{\partial u^2} \xi ) - (2^*-1) C_{m,n}^{2-2^*} F^{2^*-2} \xi = 0 \,,
\end{equation}
where $\xi$ is independent of $\theta$ and $\zeta$. Since the functions satisfying (\ref{RadEqn}) are radial, when solving (\ref{RadEqn}), we will treat $X$ as if it were a differential operator on $\mathbb{R}$ in the variable $u$. However, we will be interested in whether or not these solutions are in $\dot{H}^1$. Thus, to verify (\ref{ReONS}), we prove the following
\begin{prop}\label{FuSpanProp}
Consider (\ref{RadEqn}) as a differential equation in the variable $u$ only.  Let $\tilde{\phi}$ be a solution of (\ref{RadEqn}) that is linearly independent of $\frac{\mathrm{d}}{\mathrm{d}t} F$. Then, $\phi$ given by $\phi (u,\theta,\zeta) = \tilde{\phi} (u)$ is not in $\dot{H}^1$, because
\begin{equation}\label{NIProp}
\| \phi \|_{\dot{H}^1} = \infty \,.
\end{equation}
\end{prop}

\begin{proof}
In the following, since we are working with functions on $\mathbb{R}$ in the $u$ variable, we will take $F$ as such a function since it is independent of $\theta$ and $\zeta$. By (\ref{dFdtEqn}), it suffices to show that Proposition \ref{FuSpanProp} holds if we replace $\frac{\mathrm{d}F}{\mathrm{d} t}$ with $\frac{\partial F}{\partial u}$. Using $\frac{\partial F}{\partial u}$ instead of $\frac{\mathrm{d}F}{\mathrm{d} t}$ makes some of our calculations in this proof easier. Thus, we will prove Proposition \ref{FuSpanProp} for $\frac{\partial F}{\partial u}$ instead of $\frac{\mathrm{d}F}{\mathrm{d} t}$. We write $F_u$, $F_{uu}$, $F_{uuu}$, etc. to denote $u$ derivatives of $F$. Considering solutions of (\ref{RadEqn}), we may view linear independence with respect to $F_u$ in terms of the initial conditions $\tilde{\xi} (0)$, $\tilde{\xi}'(0)$. To be more precise, we will consider (\ref{RadEqn}) with initial conditions
\begin{equation}\label{IC}
\tilde{\xi} (0) = \tilde{\alpha}, \tilde{\xi}' (0) = \tilde{\beta} \,,
\end{equation}
where $\tilde{\alpha}$ and $\tilde{\beta}$ are constants. Since (\ref{RadEqn}) is a second order linear equation with continuous coefficients, (\ref{RadEqn}) combined with (\ref{IC}) determines a unique solution. $F_u$ is the solution of (\ref{RadEqn}) satisfying (\ref{IC}) with $\tilde{\alpha} = 0$ and $\tilde{\beta} = F_{uu} (0) \neq 0$. All solutions of (\ref{RadEqn}) that are linearly independent of $F_u$ satisfy (\ref{IC}) with some $\tilde{\alpha} \neq 0$. If such a solution, call it $\tilde{\phi}$, were to have the property that $\phi$ given by $\phi (u,\theta,\zeta) = \tilde{\phi} (u)$, is in $\dot{H}^1$, then $\tilde{\tilde{\xi}} = c_1 F_u + c_2 \tilde{\phi}$ for appropriate $c_1, c_2 \in \mathbb{R}$ would satisfy (\ref{RadEqn}) and (\ref{IC}) for $\tilde{\alpha} = 1$ and $\tilde{\beta} = 0$. Moreover, since $F_u, \tilde{\phi} \in \dot{H}^1$, $\xi$ given by $\xi (u,\theta,\zeta) = \tilde{\tilde{\xi}} (u)$, would be an element of $\dot{H}^1$. We will prove Proposition \ref{FuSpanProp} by showing that $\xi$ is not in $\dot{H}^1$.

Observe that,
\begin{align}
0 &= \int_0^{u_0} \left[(2^*-1) C_{m,n}^{-2^*} F^{2^*-2}
- \gamma^2 \right] \tilde{\tilde{\xi}} F_u - \tilde{\tilde{\xi}} \left[(2^*-1) C_{m,n}^{-2^*} F^{2^*-2} - \gamma^2 \right] F_u \mathrm{d}u \text{, since $F_u$ and $\tilde{\tilde{\xi}}$ satisfy (\ref{RadEqn})} \nonumber \\
&= \int_0^{u_0} - \tilde{\tilde{\xi}}'' F_u + \tilde{\tilde{\xi}} F_{uuu} \mathrm{d}u \nonumber \\
&= -\tilde{\tilde{\xi}}' F_u |_0^{u_0} + \tilde{\tilde{\xi}} F_{uu}|_0^{u_0} - \int_0^{u_0} -\tilde{\tilde{\xi}}' F_{uu} + \tilde{\tilde{\xi}}' F_{uu}
\mathrm{d}u
\nonumber \\
&= -\tilde{\tilde{\xi}}' (u_0) F_u (u_0) + \tilde{\tilde{\xi}}' (0) F_u (0) + \tilde{\tilde{\xi}} (u_0) F_{uu} (u_0) - \tilde{\tilde{\xi}} (0) F_{uu} (0) \nonumber \,.
\end{align}
Recall that $\tilde{\tilde{\xi}}' (0) = F_u (0) = 0$ and $\tilde{\tilde{\xi}} (0) = 1$. Thus, by the above, we have that
\[
\tilde{\tilde{\xi}} (u_0) F_{uu} (u_0) - \tilde{\tilde{\xi}}' (u_0) F_u (u_0) = \tilde{\tilde{\xi}} (0) F_{uu} (0) = F_{uu} (0) \neq 0 \,.
\]
Next, fix some $\varepsilon > 0$. Since $F_u (u), F_{uu} (u) \to 0$ uniformly as $|u| \to \infty$ (refer to (\ref{FinutzCoord}) for the formula of $F$), there is some $\delta$ such that for $|u| > \delta$
\begin{align}\label{ImpIneq}
|F_{uu} (0)| &= |\tilde{\tilde{\xi}} (u) F_{uu} (u) - \tilde{\tilde{\xi}}' (u) F_u (u)| \nonumber \\
&\leq | \tilde{\tilde{\xi}} (u) F_{uu} (u) | +  |\tilde{\tilde{\xi}}' (u)
F_u (u)| \text{, which by Cauchy-Schwarz} \nonumber \\
&\leq (| \tilde{\tilde{\xi}}(u)|^2 + |\tilde{\tilde{\xi}}'
(u)|^2)^{1/2} (|F_u (u)|^2
+ |F_{uu}|^2)^{1/2} \nonumber \\
&\leq \sqrt{2} \varepsilon (|\tilde{\tilde{\xi}} (u)|^2 + |\tilde{\tilde{\xi}}' (u)|^2)^{1/2} \,.
\end{align}
If $\xi \in \dot{H}^1$, then since we are in $(u,\theta,\zeta)$-coordinates
\begin{equation}\label{FN}
\| \xi \|_{\dot{H}^1}^2 = \omega_m \int_{\mathbb{S}^{n-1}} \int_0^{\pi/2} \int_\mathbb{R} \gamma^2 | \xi |^2 + | \xi_u |^2 \mathrm{d}u \cos^{m-1} \theta \sin^{n-1} \theta \mathrm{d}\theta  \mathrm{d}\Omega(\zeta) < \infty
\end{equation}
($\Omega$ denotes the uniform probability norm on $\mathbb{S}^{n-1}$). (\ref{ImpIneq}) and (\ref{FN}) imply that $F_{uu} (0) = 0$, contradicting the fact that $F_{uu} (0) \neq 0$. Thus, $\xi \notin \dot{H}^1$, and by the argument in the first paragraph of this proof, Proposition \ref{FuSpanProp} must hold.
\end{proof}

\noindent Since we know that $( \frac{\mathrm{d}}{\mathrm{d}t} F, 0) \in \operatorname{Null} (C_{m,n}^2 I - S)$, Proposition \ref{FuSpanProp} allows us to conclude (\ref{ReONS}).

\textit{\underline{Step 3:} Reduce proving that the zeroes of $C_{m,n}^2 A - \mathcal{L}^{Re}$ in $\{ F \}^{\perp_{\dot{H}^1}}$ are independent of $\theta$ and $\zeta$  to proving that the zeroes of $C_{m,n}^2 \hat{A} - \mathcal{L}^{Re}$ in $\{ F \}^{\perp_{\dot{H}^1}}$ are independent of $\theta$ and $\zeta$.} Recall that $\hat{A}$ (refer to (\ref{HatA}) for reference) is an operator that is closely related to $A$ such that $\hat{A} \leq A$ in $L^2$. Thus,
\[
C_{m,n}^2 \hat{A} - \mathcal{L}^{Re} \leq C_{m,n}^2 A - \mathcal{L}^{Re} \text{ in $L^2$} \,.
\]
If we can show that
\begin{equation}\label{ReONSRdn3}
\left\langle \xi, (C_{m,n}^2 \hat{A} - \mathcal{L}^{Re}) \xi \right\rangle_{L^2} \geq 0 \text{, for } \xi \in \{ F \}^{\perp_{\dot{H}^1}} \,,
\end{equation}
then we only need to show that the zeroes of $C_{m,n}^2 \hat{A} - \mathcal{L}$ in $\{ F \}^{\perp_{\dot{H}^1}}$ are independent of $\theta$ and $\zeta$ in order to prove that the zeroes of $C_{m,n}^2 A - \mathcal{L}$ in $\{ F \}^{\perp_{\dot{H}^1}}$ are independent of $\theta$ and $\zeta$. Thus, we prove the following
\begin{prop}\label{RdnProp}
Let $\hat{A}$ be as defined in (\ref{HatA}). Then (\ref{ReONSRdn3}) holds.
\end{prop}

\begin{proof}
We begin by verifying (\ref{ReONSRdn3}). First we observe that in $(u, \theta, \zeta)$-coordinates
\begin{equation}\label{HatACoord}
\hat{A} = \gamma^2 I - \frac{\partial^2}{\partial u^2} - \frac{\partial^2}{\partial \theta^2} - ([n-1] \cot \theta - [m-1] \tan \theta) \frac{\partial}{\partial \theta} - \Delta_{\mathbb{S}^{n-1} (\zeta)} \,.
\end{equation}
Then, we define
\begin{eqnarray}
X &:=& C_{m,n}^2 \left( \gamma^2 I - \frac{\partial^2}{\partial u^2} \right) - (2^*-1) C_{m,n}^{2-2^*} F^{2^*-2} \nonumber \\
Y &:=& C_{m,n}^2 \left[ \frac{\partial^2}{\partial \theta^2} - ([n-1]\cot \theta - [m-1] \tan \theta) \frac{\partial}{\partial \theta} \right] \text{, and} \nonumber \\
Z &:=& -C_{m,n}^2 \Delta_{\mathbb{S}^{n-1} (\zeta)} \,, \nonumber
\end{eqnarray}
so that
\begin{equation}\label{XYZEqn}
C_{m,n}^2 \hat{A} - \mathcal{L}^{Re} = X + Y + Z \,, \text{ on } \{ F \}^{\perp_{\dot{H}^1}} \,.
\end{equation}
$Y$ and $Z$ are positive complete operators in $L^2([0,\pi/2], \cos^{m-1} \theta \sin^{n-1} \theta \mathrm{d} \theta)$ and $L^2 (\mathbb{S}^{n-1}, \mathrm{d} \Omega(\zeta))$ respectively. $X$ is also a complete operator for $L^2 (\mathbb{R}, \mathrm{d}u)$; this is a result of standard spectral theory. Moreover, $X$ is positive for functions in $L^2 (\mathbb{R}, \mathrm{d} u)$ that are in $\{ F \}^{\perp_{\dot{H}^1}}$ when considered as functions in $L^2 (\mathbb{R} \times [0,\pi/2] \times \mathbb{S}^{n-1}, \omega_m \mathrm{d}u \cos^{m-1} \theta \sin^{n-1} \theta \mathrm{d} \theta \mathrm{d} (\zeta))$. This is because if $\xi \in \{ F \}^{\perp_{\dot{H}^1}}$ is independent of $\theta$ and $\zeta$, then
\[
0 \leq \langle \xi, (C_{m,n}^2 I - S) \xi \rangle_{\dot{H}^1} = \langle \xi, (C_{m,n}^2 A - \mathcal{L}^{Re}) \xi \rangle_{L^2} = \langle \xi, X \xi \rangle_{L^2} \,.
\]
Combining the properties of $X$, $Y$, and $Z$ deduced above with (\ref{XYZEqn}), we conclude that (\ref{ReONSRdn3}) holds.
\end{proof}

\textit{\underline{Step 4:} Show that zeroes of $C_{m,n}^2 \hat{A} - \mathcal{L}^{Re}$ in $\{ F \}^{\perp_{\dot{H}^1}}$ are independent of $\theta$ and $\zeta$.} We start by establishing independence from $\theta$ by proving the following
\begin{prop}
The space of eigenfunctions of $Y$ with Neumann boundary conditions and exactly one zero on the interval $[0, \pi/2]$ is spanned by
\begin{equation}\label{gEqn}
g (\theta) = \frac{n-m}{m+n} + \cos (2 \theta) \,,
\end{equation}
with eigenvalue
\begin{equation}\label{gEV}
\lambda = 2(m+n) C_{m,n}^2 \,.
\end{equation}
\end{prop}
\begin{proof}
The key to proving this proposition is rewriting the coefficient of the
$\frac{\partial}{\partial \theta}$ term of $Y$ in terms of $\cos(2\theta)$ and
$\sin(2\theta)$. More precisely,
\[
(n-1)\cot \theta - (m-1) \tan \theta = \frac{n-m}{\sin(2\theta)} +
(m+n-2)\frac{\cos(2\theta)}{\sin(2\theta)} \,.
\]
Thus,
\begin{eqnarray}
C_{m,n}^{-2} Y g &=& - g'' - \left[ \frac{n-m}{\sin(2\theta)} +
(m+n-2)\frac{\cos(2\theta)}{\sin(2\theta)} \right] g' \nonumber \\
&=& 4\cos(2\theta) + 2 \left[
\frac{n-m}{\sin(2\theta)}+(m+n-2)\frac{\cos(2\theta)}{
\sin(2\theta)}\right] \sin(2\theta) \nonumber \\
&=& 2 [ (n-m) + (m+n)\cos(2\theta) ] \nonumber \\
&=& 2(m+n) g \,. \nonumber
\end{eqnarray}
The fact that there are no more linearly independent eigenfunctions of $Y$ with exactly one zero in $[0,\pi/2]$ follows from standard Sturm-Liouville Theory.
\end{proof}
The only eigenfunctions of $Y$ without any zeroes satisfying the Von Neumann conditions are the constant functions. Any eigenfunctions of $Y$ satisfying the Von Neumann conditions, excluding the constant functions and $g(\theta)$, will have eigenvalue more than $2(m+n) C_{m,n}^2$ - this is also a consequence of standard Sturm-Liouville Theory. Recalling that $C_{m,n}^2 \hat{A} - \mathcal{L}^{Re} = X + Y + Z$ on $\{ F \}^{\perp_{\dot{H}^1}}$, where $X$ depends on $u$ only, $Y$ depends on $\theta$ only, $Z$ depends on $\zeta$ only, and that $\langle \xi, (C_{m,n}^2 \hat{A} - \mathcal{L}^{Re}) \xi \rangle_{L^2} \geq 0$ for $\xi \in \{ F \}^{\perp_{\dot{H}^1}}$, our analysis of $Y$ allows us to conclude that the 0-modes of $C_{m,n}^2 \hat{A} - \mathcal{L}^{Re}$ must be independent of $\theta$. 

Next, we establish independence from $\zeta$. The eigenvalues of $Z$ are nonnegative and discrete. The smallest eigenvalue is $\sigma_0 = 0$ and the corresponding space of eigenfunctions are the constant functions. The second smallest eigenvalue is $\sigma_1 = (n-1) C_{m,n}^2$, see [20] for reference. Thus, the zeroes of $C_{m,n}^2 \hat{A} - \mathcal{L}^{Re}$ in $\{ F \}^{\perp_{\dot{H}^1}}$ must be independent of $\zeta$. This concludes step 4.

Combining the results of steps 1-4 allows us to conclude Lemma \ref{NSLm2}.
\end{proof}

The last thing we need to prove in order to conclude Theorem \ref{NSThm} is that $\psi \perp_{\dot{H}_\mathbb{C}^1} \{ (F, 0), ( \frac{\mathrm{d}}{\mathrm{d}t} F, 0), (0, F) \}$.  We already showed that $\psi \perp_{\dot{H}_\mathbb{C}^1} (F, 0)$ in Lemma \ref{LocBELm1}.  Thus, it suffices to prove the following
\begin{lm}\label{OLm}
$\psi \perp_{\dot{H}_\mathbb{C}^1} \{ ( \frac{\mathrm{d}}{\mathrm{d}t} F, 0), (0, F) \}$.
\end{lm}
\begin{proof}
First we prove that $\psi \perp_{\dot{H}_\mathbb{C}^1} (\frac{\mathrm{d}}{\mathrm{d}t} F, 0)$. We begin by observing that
\[
\delta (\varphi, M) = \| \varphi - z F_{t,0} \|_{\dot{H}^1}
\]
implies that
\begin{equation}\label{dtE0}
0 = \frac{\mathrm{d}}{\mathrm{d} s} \big|_{s = t} \| \varphi - z F_{s,0} \|_{\dot{H}^1}^2 \,.
\end{equation}
Exploiting (\ref{dtE0}), we get that (in the following subscripts denote partial derivatives)
\begin{eqnarray}
0 &=& \frac{\mathrm{d}}{\mathrm{d} s} \big|_{s = t} \| \varphi - z F_{s,0} \|_{\dot{H}^1}^2 \nonumber \\
&=& \frac{\mathrm{d}}{\mathrm{d} s} \big|_{s = t} \int \gamma^2 | \varphi - z F_{s,0}|^2 + |(\varphi - z F_{s,0})_u|^2 + | (\varphi - z F_{s,0})_{\theta}|^2 + \csc^2 \theta |\nabla_{\mathbb{S}^{n-1}} (\varphi - z F_{s,0})|^2 \mathrm{d} \Lambda \nonumber \\
&=& 2 \int \gamma^2 (- z \frac{\mathrm{d}}{\mathrm{d}t} F) \operatorname{Re} (\varphi - z F) + (- z \frac{\mathrm{d}}{\mathrm{d} t} F)_u \operatorname{Re} [ (\varphi - z F)_u] + (- z \frac{\mathrm{d}}{\mathrm{d} t} F)_\theta \operatorname{Re} [(\varphi - z F)_\theta] \nonumber \\
&& + 2 \csc^2 \theta [\nabla_{\mathbb{S}^{n-1}} (- z \frac{\mathrm{d}}{\mathrm{d} t} F) ] \cdot  [ \nabla_{\mathbb{S}^{n-1}} \operatorname{Re} (\varphi - z F)] \mathrm{d} \Lambda \nonumber \\
&=& -2 z^2 \left\langle \delta (\varphi, M) \psi, (\frac{\mathrm{d}}{\mathrm{d}t} F, 0) \right\rangle_{\dot{H}_\mathbb{C}^1} \,, \nonumber
\end{eqnarray}
i.e. $\psi \perp_{\dot{H}^1_\mathbb{C}} (\frac{\mathrm{d}}{\mathrm{d}t} F, 0)$.

Next, we show that $\psi \perp_{\dot{H}_\mathbb{C}^1} (0, F)$. Recall that $(0, F)$ is an eigenfunction of $S_t: \dot{H}_\mathbb{C}^1 \to \dot{H}_\mathbb{C}^1$, which is self-adjoint and compact.  Thus, if $\psi$ were not perpendicular to $(0, F)$, then since $\| (0,F) \|_{\dot{H}^1} = 1$,
\begin{equation}\label{OPLT1}
\left\| \psi - \langle \psi, (0, F) \rangle_{\dot{H}_\mathbb{C}^1} (0, F) \right\|_{\dot{H}^1} < \| \psi \|_{\dot{H}^1} = 1 \,.
\end{equation}
Letting $\varepsilon = \langle \psi, (0, F) \rangle_{\dot{H}_\mathbb{C}^1}$, we deduce that
\begin{eqnarray}
\| \varphi - (z F, \delta (\varphi, M) \varepsilon F) \|_{\dot{H}^1} &=& \| \delta (\varphi, M) \psi - \delta (\varphi, M) \varepsilon (0, F) \|_{\dot{H}^1} \nonumber \\
&=& \delta (\varphi, M) \| \psi - \varepsilon (0, F) \|_{\dot{H}^1} \text{, which by (\ref{OPLT1})} \nonumber \\
&<& \delta (\varphi, M) \nonumber \,.
\end{eqnarray}
This contradicts the assumption (\ref{MinFCdn}) made in Lemma \ref{LocBELm} that
\[
\delta (\varphi, M) = \| \varphi - z F \|_{\dot{H}^1} \,,
\]
because $(z F, \delta (\varphi, M) \varepsilon F) \in M$.  Thus, $\psi \perp_{\dot{H}_\mathbb{C}^1} (0, F)$.
\end{proof}

Combining lemmas \ref{OpPLm}, \ref{NSLm2}, and \ref{OLm} we conclude Theorem \ref{NSThm}. Thus, we have proven theorems \ref{StSACThm} and \ref{NSThm}. Combining these theorems with the outline of the proof of Theorem \ref{LocBEThm} provided in section 3, we conclude Theorem \ref{LocBEThm}.

\section{Proof of Theorem \ref{MainThm}}

We begin by proving the sharpness statement. Let $\varphi \in \dot{H}_\mathbb{C}^1$ satisfy the assumptions of Lemma \ref{LocBELm}. Applying the results of the second order Taylor Expansion with the remainder bound, a calculation similar to the one used to obtain (\ref{SimpExp}) yields that
\begin{align}\label{UBd}
C_{m,n}^2 \| \varphi \|_{\dot{H}^1}^2 - \| \varphi \|_{2^*}^2 &\leq \left\langle (C_{m,n}^2 I - S_t) \psi, \psi \right\rangle_{\dot{H}_\mathbb{C}^1} \delta (\varphi, M)^2 + \frac{\kappa_{2^*} C_{m,n}^2}{4 \cdot 3^{\frac{\beta_{2^*}}{2} - 1}} \delta (\varphi, M)^{\beta_{2^*}} \nonumber \\
&\leq \tilde{C} \delta (\varphi, M)^2 \text{, for some $\tilde{C} > 0$} \,,
\end{align}
because $S_t: \dot{H}_\mathbb{C}^1 \to \dot{H}_\mathbb{C}^1$ has a bounded spectrum, $\delta (\varphi, M) \leq 1$, and $\beta_{2^*} > 2$. If the sharpness statement at the end of Theorem \ref{MainThm} were false, then we could find some $\tilde{\alpha} > 0$ such that
\begin{equation}\label{Sharp}
C_{m,n}^2 \| \varphi \|_{\dot{H}^1}^2 - \| \varphi \|_{2^*}^2 \geq \tilde{\alpha} \delta (\varphi, M)^\beta \,,
\end{equation}
for some $\beta < 2$ and all $\varphi \in \dot{H}_\mathbb{C}^1$. However, for $\varphi$ obeying the conditions of Lemma \ref{LocBELm}, (\ref {UBd}) and (\ref{Sharp}) would imply
\[
\tilde{C} / \tilde{\alpha} \geq \delta (\varphi, M)^{-(2 - \beta)} \text{, with $\beta < 2$} \,,
\]
which is clearly a contradiction for $\delta (\varphi, M)$ sufficiently small. This proves the sharpness statement.

The rest of the proof of Theorem \ref{MainThm} follows by contradiction.  Assume Theorem \ref{MainThm} is false. Then, there is some $( \varphi_j) \subseteq \dot{H}_\mathbb{C}^1$ such that
\begin{equation}\label{ConvCdn}
 \frac{C_{m,n}^2 \| \varphi_j \|_{\dot{H}^1}^2 - \| \varphi_j \|_{2^*}^2}{\delta (\varphi_j, M)^2} \to 0 \,.
\end{equation}
We can assume that $\| \varphi_j \|_{\dot{H}^1} = 1$ for all $j$, because replacing
$\varphi_j$ with $c \varphi_j$, $c$ a nonzero constant, does not change the value of
the left hand side of (\ref{ConvCdn}).  In this case, $( \delta (\varphi_j, M) ) \in [0,1]$
for all $j$, and some subsequence, $( \delta (\varphi_{j_k},
M) )$, converges to some $B \in [0,1]$.  If $B = 0$, then (\ref{ConvCdn}) contradicts
Theorem \ref{LocBEThm}, specifically condition (\ref{LBEwRem}) of Theorem \ref{LocBEThm}.

Next, we show that $B$ must equal 0. A fortiori, (\ref{ConvCdn}) implies
\begin{equation}\label{ConvCdn1}
C_{m,n}^2 \| \varphi_{j_k} \|_{\dot{H}^1}^2 - \| \varphi_{j_k} \|_{2^*}^2 \to 0 \,.
\end{equation}
We will use (\ref{ConvCdn1}) in a Concentration Compactness argument to show that a
subsequence, say $( \delta (\hat{\varphi}_{k_l}, M) )$, of $( \delta (\hat{\varphi}_k, M) )$ converges to zero, where $\hat{\varphi}_k$ is given by
\[
\hat{\varphi}_k (\rho, x) = \sigma_k^\gamma \varphi_{j_k} (\sigma_k \rho, \sigma_k (x - x_k)) \,,
\]
for some $(\sigma_k) \subseteq \mathbb{R}_+$ and $(x_k) \subseteq \mathbb{R}^n$. 
Since $\delta (\cdot, M)$ is
conformally invariant,
\[
\delta (\varphi_{j_{k_l}}, M) = \delta ( \hat{\varphi}_{k_l}, M) \to 0 \,,
\]
from which we conclude that $B$ must in fact be 0.  Thus, we have reduced the proof of Theorem \ref{MainThm} to illustrating the Concentration Compactness argument for cylindrically symmetric functions in continuous dimension. We do this in detail in the next and final section.

\section{Concentration Compactness}

In the following, we will assume that $\| \varphi_j \|_{2^*} = 1$, instead of $\| \varphi_j \|_{\dot{H}^1} = 1$.  This does not change any key properties.  More
precisely, (\ref{ConvCdn}) will still hold and we can replace the assumption that $\delta (\varphi_{j_k}, M) \to B \neq 0$ with $\delta (\varphi_{j_k}, M) \to B/C_{m,n}$, which is also nonzero.  We also establish some notation.  Let $\varphi$ be a cylindrically symmetric function and $\sigma > 0$.  Then $\varphi^\sigma$ is given by (recall by (\ref{2*aGDef}) that $\gamma  = \frac{m+n-2}{2}$)
\[
\varphi^\sigma (\rho, x) = \sigma^\gamma \varphi (\sigma \rho, \sigma x) \,.
\]
In this section, we will prove
\begin{thm}\label{CCThm}
Let $( \varphi_j ) \subseteq \dot{H}_\mathbb{C}^1$ be such that
\begin{equation}\label{NRCdn}
C_{m,n}^2 \| \varphi_j \|_{\dot{H}^1}^2 - \| \varphi_j \|_{2^*}^2 \to 0 \text{ and } \|
\varphi_j \|_{2^*} = 1, \forall j \,.
\end{equation}
Then there is some $( \sigma_j ) \subseteq \mathbb{R}_+$ and $(x_j) \subseteq
\mathbb{R}^n$ such that $( \hat{\varphi}_j )$ given by
\begin{equation}\label{ASCdn}
\hat{\varphi}_j (\rho, x) = \varphi_j^{\sigma_j} (\rho, x + x_j) \,,
\end{equation}
has a subsequence, $( \hat{\varphi}_{j_k} )$, that converges strongly in $\dot{H}_\mathbb{C}^1$ to some $F \in M$ such that $\| F \|_{2^*} = 1$.
\end{thm}
The proof of this theorem breaks into three parts.  In the first part, we
prove the following
\begin{lm}\label{MLm}
Let $( \varphi_j ) \subseteq \dot{H}_\mathbb{C}^1$ satisfy condition (\ref{NRCdn}).  Then, there is some $( \sigma_j ) \subseteq \mathbb{R}_+$ such that $(\varphi_j^{\sigma_j})$ has a subsequence, $( \varphi_{j_k}^{\sigma_{j_k}} )$, and some $\varepsilon, C > 0$ such that
\begin{equation}\label{MCdn}
\Lambda ( \{ | \varphi_{j_k}^{\sigma_{j_k}} (\rho, x) | > \varepsilon, \rho \leq 4 \} ) > C \,,
\end{equation}
where $\Lambda$ denotes the measure defined in (\ref{LDef}).
\end{lm}
\noindent Once we have proved Lemma \ref{MLm}, we can apply an analogue of Lieb's Concentration Compactness Theorem, Theorem 8.10 on page 215 of [16], to the subsequence $( \varphi_{j_k}^{\sigma_{j_k}} )$ that satisfies (\ref{MCdn}).  We state this analogue below:
\begin{thm}\label{LiebThm}
Let $( \varphi_j )$ be a bounded sequence of functions in $\dot{H}_\mathbb{C}^1$.  Suppose
there exist $\varepsilon > 0$ and $R < \infty$ such that $E_j := \{ 
| \varphi_j ( \rho, x) | > \varepsilon, \rho \leq R \}$ has measure $\Lambda (E_j) \geq \delta >
0$ for some $\delta$ and for all $j$.  Then, there exists $(x_j) \subseteq
\mathbb{R}^n$ such that $\varphi_j^T (\rho, x) := \varphi_j (\rho, x + x_j)$ has a
subsequence that converges weakly in $\dot{H}_\mathbb{C}^1$ to a nonzero function.
\end{thm}
\noindent We delay the proof of this theorem to the end of this section, because it is a relatively straightforward adaptation of Lieb's proof of his original Concentration Compactness Theorem. After applying Theorem \ref{LiebThm} and relabeling indices, we deduce some $(\hat{\varphi}_j )$, as given in (\ref{ASCdn}), such that some subsequence, $(\hat{\varphi}_{j_k} )$, converges weakly in $\dot{H}_\mathbb{C}^1$ to a nonzero element, $\varphi$.  The second part of the proof of Theorem \ref{CCThm} involves some relatively straightforward functional analysis arguments that allow us to show that $( \hat{\varphi}_{j_k} )$ converges strongly in $\dot{H}_\mathbb{C}^1$ to some $\varphi \in M$ such that $\| \varphi \|_{2^*} = 1$. The third part of the proof of Theorem \ref{CCThm} is the delayed proof of Theorem \ref{LiebThm}.

\underline{\textbf{Part 1 of proof of Theorem \ref{CCThm}} - Proving Lemma \ref{MLm}:} Most of the hard work in proving Theorem \ref{CCThm} is devoted to proving Lemma \ref{MLm}. We break the proof of Lemma \ref{MLm} into three steps, which we outline in this paragraph. In the first step, we reduce proving Lemma \ref{MLm} to proving that (\ref{MCdn}) holds for a modified subsequence of $(\varphi_j)$. More precisely, we take the sequence $(\varphi_j)$ and dilate its elements to obtain a new sequence $(\varphi_j^{\sigma_j})$ such that the symmetric decreasing rearrangments, $\tilde{\varphi}_j$, in the $x$-variable of the $\varphi_j^{\sigma_j}$ have the property that
\begin{equation}\label{HCdn}
\| \chi_{ \{ ( \rho^2 + |x|^2 )^{1/2} \leq 1 \} } \tilde{\varphi}_j \|_{2^*}^{2^*} = 1/2, \forall j \,.
\end{equation}
We then conclude that to prove Lemma \ref{MLm}, it suffices to show that (\ref{MCdn}) holds for some subsequence of the modified sequence, $(\tilde{\varphi}_j)$. In the second and third steps, we prove that (\ref{MCdn}) holds for a subsequence of $(\tilde{\varphi}_j)$. Actually, we show that (\ref{MCdn}) holds not just for a subsequence of $( \tilde{\varphi}_j )$, but for a subsequence of $( \chi
\tilde{\varphi}_j )$, where $\chi$ is a nicely behaved cutoff function.  In the second step, we show that a subsequence of $( \| \chi_{ \{1 \leq w \leq 2 \} } \tilde{\varphi}_j \|_{2^*}^{2^*} )$ converges to a positive constant; the result of this step is summarized in Proposition \ref{ANZProp}. In the third step, we leverage the result of the second step to show the $\tilde{\varphi}_j$ times one of two possible cutoff functions yields a sequence that satisfies the $p,q,r$-Theorem; the result of this step is summarized in Proposition \ref{PQRProp}. From this, we deduce that there must be some cutoff, $\chi \leq 1$, such that a subsequence of $( \chi \tilde{\varphi}_j )$ satisfies (\ref{MCdn}). This concludes the proof of Lemma (\ref{MLm}).

\textit{\underline{Step 1:} Reduce the proof of Lemma \ref{MLm} to its analogue for a modified sequence, $(\tilde{\varphi}_j)$.} For a cylindrically symmetric function, $\varphi$, let $\varphi^*$ denote the symmetric decreasing rearrangement of $\varphi$ in the $x$-variable.  To be more precise, $\varphi^*$ is the nonnegative function obtained by the following: Fix $\rho \in (0, \infty)$.  Then, $\varphi^* (\rho,x)$ is the symmetric decreasing function in $x$ such that for all
$\varepsilon > 0$
\begin{equation}\label{SymDecRe}
\left| \{ x \in \mathbb{R}^n \big| | \varphi^* (\rho,x) | > \varepsilon \}
\right| = \left| \{ x \in \mathbb{R}^n \big| | \varphi (\rho, x) | > \varepsilon \}
\right| \,,
\end{equation}
where $\left| \cdot \right|$ denotes Lebesgue measure on $\mathbb{R}^n$.

We will choose $(\sigma_j) \subseteq \mathbb{R}_+$ such that $\tilde{\varphi}_j := ( \varphi_j^{\sigma_j})^*$ satisfies (\ref{HCdn}). We can choose such $\sigma_j$, because of
\begin{lm}\label{ArrProp}
Let $\varphi \in \dot{H}_\mathbb{C}^1$ be a function such that
\begin{equation}\label{N1Cdn}
\| \varphi \|_{2^*} = 1 \,.
\end{equation}
Then, for any $c \in (0,1)$, there is some $\sigma \in \mathbb{R}_+$ such that
\begin{equation}\label{ReSiz}
\| \chi_{ \{ (\rho^2 + |x|^2)^{1/2} \leq 1 \} } ( \varphi^\sigma )^* \|_{2^*}^{2^*}
= c \,.
\end{equation}
\end{lm}

\begin{proof}
We will use $(w, \theta, \zeta)$-coordinates instead of $(\rho,x)$-coordinates (refer to (\ref{wtzCoord}) for reference). 
We prove this lemma by first showing that we can pick some $\sigma_2 > \sigma_1 > 0$ such that
\begin{align}
\| \chi_{ \{ w \leq 1 \} } (\varphi^{\sigma_1})^* \|_{2^*}^{2^*} &\leq c/2 \text{, and} \label{ReSiz1} \\ 
\| \chi_{ \{ w \leq 1 \} } ( \varphi^{\sigma_2} )^* \|_{2^*}^{2^*} &\geq (c + 1)/2 \,, \label{ReSiz2}
\end{align}
and then proving that the map,
\[
\sigma \mapsto \| \chi_{ \{ w \leq 1 \} } ( \varphi^\sigma )^* \|_{2^*},
\]
is continuous.  Thus, by the Intermediate Value Theorem, there is some $\sigma \in \mathbb{R}_+$ for which (\ref{ReSiz}) holds.

First, we show (\ref{ReSiz1}).  By the Dominated Convergence Theorem, we can pick some $\tilde{M}$ such that
\[
E_{\tilde{M}} = \{ (w, \theta, \zeta) \in \mathbb{R}_+ \times [0,\pi/2] \times \mathbb{S}^{n-1} \big| | \varphi (w, \theta, \zeta) | \leq \tilde{M} \} \,,
\]
has the property that
\begin{equation}\label{Siz}
\| \chi_{E_{\tilde{M}}} \varphi \|_{2^*}^{2^*} \geq 1 - \frac{c}{4} \,.
\end{equation}
If we pick $\sigma_1$ small enough such that
\begin{equation}\label{S1Cdn}
\sigma_1^{m+n} \tilde{M}^{2^*} \leq \frac{c}{4 \Lambda ( \{ w \leq 1 \} )} \,,
\end{equation}
and take $E_{\tilde{M}}^C$ to be the complement of $E_{\tilde{M}}$, then
\begin{eqnarray}
\| \chi_{ \{ w \leq 1\} } ( \varphi^{\sigma_1} )^* \|_{2^*}^{2^*} &=& \| \chi_{ \{ w \leq 1 \} } ( [ \chi_{E_{\tilde{M}}} \varphi ]^{\sigma_1} + [ \chi_{E_{\tilde{M}}^C}
\varphi ]^{\sigma_2} )^* \|_{2^*}^{2^*} \text{, because } E_{\tilde{M}} \cap E_{\tilde{M}}^C = \emptyset \nonumber \\
&\leq& \| \chi_{ \{ w \leq 1 \} } ( [ \chi_{E_{\tilde{M}}} \varphi ]^{\sigma_1} )^*
\|_{2^*}^{2^*} + \| \chi_{ \{ w \leq 1 \} } ( [ \chi_{E_{\tilde{M}}^C} \varphi ]^{\sigma_1}
)^* \|_{2^*}^{2^*} \nonumber \\
&& \text{because } \{ (w, \theta, \zeta) \big| ( \chi_{E_{\tilde{M}}} \varphi )^{\sigma_1} >
0 \} \cap \{ (w, \theta, \zeta) \big| ( \chi_{E_{\tilde{M}}^C} \varphi )^{\sigma_1} > 0 \} =
\emptyset \nonumber \\
&\leq& \Lambda ( \{ w \leq 1 \} ) \| [\chi_{E_{\tilde{M}}} \varphi ]^{\sigma_1}
\|_\infty^{2^*} + \| \chi_{E_{\tilde{M}}^C} \varphi \|_{2^*}^{2^*} \text{, which by (\ref{N1Cdn}), (\ref{Siz}), and (\ref{S1Cdn})} \nonumber \\
&\leq& c/2 \,. \nonumber
\end{eqnarray}
This proves (\ref{ReSiz1}).

Next, we show (\ref{ReSiz2}).  By the Dominated Convergence Theorem, there is some $R < \infty$ such that
\[
\| \chi_{ \{w \leq R \} } \varphi \|_{2^*}^{2^*} \geq (c + 1)/2 \,.
\]
Taking $\sigma_2 = R$ yields (\ref{ReSiz2}).

Finally, we will show that
\begin{equation}\label{Map}
\sigma \mapsto \| \chi_{ \{ w \leq 1 \} } ( \varphi^\sigma )^* \|_{2^*}
\end{equation}
is continuous.  Fix some $\varepsilon > 0$ and $\sigma_0 \in \mathbb{R}_+$. 
Take some $\Psi \in C_C^\infty ( [0, \infty) \times [0, \pi/2] \times
\mathbb{S}^{n-1} )$ such that
\[
\| \varphi - \Psi \|_{2^*} < \varepsilon / 3 \,.
\]
Then,
\begin{align}\label{CtyIneq}
\left| \| \chi_{ \{ w \leq 1 \} } ( \varphi^\sigma )^* \|_{2^*} - \| \chi_{ \{ w \leq 1 \} } ( \varphi^{\sigma_0} )^* \|_{2^*} \right| \leq& \| ( \varphi^\sigma )^* - ( \varphi^{\sigma_0} )^* \|_{2^*} \nonumber \\
\leq& \| \varphi^\sigma - \varphi^{\sigma_0} \|_{2^*} \nonumber \\
\leq& \| \varphi^\sigma - \Psi^\sigma \|_{2^*} + \| \Psi^\sigma - \Psi^{\sigma_0}
\|_{2^*} + \| \Psi^{\sigma_0} - \varphi^{\sigma_0} \|_{2^*} \nonumber \\
=& \| (\varphi - \Psi)^\sigma \|_{2^*} + \| \Psi^\sigma - \Psi^{\sigma_0} \|_{2^*}
+ \| (\varphi - \Psi)^{\sigma_0} \|_{2^*} \nonumber \\
<& 2 \varepsilon/3 + \| \Psi^\sigma - \Psi^{\sigma_0} \|_{2^*} \,.
\end{align}
Thus, we only need to show that the map given by
\begin{equation}\label{Map1}
\sigma \mapsto \Psi^\sigma
\end{equation}
is continuous at $\sigma_0$ to prove continuity of the map given by (\ref{Map}).  We do this by proving sequential continuity.  Let $( \sigma_j ) \subseteq \mathbb{R}_+$ be a sequence such that $\sigma_j \to \sigma_0$. Then, $(\sigma_j)$ has a finite supremum, $k_1$, and a positive infimum, $k_2$. Also, since $\Psi \in C_C ([0,\infty) \times \mathbb{R}^n \times \mathbb{S}^n)$, there is some $N < \infty$ that bounds $\Psi$ from above and some $R < \infty$ such that $\operatorname{supp} (\Psi) \subseteq \{ w \leq R \}$. Combining these facts with the definition of the dilation operation given by $\Psi \mapsto \Psi^\sigma$, we conclude that $(\Psi^{\sigma_j})$ has a ceiling function, $\Xi \in L^{2^*}$, given by
\[
\Xi (w, \theta, \zeta) = k_1^{(m+n)/2^*} N \chi_{ \{w \leq R/k_2\} } \,.
\]
Also, $\Psi^{\sigma_j} \to \Psi^{\sigma_0}$ pointwise, because $\Psi$ is continuous. Combining the existence of the $L^{2^*}$ ceiling function, $\Xi$, for $(\Psi_{\sigma_j})$ with the pointwise convergence of $\Psi^{\sigma_j}$ to $\Psi^{\sigma_0}$, we apply the Dominated Convergence Theorem and conclude that
\[
\lim_{j \to \infty} \| \Psi^{\sigma_j} - \Psi^{\sigma_0} \|_{2^*} = 0 \,.
\]
Thus, the map given by (\ref{Map1}) is continuous.  Combining this with (\ref{CtyIneq}), we conclude that
\[
\left| \| \chi_{ \{ w \leq 1\} } (\varphi^\sigma)^* \|_{2^*} - \| \chi_{ \{ w \leq 1 \} } ( \varphi^{\sigma_0} )^* \|_{2^*} \right| < \varepsilon \,,
\]
for $\sigma$ sufficiently close to $\sigma_0$.  Hence, the map given by (\ref{Map}) is continuous.
\end{proof}
The definition of $\varphi^*$ ensures that if a subsequence of $( \tilde{\varphi}_j )$
satisfies (\ref{MCdn}), in the sense that
\[
\Lambda ( \{ | \tilde{\varphi}_{j_k}| > \varepsilon, \rho \leq 4 \}) > C \,,
\]
for some $\varepsilon, C > 0$, then the corresponding subsequence, $(
\varphi_{j_k}^{\sigma_{j_k}} )$, will satisfy (\ref{MCdn}).  Thus, in order to prove Lemma \ref{MLm}, it suffices to prove
\begin{lm}\label{MLm1}
The sequence of functions, $( \tilde{\varphi}_j )$, has a subsequence that
satisfies (\ref{MCdn}).
\end{lm}

\noindent Steps two and three are devoted to proving Lemma \ref{MLm1}. Each step contains a proposition that helps us to prove Lemma \ref{MLm1}.

\textit{\underline{Step 2:} Show that a subsequence of $( \| \chi_{ \{1 \leq w \leq 2 \} } \tilde{\varphi}_j \|_{2^*}^{2^*} )$ converges to a positive constant.} Before stating and proving the main proposition in this step, we lay some foundation.  First, we collect a couple of inequalities that we will use later.
\begin{prop}\label{HProp}
Let $h_1,h_2 \in L^p$ be such that $0<\| h_1 \|_p \leq \| h_2 \|_p$. Then
\begin{align}
\| h_1 + h_2 \|_p &\leq \| h_1 \|_p + \| h_2 \|_p - \frac{(p-1) \| h_1 \|_p}{4} \left\| \frac{h_1}{\|h_1\|_p} - \frac{h_2}{\|h_2\|_p} \right\|_p^2 \text{, if $1 < p \leq 2$} \label{HRslt1} \\
\| h_1 + h_2 \|_p &\leq \| h_1 \|_p + \| h_2 \|_p - \frac{\|h_1\|_p}{p2^{p-1}}
\left\|
\frac{h_1}{\|h_1\|_p} - \frac{h_2}{\|h_2\|_p} \right\|_p^p \text{, if $2 \leq p < \infty$} \,. \label{HRslt2}
\end{align}
\end{prop}

\begin{proof}
If $1 < p \leq 2$, then by (3.3) of [7]
\begin{equation}\label{HIneq}
\| \hat{h}_1 + \hat{h}_2 \|_p \leq 2 - \frac{p-1}{4} \| \hat{h}_1 - \hat{h}_2 \|_p \,.
\end{equation}
for $\hat{h}_1, \hat{h}_2 \in L^p$ such that $\| \hat{h}_1 \|_p = \| \hat{h}_2 \|_p = 1$.  Thus,
\begin{eqnarray}
\| h_1 + h_2 \|_p &\leq& \left\| \frac{\|h_2\|_p - \|h_1\|_p}{\|h_2\|_p} h_2
\right\|_p + \left\|
\frac{\|h_1\|_p}{\|h_2\|_p} h_2 + h_1 \right\|_p \nonumber \\
&=& \|h_2\|_p - \|h_1\|_p + \|h_1\|_p \left\| \frac{h_2}{\|h_2\|_p} +
\frac{h_1}{\|h_1\|_p}
\right\|_p \text{, which by (\ref{HIneq})} \nonumber \\
&\leq& \| h_2 \|_p + \| h_1 \|_p - \frac{(p-1)\|h_1\|_p}{4} \left\| \frac{h_1}{\|h_1\|_p} - \frac{h_2}{\|h_2\|_p} \right\|_p \,. \nonumber
\end{eqnarray}

The proof of (\ref{HRslt2}) is similar to that of (\ref{HRslt1}), except instead of using (\ref{HIneq}) we use: If $2 \leq p < \infty$, then by (3.4) of [7],
\[
\| \hat{h}_1 + \hat{h}_2 \|_p \leq 2 - \frac{1}{p2^{p-1}} \| \hat{h}_1 - \hat{h}_2 \|_p^p \,,
\]
for $\hat{h}_1, \hat{h}_2 \in L^p$ such that $\| \hat{h}_1 \|_p = \| \hat{h}_2 \|_p = 1$.
\end{proof}
Next, we observe that
\begin{equation}\label{NRCdn1}
C_{m,n}^2 \| \tilde{\varphi}_j \|_{\dot{H}^1}^2 - \| \tilde{\varphi}_j \|_{2^*}^2 \to 0 \text{ and } \| \tilde{\varphi}_j \|_{2^*}^{2^*} = 1, \forall j \,.
\end{equation}
Moreover, $\tilde{\varphi}_j$, is independent of $\zeta$ for all $j$, because each
$\tilde{\varphi}_j$ is rearranged in $x$.  Letting
\[
a_j = \| \chi_{ \{ 1 < w < 2 \} } \tilde{\varphi}_j \|_{2^*}^{2^*} \text{ and } b_j = \| \chi_{ \{ w \geq 2 \} } \tilde{\varphi}_j \|_{2^*}^{2^*} \,,
\]
and passing to a subsequence, if necessary, we have that
\[
a_j \to a \in [0, 1/2] \text{ and } b_j \to b \in [0, 1/2] \,,
\]
where, due to (\ref{HCdn}), $a_j + b_j = a + b = 1/2$.  This brings us to the main proposition in this step:
\begin{prop}\label{ANZProp}
$a \neq 0$.
\end{prop}

\begin{proof}
Assume $a = 0$.  Then, $b = 1/2$.  Let, $\chi_1,
\chi_2 \in C^\infty ([0,\infty) \times [0,\pi/2] \times \mathbb{S}^{n-1})$ be
such that
\begin{align}
& 0 \leq \chi_1, \chi_2 \leq 1 \nonumber \\
& \chi_1 \text{ and } \chi_2 \text{ are independent of } \theta \text{ and }
\zeta \nonumber \\
& \chi_1 = 1 \text{ for } w \leq 1 \text{ and } \chi_1 = 0 \text{ for } w \geq
2 \label{Chi1} \\
& \chi_2 = 0 \text{ for } w \leq 1 \text{ and } \chi_2 = 1 \text{ for } w \geq
2 \text{, and} \label{Chi2} \\
& \chi_1^2 + \chi_2^2 = 1 \,. \label{SumEq1}
\end{align}
Then, for $\varphi \in \dot{H}^1$ and $i = 1,2$,
\begin{align}
\| \nabla_{w, \theta, \zeta} (\chi_i \varphi) \|_2^2 =& \int [(\chi_i)_w \varphi + \chi_i \varphi_w]^2 + w^{-2} [\chi_i \varphi_\theta]^2 + w^{-2} \csc^2 \theta | \chi_i \nabla_{\mathbb{S}^{n-1}} \varphi |^2 \mathrm{d}\Lambda \nonumber \\
=& \int \varphi^2 (\chi_i)_w^2 \mathrm{d}\Lambda + 2 \int \chi_i \varphi (\chi_i)_w \varphi_w \mathrm{d}\Lambda + \int \chi_i^2 (\varphi_w^2 + w^{-2} \varphi_\theta^2 + w^{-2} \csc^2 \theta | \nabla_{\mathbb{S}^{n-1}} \varphi |^2) \mathrm{d}\Lambda \nonumber \\
\leq& C_1 \| \chi_{ \{ 1 \leq w \leq 2 \} } \varphi^2 \|_{2^*/2} \| \chi_{ \{ 1 \leq w \leq 2 \} } \|_{(2^*/2)'} + C_2 \| \chi_{ \{ 1 \leq w \leq 2 \} } \varphi^2 \|_{2^*/2}^{1/2} \| \chi_{ \{ 1 \leq w \leq 2 \} } \|_{(2^*/2)'}^{1/2} \| \nabla_{w, \theta, \zeta} \varphi \|_2 \nonumber \\
& + \| \chi_i^2 \nabla_{w, \theta, \zeta} \varphi \|_2^2 \text{, for some } C_1, C_2 > 0  \,. \nonumber
\end{align}
The last inequality is obtained by applying Holder's Inequality (twice to obtain
the middle term on the right hand side) and the fact that $(\chi_i)_w \leq \tilde{C} \chi_{ \{ 1 \leq w \leq 2 \} }$ for some finite constant $\tilde{C}$ and for $i = 1,2$. We can rewrite the inequality above as
\begin{equation}\label{CIneq}
\| \nabla_{w, \theta, \zeta} (\chi_i \varphi) \|_2^2 \leq \tilde{C}_1 \| \chi_{ \{ 1 \leq w \leq 2 \} } \varphi \|_{2^*}^2 + \tilde{C}_2 \| \chi_{ \{ 1 \leq w \leq 2 \} } \varphi \|_{2^*} + \| \chi_i \nabla_{w, \theta, \zeta} \varphi \|_2^2 \,,
\end{equation}
for appropriate constants $\tilde{C}_1$ and $\tilde{C}_2$.  Let $\varphi_{i,j} :=
\chi_i \tilde{\varphi}_j$.  Then
\begin{align}\label{IneqApp}
\| \tilde{\varphi}_j \|_{2^*}^2 =& \| \varphi_{1,j}^2 + \varphi_{2,j}^2 \|_{2^*/2} \text{, which by Proposition \ref{HProp}} \nonumber \\
\leq& \begin{cases}
\| \varphi_{1,j}^2 \|_{2^*/2} + \| \varphi_{2,j}^2 \|_{2^*/2} - \frac{(2/2^* - 1) \| \varphi_{2,j}^2 \|_{2^*/2}}{4} \left\| \frac{\varphi_{1,j}^2}{\| \varphi_{1,j}^2 \|_{2^*/2}} - \frac{\varphi_{2,j}^2}{\| \varphi_{2,j}^2 \|_{2^*/2}} \right\|_{2^*/2} &\text{ if } 1 < \frac{2^*}{2} \leq 2 \nonumber \\
\| \varphi_{1,j}^2 \|_{2^*/2} + \| \varphi_{2,j}^2 \|_{2^*/2} - \frac{\| \varphi_{2,j}^2 \|_{2^*/2}}{(2^*/2) 2^{2/2^*-1}} \left\| \frac{\varphi_{1,j}^2}{\| \varphi_{1,j}^2 \|_{2^*/2}} - \frac{\varphi_{2,j}^2}{\| \varphi_{2,j}^2 \|_{2^*/2}} \right\|_{2^*/2} &\text{ if } 2 \leq \frac{2^*}{2} < \infty \nonumber
\end{cases} \\
\leq& \begin{cases}
\| \varphi_{1,j} \|_{2^*}^2 + \| \varphi_{2,j} \|_{2^*}^2 - \frac{2/2^* - 1}{4} \left( \frac{1}{4} \right)^{4/2^*} & \text{ if } 1 < \frac{2^*}{2} \leq 2 \\
\| \varphi_{1,j} \|_{2^*}^2 + \| \varphi_{2,j} \|_{2^*}^2 - \frac{2}{2^* 2^{2^*/2-1}} \left(\frac{1}{4}\right)^{4/2^*} & \text{ if } 2 \leq \frac{2^*}{2} < \infty \,.
\end{cases}
\end{align}
We deduce the last inequality, because
\[
\| \varphi_{2,j}^2 \|_{2^*/2} = \| \chi_2 \tilde{\varphi}_j \|_{2^*}^2 \geq \left(\frac{1}{4} \right)^{2/2^*} \text{, for large $j$, by (\ref{Chi2}) and because $b_j \to b = 1/2$} \,, \nonumber
\]
and
\begin{eqnarray}
&& \left\| \frac{\varphi_{1,j}^2}{\| \varphi_{1,j}^2 \|_{2^*/2}} -
\frac{\varphi_{2,j}^2}{\| \varphi_{2,j}^2 \|_{2^*/2} } \right\|_{2^*/2} \nonumber \\
&\geq& \left\| \frac{\chi_{ \{ w \geq 2 \} } (\tilde{\varphi}_j)^2 }{\| (\chi_2 \tilde{\varphi}_j)^2 \|_{2^*/2} } \right\|_{2^*/2} \text{, by (\ref{Chi1}), (\ref{Chi2}), and the definition of $\varphi_{i,j}$, $i= 1, 2$} \nonumber \\
&\geq& \| \chi_{ \{ w \geq 2 \} } (\tilde{\varphi}_j)^2 \|_{2^*/2} \,, \text{ because } \|
(\chi_2 \tilde{\varphi}_j)^2 \|_{2^*/2} \leq \| \tilde{\varphi}_j \|_{2^*}^2 = 1
\nonumber \\
&\geq& \left( \frac{1}{4} \right)^{2/2^*} \text{, for large $j$, because $b = 1/2$} \,. \nonumber
\end{eqnarray}
Since
\[
\| \varphi_{1,j} \|_{2^*}^2, \| \varphi_{2,j} \|_{2^*}^2 \leq \| \tilde{\varphi}_j \|_{2^*}^2 = 1 \,,
\]
we can conclude by (\ref{IneqApp}) that there is some constant $d_{2^*} < 1$ dependent
only on the value of $2^*$ such that
\begin{align}\label{AnZ}
\| \tilde{\varphi}_j \|_{2^*}^2 \leq& d_{2^*} ( \| \varphi_{1,j} \|_{2^*}^2 + \| \varphi_{2,j} \|_{2^*}^2 ) \,, \text{ which by the Sobolev Inequality} \nonumber \\
\leq& d_{2^*} C_{m,n}^2 ( \| \nabla_{w, \theta, \zeta} \varphi_{1,j} \|_2^2 + \| \nabla_{w, \theta, \zeta} \varphi_{2,j} \|_2^2 ) \,, \text{ which by (\ref{CIneq})} \nonumber \\
\leq& d_{2^*} C_{m,n}^2 \sum_{i=1}^2 ( \tilde{C}_1 \| \chi_{ \{ 1 \leq w \leq 2 \} } \tilde{\varphi}_j \|_{2^*}^2 + \tilde{C}_2 \| \chi_{ \{ 1 \leq w \leq 2 \} } \tilde{\varphi}_j \|_{2^*} + \| \chi_i \nabla_{w, \theta, \zeta} \tilde{\varphi}_j \|_2^2 ) \nonumber \\
\leq& d_{2^*} C_{m,n}^2 (\varepsilon + \| \chi_1 \nabla_{w, \theta, \zeta} \tilde{\varphi}_j \|_2^2 + \| \chi_2 \nabla_{w, \theta, \zeta} \tilde{\varphi}_j \|_2^2 ) \nonumber \\
&\text{for any $\varepsilon > 0$ and sufficiently large $j$, because we assumed that $a = 0$} \nonumber \\
=& d_{2^*} C_{m,n}^2 (\varepsilon + \| \tilde{\varphi}_j \|_{\dot{H}^1}^2 ) \,, \text{ by (\ref{SumEq1})} \,.
\end{align}
Since $d_{2^*} < 1$; $\| \tilde{\varphi}_j \|_{2^*} = 1$, for all $j$; $\frac{\| \tilde{\varphi}_j \|_{2^*}^2}{\| \tilde{\varphi}_j \|_{\dot{H}^1}^2} \to C_{m,n}^2$; and we can pick $\varepsilon$ to be arbitrarily small, (\ref{AnZ}) contradicts the Sobolev Inequality.  Recall that we arrived at (\ref{AnZ}) by assuming that $a = 0$. Hence, we conclude that $a \neq 0$.
\end{proof}

\textit{\underline{Step 3:} Show the $\tilde{\varphi}_j$ times one of two possible cutoff functions yields a sequence that satisfies the $p,q,r$-Theorem.} Let, $\chi_3 \in C^\infty ([0, \infty) \times [0, \pi/2] \times \mathbb{S}^{n-1})$ be such that
\begin{align}
& 0 \leq \chi_3 \leq 1 \label{Chi3Siz} \\
& \chi_3 \text{ is independent of } \zeta \nonumber \\
& \chi_3 = 1 \text{ on } \{ 1 \leq w \leq 2 \} \cap \{ 0 \leq \theta \leq \pi/4 \} \nonumber \\
& \chi_3 = 0 \text{ on } \{ w \leq 1/2 \} \cup \{ w \geq 4 \} \cup \{ \theta \geq \pi/3 \} \,, \label{Chi3Supp}
\end{align}
and $\chi_3^R$ be the cutoff function given by
\[
\chi_3^R (w, \theta, \zeta) = \chi_3 (w, \frac{\pi}{2} - \theta, \zeta) \,.
\]
Next, let
\begin{equation}\label{P3JDef}
\varphi_{3,j} := \chi_3 \tilde{\varphi}_j \text{ and } \varphi_{3,j}^R := \chi_3^R \tilde{\varphi}_j \,.
\end{equation}
This brings us to the main proposition in this step:
\begin{prop}\label{PQRProp}
A subsequence of either $(\varphi_{3,j})$ or $(\varphi_{3,j}^R)$ satisfies the
$p,q,r$-Theorem with
\[
p = 1, q = 2^*, r = \frac{2(m+n-\delta)}{m+n-\delta-2} \,,
\]
for some $\delta > 0$.
\end{prop}
\begin{proof}
The definitions of $\varphi_{3,j}$ and $\varphi_{3,j}^R$ imply that
\[
\| \varphi_{3,j} \|_{2^*}^{2^*} + \| \varphi_{3,j}^R \|_{2^*}^{2^*} \geq \| \chi_{ \{ 1 < w < 2 \} } \tilde{\varphi}_j \|_{2^*}^{2^*} = a_j \,.
\]
Thus, passing to a subsequence, if necessary, either $( \| \varphi_{3,j} \|_{2^*} )$ or $( \| \varphi_{3,j}^R \|_{2^*})$ is bounded below by $(a/3)^{1/2^*}$. Whichever sequence is bounded below by $(a/3)^{1/2^*}$ will satisfy the $p,q,r$-Theorem.  The sequence that is bounded below will satify the ``$q$'' part of the $p,q,r$-Theorem as posed in the proposition (i.e. bounded below in the $L^{2^*}$ norm).  It will also satisfy the ``$p$'' part as posed in the proposition (i.e. bounded above in the $L^1$ norm).  We show this below for $( \varphi_{3,j} )$ - the proof for $( \varphi_{3,j}^R )$ being identical.
\begin{align}\label{MeasProp}
\| \varphi_{3, j} \|_1 =& \| \chi_3 \tilde{\varphi}_j \|_1 \text{, which by the definition of $\chi_3$ and Holder's Inequality} \nonumber \\
\leq& \| \tilde{\varphi}_j \|_{2^*} \Lambda(\{ 1/2 \leq w \leq 4 \})^{1/(2^*)'} \nonumber \\
=& \Lambda(\{ 1/2 \leq w \leq 4 \})^{1/(2^*)'} \,.
\end{align}
At this point, we only need to show the ``$r$'' part of the $p,q,r$-Theorem is satisfied.  We deal with the case when $( \| \varphi_{3,j} \|_{2^*} )$ is bounded below by $(a/3)^{1/2^*}$ and the case when $( \| \varphi_{3,j}^R \|_{2^*})$ is bounded below by $(a/3)^{1/2^*}$ separately.

\textit{$( \| \varphi_{3,j} \|_{2^*})$ is bounded below by $(a/3)^{1/2^*}$}: We begin with the following identity and definitions:
\begin{eqnarray}
\mathrm{d}\Lambda &=& \omega_m w^{m+n-1} \cos^{m-1} \theta \sin^{n-1} \theta \mathrm{d}w \mathrm{d}\theta \mathrm{d} \Omega (\zeta) \nonumber \\
\mathrm{d} \Lambda_{m/2,n} &:=& \omega_{m/2} w^{\frac{m}{2} + n -1} \cos^{\frac{m}{2}-1} \theta \sin^{n-1} \theta \mathrm{d}w \mathrm{d}\theta \mathrm{d}\Omega(\zeta) \nonumber \\
\| \varphi \|_{\dot{H}^1; (m/2, n)} &:=& \left( \int \varphi_w^2 + w^{-2} \varphi_\theta^2 + w^{-2} \csc^2 \theta | \nabla_{\mathbb{S}^{n-1}} \varphi |^2 \mathrm{d} \Lambda_{m/2, n} \right)^{1/2} \nonumber \\
\hat{2}^* &:=& \frac{2 (m/2 + n)}{m/2 + n - 2} \nonumber \\
\| \varphi \|_{\hat{2}^*; (m/2, n)} &:=& \left( \int | \varphi |^{\hat{2}^*} \mathrm{d} \Lambda_{m/2, n} \right)^{1/\hat{2}^*} \,. \nonumber
\end{eqnarray}
Then,
\begin{align}
\| \varphi_{3,j} \|_{\hat{2}^*} =& \left( \int | \varphi_{3,j} |^{\hat{2}^*} \mathrm{d}\Lambda \right)^{1/\hat{2}^*} \text{, which by (\ref{Chi3Siz}), (\ref{Chi3Supp}), and (\ref{P3JDef})} \nonumber \\
\leq& \left( \frac{4^{m/2} \omega_m}{\omega_{m/2}} \right)^{1/\hat{2}^*} \left( \int |\varphi_{3,j}|^{\hat{2}^*} \mathrm{d}\Lambda_{m/2,n} \right)^{1/\hat{2}^*} \text{, which by the Sobolev Inequality} \nonumber \\
\leq& C_{m/2, n} \left( \frac{4^{m/2} \omega_m}{\omega_{m/2}} \right)^{1/\hat{2}^*}  \| \varphi_{3,j} \|_{\dot{H}^1; (m/2, n)} \nonumber \\
=& C_{m/2, n} \left( \frac{4^{m/2} \omega_m}{\omega_{m/2}} \right)^{1/2^*} \left( \int [ \varphi_{3,j} ]_w^2 + w^{-2} [ \varphi_{3,j} ]_\theta^2 + w^{-2} \csc^2 \theta | \nabla_{\mathbb{S}^{n-1}} \varphi_{3,j} |^2 \mathrm{d}\Lambda_{m/2, n} \right)^{1/2} \nonumber \\
& \text{which by (\ref{Chi3Siz}), (\ref{Chi3Supp}), and (\ref{P3JDef})} \nonumber \\
\leq& C_{m/2, n} 4^{(1/\hat{2}^* + 1/2) \frac{m}{2}} \left( \frac{\omega_{m/2}}{\omega_m} \right)^{1/2 - 1/\hat{2}^*} \| \varphi_{3,j} \|_{\dot{H}^1} \nonumber \\
<& M_1, \forall j \,, \nonumber
\end{align}
for some finite $M_1$, because $( \| \varphi_{3,j} \|_{\dot{H}^1} )$ must be bounded
due to (\ref{NRCdn1}).

\textit{$( \| \varphi_{3,j}^R \|_{2^*} )$ is bounded below by $(a/3)^{1/2^*}$}:
Choose some $\delta > 0$ such that
\[
m+n-\delta>2 \text{ and } \delta < n \,.
\]
Let
\[
\tilde{2}^* := \frac{2(m+n-\delta)}{m+n-\delta-2} \,.
\]
In this case, it is crucial to note that since we assumed that $\chi_3$ is independent of $\zeta$,
\[
\varphi_{3,j}^R = \chi_3^R \tilde{\varphi}_j^* \implies \varphi_{3,j}^R \text{ is independent of $\zeta$} \,.
\]
Thus,
\begin{align}
\| \varphi_{3,j}^R \|_{\tilde{2}^*} =& \left( \int_{\mathbb{S}^{n-1}} \int_0^{\pi/2} \int_0^\infty | \varphi_{3,j}^R |^{\tilde{2}^*} \omega_m w^{m+n-1} \mathrm{d}w \cos^{m-1} \theta \sin^{n-1} \theta \mathrm{d} \theta \mathrm{d} \Omega (\zeta) \right)^{1/2^*} \nonumber \\
=& \left( \int_0^{\pi/2} \int_0^\infty | \varphi_{3,j}^R |^{\tilde{2}^*} \omega_m \omega_n w^{m+n-1} \mathrm{d}w \cos^{m-1} \theta \sin^{n-1} \theta \mathrm{d} \theta \right)^{1/\tilde{2}^*} \text{, which by (\ref{Chi3Siz}), (\ref{Chi3Supp}), and (\ref{P3JDef})} \nonumber \\
\leq&  \left( \frac{4^\delta \omega_n}{\omega_{n-\delta}} \right)^{1/\tilde{2}^*} C_{m,n-\delta} \left( \int_0^{\pi/2} \int_0^\infty | \varphi_{3,j}^R |^{\tilde{2}^*} \omega_m \omega_{n-\delta} w^{m+n-\delta-1} \mathrm{d}w \cos^{m-1} \theta \sin^{n-\delta-1} \theta \mathrm{d} \theta \right)^{1/2} \nonumber \\
& \text{which by the Sobolev Inequality} \nonumber \\
\leq& \left( \frac{4^\delta \omega_n}{\omega_{n-\delta}} \right)^{1/\tilde{2}^*} C_{m,n-\delta} \left( \int_0^{\pi/2} \int_0^\infty ( [\varphi_{3,j}]_w^2 + w^{-2} [ \varphi_{3,j} ]_\theta^2 ) w^{m+n-\delta-1} \mathrm{d}w \cos^{m-1} \theta \sin^{n-\delta-1} \theta \mathrm{d} \theta \right)^{1/2}
\nonumber \\
& \text{which by (\ref{Chi3Siz}), (\ref{Chi3Supp}), and (\ref{P3JDef})} \nonumber \\
\leq& \left( \frac{4^\delta \omega_n}{\omega_{n-\delta}} \right)^{1/2^*} \left( \frac{4^\delta \omega_{n-\delta}}{\omega_n} \right)^{1/2} \| \varphi_{3,j}^R \|_{\dot{H}^1} \nonumber \\
<& M_2, \forall j \,, \nonumber
\end{align}
for some finite $M_2$, because $( \| \varphi_{3,j}^R \|_{\dot{H}^1} )$ must be bounded due to (\ref{NRCdn1}).
\end{proof}

We can now apply the result of Proposition \ref{PQRProp} to prove Lemma \ref{MLm}.
\begin{proof}[Proof of Lemma \ref{MLm}]
Since $\operatorname{supp} (\chi_3), \operatorname{\supp} (\chi_3^R) \subseteq \{ w \leq 4 \}$ and $0 \leq \chi_3 \leq 1$, Proposition \ref{PQRProp} shows that some sequence, $\left( \Lambda ( \{ | \tilde{\varphi}_{j_k} (w, \theta, \zeta) | > \varepsilon, w \leq 4 \} ) \right)$, is bounded below by a positive constant, $C$. Since $\{ w \leq 4 \} \subseteq \{ \rho \leq 4 \}$, Lemma \ref{MLm1} is true a fortiori. Thus, by the reduction in step one, Lemma \ref{MLm} holds.
\end{proof}
\noindent At this point, we relabel
indices and apply Theorem \ref{LiebThm} - we have not proved Theorem \ref{LiebThm} yet, but will prove it at the end of this section - to conclude that there exists $(x_j) \subseteq \mathbb{R}^n$ such that
\[
\hat{\varphi}_j (\rho, x) = \varphi_j^{\sigma_j} ( \rho, x + x_j )
\]
has a subsequence that converges to some nonzero $\varphi$ in $\dot{H}_\mathbb{C}^1$.  We will show that this convergence is in fact strong convergence in $\dot{H}_\mathbb{C}^1$ and that $\varphi$ is an extremal of the Sobolev Inequality with $L^{2^*}$ norm equal one.

\underline{\textbf{Part 2 of proof of Theorem \ref{CCThm}} - Conclusion of proof using Functional Analysis arguments:} The title of this part is self-explanatory. A notable feature of this part is an application of the local compactness theorem, Theorem \ref{LocalCompactnessThm}. This theorem is a Rellich-Kondrachov type Theorem for cylindrically symmetric functions in continuous dimension and anything of its type is, to our knowledge, absent from literature. We prove Theorem \ref{LocalCompactnessThm} in the next and final section.

Applying Theorem \ref{LocalCompactnessThm}, passing to a subsequence if necessary, we may assume that $\hat{\varphi}_j$ converges to $\varphi$ almost everywhere. We now argue that $\hat{\varphi}_j$ converges strongly in $\dot{H}_\mathbb{C}^1$ and $\varphi$ is an extremal of the Sobolev Inequality.  Weak convergence in $\dot{H}_\mathbb{C}^1$ implies that
\begin{equation}\label{Cvg1}
\| \hat{\varphi}_j \|_{\dot{H}^1}^2 = \| \varphi \|_{\dot{H}^1}^2 + \| \hat{\varphi}_j - \varphi \|_{\dot{H}^1}^2 + o (1) \,.
\end{equation}
Next, we observe that almost everywhere convergence and the Brezis-Lieb Lemma imply that
\begin{equation}\label{Cvg2}
\| \hat{\varphi}_j \|_{2^*}^{2^*} = \| \varphi \|_{2^*}^{2^*} + \| \hat{\varphi}_j - \varphi \|_{2^*}^{2^*} + o(1) \,.
\end{equation}
Combining this with the concavity of $y \mapsto y^{2/2^*}$, and passing to a subsequence if necessary, we deduce that
\[
\lim \| \hat{\varphi}_j \|_{2^*}^{2} \leq \| \varphi \|_{2^*}^{2} + \lim \| \hat{\varphi}_j - \varphi \|_{2^*}^2 \,.
\]
Thus,
\begin{align}\label{Eq}
1 =& \lim \| \hat{\varphi}_j \|_{2^*}^2 \text{, by assumption} \nonumber \\
\leq& \| \varphi \|_{2^*}^2 + \lim \| \hat{\varphi}_j - \varphi \|_{2^*}^2
\text{, which by the Sobolev Inequality} \nonumber \\
\leq& C_{m,n}^2 ( \| \varphi \|_{\dot{H}^1}^2 + \lim \| \hat{\varphi}_j - \varphi \|_{\dot{H}^1}^2 ) \text{, which by (\ref{ConvCdn}) and (\ref{Cvg1})} \nonumber \\
=& 1 \,.
\end{align}
(\ref{Eq}) implies that
\begin{equation}\label{Cvg3}
C_{m,n}^2 \| \varphi \|_{\dot{H}^1}^2 - \| \varphi \|_{2^*}^2 + \lim (C_{m,n}^2 \| \hat{\varphi}_j - \varphi \|_{\dot{H}^1}^2 - \| \hat{\varphi}_j - \varphi \|_{2^*}^2 ) = 0 \,.
\end{equation}
The Sobolev Inequality implies that
\[
C_{m,n}^2 \| \hat{\varphi}_j - \varphi \|_{\dot{H}^1}^2 - \| \hat{\varphi}_j - \varphi \|_{2^*}^2 \text{, } C_{m,n}^2 \| \varphi \|_{\dot{H}^1}^2 - \| \varphi \|_{2^*}^2 \geq 0 \,.
\]
Combining this with (\ref{Cvg3}), we conclude that
\[
C_{m,n}^2 \| \varphi \|_{\dot{H}^1}^2 - \| \varphi \|_{2^*}^2 = 0 \,,
\]
i.e. $\varphi$ is an extremal.

(\ref{Eq}) also implies that
\begin{equation}\label{Cvg4}
1 = \| \varphi \|_{2^*}^2 + \lim \| \hat{\varphi}_j - \varphi \|_{2^*}^2 \,.
\end{equation}
Since $y \mapsto y^{2/2^*}$ is strictly concave and $\varphi$ is nonzero, (\ref{Cvg2}); $\| \varphi_j \|_{2^*} = 1$, $\forall j$; and, (\ref{Cvg4}) allow us to conclude that $\| \varphi \|_{2^*}^2 = 1$.  Thus, $\hat{\varphi}_j$ converges to $\varphi$ in norm and weakly.  These two characteristics imply that $\hat{\varphi}_j$ converges to $\varphi$ strongly in $\dot{H}_\mathbb{C}^1$. This would conclude the proof of Theorem \ref{CCThm}, except we have not yet proved Theorem \ref{LiebThm}.  We conclude this section by proving Theorem \ref{LiebThm}.

\begin{proof} [\underline{\textbf{Part 3 of proof of Theorem \ref{CCThm}} - Proof of Theorem \ref{LiebThm}:}] Let $B_y$ denote the ball of unit radius in $\mathbb{R}^n$ centered at $y \in \mathbb{R}^n$. By Theorem \ref{LocalCompactnessThm} and the Banach-Alaoglu Theorem, it suffices to prove that we can find $x_j$ and $\delta > 0$ such that $\Lambda \left( \{ \rho \leq R + 1, x \in B_{x_j} \} \cap \{ |\varphi_j (\rho,x) \geq \varepsilon/2 \} \right) \geq \delta$, for all $j$, for then $\int_{ \{ \rho \leq R + 1, x \in B_0 \} } | \varphi_j^T | \mathrm{d} \Lambda \geq \delta \varepsilon/2$, and so no weak limit can vanish.  Without loss of generality, we may assume $\varphi_j \geq 0$, for all $j$.  Thus, we henceforth assume that $E_j := \{ \varphi_j (\rho,x) > \varepsilon , \rho \leq R\}$ - refer back to Theorem \ref{LiebThm} for the original definition of $E_j$.

Let $\Psi_j = \chi_4 ( \varphi_j - \varepsilon/2 )_+$, where $\chi_4 \in C^\infty (
[0, \infty) \times \mathbb{R}^n )$ is such that
\begin{eqnarray}
&& 0 \leq \chi_4 \leq 1 \nonumber \\
&& \chi_4 = 1 \text{ for } \rho \leq R \nonumber \\
&& \chi_4 = 0 \text{ for } \rho \geq R + 1 \,. \nonumber
\end{eqnarray}
Note that $\Lambda (\operatorname{supp}[(\varphi_j - \varepsilon/2)_+]) \leq C < \infty$ for all $j$ and some $C$, because $( \| \varphi_j \|_{2^*} )$ is uniformly bounded. More precisely, if $C_* < \infty$ bounds $( \| \varphi_j \|_{2^*} )$ above, then
\begin{equation}\label{C*Bd}
C_* \geq \left( \frac{\varepsilon}{2} \right)^{2^*} \Lambda ( \{ \varphi_j \geq \varepsilon/2 \} ) \implies \Lambda ( \operatorname{supp} [ ( \varphi_j - \varepsilon/2)_+ ] ) \leq C_* \left( \frac{2}{\varepsilon} \right)^{2^*} =: J \,.
\end{equation}
Thus, $\Psi_j \in L^2$, for all $j$.  Also, $\Psi_j \geq \varepsilon/2$ on $E_j$.  Thus,
\begin{eqnarray}
\frac{\int | \nabla_{\rho, x} \Psi_j |^2 \mathrm{d} \Lambda}{\int | \Psi_j |^2 \mathrm{d} \Lambda} &\leq& \frac{2 \int | \nabla_{\rho, x} \chi_4 |^2 | ( \varphi_j - \varepsilon/2 )_+ |^2 + | \chi_4 |^2 | \nabla_{\rho, x} ( \varphi - \varepsilon / 2)_+ |^2 \mathrm{d}\Lambda}{\delta (\varepsilon/2)^2} \nonumber \\
&& \text{which by Holder's Inequality and (\ref{C*Bd})} \nonumber \\
&\leq& \frac{2 ( \| \nabla_{\rho,x} \chi_4 \|_\infty J^{1 / (2^*/2)'} \| (\varphi -
\varepsilon/2)_+ \|_{2^*}^2 + \| \nabla_{\rho,x} ( \varphi_2 - \varepsilon/2 ) \|_2
) }{\delta (\varepsilon/2)^2} \nonumber \\
&=:& W \,. \nonumber
\end{eqnarray}

Let $G$ be a nonzero $C_C^\infty (\mathbb{R}^n)$ function supported on $B_0$ and let $G^y (x) = G (x - y)$.  Define $\lambda :=  \int_{\mathbb{R}^n} | \nabla_x G |^2 \mathrm{d}x / \int_{\mathbb{R}^n} | G |^2 \mathrm{d}x$, where $\nabla_x$ denotes the gradient over the $x$ variable.

Let $\Theta_j^y (\rho,x) = G^y (x) \Psi_j (\rho, x)$.  Then,
\[
| \nabla_{\rho,x} \Theta_j^y |^2 \leq 2 ( | \nabla_x G^y |^2 | \Psi_j |^2 + | G^y |^2 | \nabla_{\rho,x} \Psi_j |^2 ) \,.
\]
Consider
\begin{align}\label{A}
T_j^y :=& \int | \nabla_{\rho,x} \Theta_j^y |^2 - 4 (W + \lambda) | \Theta_j^y |^2 \mathrm{d} \Lambda \nonumber \\
\leq& 2 \int | \nabla_x G^y |^2 | \Psi_j |^2 + | G^y |^2 | \nabla_{\rho,x} \Psi_j |^2 - 2(W + \lambda)|G_j^y|^2 |\Psi_j|^2 \mathrm{d} \Lambda \,.
\end{align}
Thus,
\begin{align}\label{B}
\frac{1}{2} \int_{\mathbb{R}^n} T_j^y \mathrm{d}y \leq& \int | \nabla_x G |^2 \mathrm{d}x \int | \Psi_j |^2 \mathrm{d} \Lambda + \int |G|^2 \mathrm{d}x \int | \nabla_{\rho,x} \Psi_j |^2 \mathrm{d} \Lambda - 2 (W + \lambda) \int |G|^2 \mathrm{d}x \int | \Psi_j |^2 \mathrm{d} \Lambda \nonumber \\
<& 0 \,.
\end{align}
Combining (\ref{A}) and (\ref{B}), we conclude that for each $j$ there is some $x_j$ such that $\| \Theta_j^{x_j} \|_2 > 0$ and
\begin{equation}\label{Bd1}
\| \nabla_{\rho, x} \Theta_j^{x_j} \|_2 / \| \Theta_j^{x_j} \|_2 < 2(W + \lambda) \,.
\end{equation}
We will use this fact to prove that $( \Lambda ( \operatorname{supp} (\Theta_j^{x_j}) ) )$ is uniformly bounded below by a positive constant.

Let $\varphi \in \dot{H}_\mathbb{C}^1$ be such that $\Lambda ( \operatorname{supp} (\varphi) ) < \infty$.  Then
\begin{align}\label{Bd2}
\| \varphi \|_2^2 \leq& \Lambda ( \operatorname{\supp} (\varphi) )^{1/(2^*/2)'} \| \varphi \|_{2^*}^2 \nonumber \\
\leq& C_{m,n} \Lambda ( \operatorname{supp} (\varphi) )^{1/(2^*/2)'} \| \nabla_{\rho,x} \varphi \|_2^2 \nonumber \\
\implies& \| \nabla_{\rho,x} \varphi \|_2^2 / \| \varphi \|_2^2 \geq C_{m,n}^{-1} \Lambda( \operatorname{supp} ( \varphi ))^{-1/(2^*/2)'} \,.
\end{align}
(\ref{Bd1}) and (\ref{Bd2}) imply that $\Lambda ( \operatorname{supp} ( \Theta_j^{x_j} ) ) \geq \delta > 0$, for all $j$, for some $\delta$.  Combining this with the fact that $\operatorname{supp} (\Theta_j^{x_j}) \subseteq \{ \rho \leq R+1, x \in B_{x_j} \}$, we conclude that $\Lambda ( \{ \rho \leq R + 1, x \in B_{x_j} \} \cap \{ \varphi_j (\rho,x) \geq \varepsilon/2 \} ) \geq \delta$, for all $j$.
\end{proof}

\section{Proof of Rellich-Kondrachov Type Theorem}

In the following, we prove Theorem \ref{LocalCompactnessThm}.  We restate this theorem below:

\noindent \textbf{Theorem \ref{LocalCompactnessThm}.} \textit{Let $K \subseteq [0,\infty) \times \mathbb{R}^n$ satisfy the cone property in $\mathbb{R}^{n+1}$, $K \subseteq \{ (\rho, x) \in [0, \infty) \times \mathbb{R}^n | \rho_1 < \rho < \rho_2 \}$ for some $0 < \rho_1 < \rho_2 < \infty$, and $\Lambda (K) < \infty$, where $\Lambda$ denotes the measure on $\mathbb{R}_+ \times \mathbb{R}^n$ defined by (\ref{LDef}).  If $(\varphi_j)$ is bounded in $\dot{H}_\mathbb{C}^1$ and $U$ is an open subset of $K$, then for  $1 \leq p < \max \left\{ 2^*, \frac{2n+2}{n-1} \right\}$, there is some $\varphi \in \dot{H}_\mathbb{C}^1$ and some subsequence, $( \varphi_{j_k})$, such that $\varphi_{j_k} \to \varphi$ in $L_\mathbb{C}^p ( U, \omega_m \rho^{m-1} \mathrm{d}\rho \mathrm{d}x)$.}

\begin{proof}
First, we note that, taking a subsequence if necessary, $\varphi_j \rightharpoonup
\varphi$ in $\dot{H}_\mathbb{C}^1$ for some $\varphi$.  Next, we show that $( \varphi_j )$ is bounded in $H^1 (K)$.  To this end, we show that $L_\mathbb{C}^q (V, \mathrm{d}\rho \mathrm{d}x)$ and $L_\mathbb{C}^q (V, \omega_m \rho^{m-1} \mathrm{d}\rho \mathrm{d}x)$ are equivalent norms for $1 \leq q < \infty$, when
\begin{equation}\label{VCdn}
V \subseteq \{ (\rho, x) \in [0, \infty) \times \mathbb{R}^n \big| \rho_1 < \rho < \rho_2 \} \,,
\end{equation}
because (\ref{VCdn}) implies that
\[
\omega_m^{-1/q} \rho_2^{-(m-1)/q} \| \cdot \|_{L^q (V, \omega_m \rho^{m-1} \mathrm{d}\rho \mathrm{d}x)} \leq \| \cdot \|_{L^q ( V, \mathrm{d}\rho \mathrm{d}x)} \leq \omega_m^{-1/q} \rho_1^{-(m-1)/q} \rho_1^{-(m-1)/q} \| \cdot \|_{L^q (V, \omega_m \rho^{m-1} \mathrm{d}\rho \mathrm{d}x)} \,. \nonumber
\]
Thus,
\begin{equation}\label{NRel}
\| \nabla \varphi_j \|_{L^2 (K, \mathrm{d}\rho \mathrm{d}x)} \leq \omega_m^{-1/2} \rho_2^{-(m-1)/2} \| \nabla \varphi_j \|_{L^2 (K, \omega_m \rho^{m-1} \mathrm{d} \rho \mathrm{d}x)} \,,
\end{equation}
and
\begin{align}\label{L2Bd}
\| \varphi_j \|_{L^2 (K, \mathrm{d}\rho \mathrm{d}x)} \leq& \omega_m^{-1/2} \rho_2^{-(m-1)/2} \| \varphi_j \|_{L^2 (K, \omega_m \rho^{m-1} \mathrm{d}\rho \mathrm{d}x)} \nonumber \\
\leq& \omega_m^{-1/2} \rho_2^{-(m-1)/2} \Lambda(K)^{1/(m+n)} \| \varphi_j \|_{L^{2^*} (K, \omega_m \rho^{m-1} \mathrm{d}\rho \mathrm{d}x)}^2 \,, \text{ which by Theorem \ref{SobExtThm}} \nonumber \\
\leq& \omega_m^{-1/2} \rho_2^{-(m-1)/2} \Lambda(K)^{1/(m+n)} C_{m,n}^2 \| \nabla \varphi_j \|_{L^2 (K, \omega_m \rho^{m-1} \mathrm{d}\rho \mathrm{d}x)}^2 \,.
\end{align}
Combining (\ref{NRel}) and (\ref{L2Bd}), we conclude that $(\varphi_j)$ is bounded $H^1 (K)$. Applying the Rellich-Kondrachov Theorem, we conclude that if $1 \leq p < \frac{2n+2}{n-1}$, then there is some $\Psi \in H^1 (K)$ and some subsequence, $(\varphi_{j_k})$, such that for $U \subseteq K$, $U$ open,
\[
\varphi_{j_k} \to \Psi \text{ in } L^p (U, \mathrm{d}\rho \mathrm{d}x) \,.
\]
Since $\| \cdot \|_{L^p (U, \mathrm{d}\rho \mathrm{d}x)}$ and $\| \cdot \|_{L^p (U, \omega_m \rho^{m-1} \mathrm{d} \rho \mathrm{d}x)}$ are equivalent norms, we
conclude that $\varphi_{j_k} \to \Psi$ in $L_\mathbb{C}^p (U, \omega_m \rho^{m-1}
\mathrm{d}\rho \mathrm{d}x)$.  Since $\varphi_j \rightharpoonup \varphi$ in
$\dot{H}_\mathbb{C}^1$, we conclude that $\Psi = \varphi$.

We conclude by showing that if $\frac{2n+2}{n-1} < 2^*$, then for $\frac{2n+2}{n-1} \leq p < 2^*$, there is some $(\varphi_{j_k})$ such that $\varphi_{j_k} \to \varphi$ in $L_\mathbb{C}^p (U, \omega_m \rho^{m-1} \mathrm{d}\rho \mathrm{d}x)$. By the Holder Inequality,
\begin{equation}\label{0}
\| \varphi_j - \varphi \|_p^p \leq \| \varphi_j - \varphi \|_{(p-q) \frac{r}{r-1}}^\alpha \| \varphi_j - \varphi \|_{qr}^\beta
\end{equation}
for some $\alpha, \beta > 0$, $1 < q < p$, and $1 < r < \infty$. If we choose $q$ and $r$ such that
\[
qr = 2^* \,,
\]
then
\begin{equation}\label{1}
(p-q) \frac{r}{r-1} = (p-q) \frac{2^*/q}{(2^*/q)-1} \,.
\end{equation}
Note that
\begin{equation}\label{2}
q < p < 2^* \implies \frac{2^*/q}{(2^*/q)-1} < \frac{2^*/p}{(2^*/p)-1} < \infty \,.
\end{equation}
Combining (\ref{1}) and (\ref{2}), we conclude that if $q$ is close enough to $p$, then
\begin{equation}\label{3}
(p-q) \frac{r}{r-1} < p \,.
\end{equation}
Choosing a value for $q$ for which (\ref{3}) holds and $(p-q)\frac{q}{q-1} \geq 1$, and a corresponding subsequence, $(\varphi_{j_k})$, such that $\varphi_{j_k} \to \varphi$ in $L_\mathbb{C}^{(p-q) \frac{r}{r-1}} (U, \omega_m \rho^{m-1} \mathrm{d}\rho \mathrm{d}x)$, (\ref{0}) yields
\begin{eqnarray}
\lim_{k \to \infty} \| \varphi_{j_k} - \varphi \|_p^p &\leq& \lim_{k \to \infty} \| \varphi_{j_k} - \varphi \|_{(p-q)\frac{r}{r-1}}^\alpha \| \varphi_{j_k} - \varphi \|_{2^*}^\beta \nonumber \\
&& \text{and since $\varphi_j$ is bounded in $\dot{H}_\mathbb{C}^1$ and $\varphi_{j_k} \to \varphi$ in $L_\mathbb{C}^{(p-q) \frac{r}{r-1}}(U, \omega_m \rho^{m-1} \mathrm{d}\rho \mathrm{d}x)$} \nonumber \\
&=& 0 \nonumber \,.
\end{eqnarray}
\end{proof}

\end{document}